\documentclass[reqno,11pt]{amsart}
\usepackage{graphicx} 
\usepackage[margin=1.1in]{geometry}

\usepackage[utf8]{inputenc}

\usepackage{amsmath, amsfonts, amssymb, tikz, amsthm, mathtools}
\usepackage[pdfencoding=auto,psdextra]{hyperref}

\usepackage{dsfont}
\usepackage{mathabx,epsfig}
\usepackage{bbold}
\usepackage{faktor}
\usepackage{multicol}
\usepackage{wrapfig}
\usepackage{enumitem}

{\theoremstyle{plain}

\newtheorem{theorem}{Theorem}[section]

\newtheorem{proposition}[theorem]{Proposition}
\newtheorem{lemma}[theorem]{Lemma}
\newtheorem{corollary}[theorem]{Corollary}

}
{\theoremstyle{definition}
\newtheorem{definition}[theorem]{Definition}
\newtheorem*{definition*}{Definition}
\newtheorem{remark}[theorem]{Remark}
\newtheorem*{remark*}{Remark}

\newtheorem{example}[theorem]{Example}
}
\newtheorem{innercustomgeneric}{\customgenericname}
\providecommand{\customgenericname}{}
\newcommand{\newcustomtheorem}[2]{%
  \newenvironment{#1}[1]
  {%
   \renewcommand\customgenericname{#2}%
   \renewcommand\theinnercustomgeneric{##1}%
   \innercustomgeneric
  }
  {\endinnercustomgeneric}
}

\newcustomtheorem{customthm}{Theorem}
\newcustomtheorem{customlemma}{Lemma}
\newcustomtheorem{customconj}{Conjecture}
\newcustomtheorem{customcor}{Corollary}
\newcustomtheorem{customremark}{Remark}

\counterwithin{equation}{section}

\newcommand{\ca}{\mathcal{A}}
 \newcommand{\cb}{\mathcal{B}}
\newcommand{\cc}{\mathcal{C}}

\newcommand{\ce}{\mathcal{E}}
\newcommand{\cf}{\mathcal{F}}
\newcommand{\cg}{\mathcal{G}}

 \newcommand{\ci}{\mathcal{I}}
 \newcommand{\cl}{\mathcal{L}}

 \newcommand{\ct}{\mathcal{T}} 
   
 \newcommand{\cv}{\mathcal{V}}

 \newcommand{\cx}{\mathcal{X}}
\newcommand{\cy}{\mathcal{Y}}

\newcommand{\C}{\mathbb{C}}           
\newcommand{\N}{\mathbb{ N}}           

\newcommand{\ds}{\displaystyle}
\newcommand{\abs}[1]{\left| #1 \right|}
\newcommand{\llangle}{\left\langle}
\newcommand{\rrangle}{\right\rangle}
\newcommand{\polyr}[1]{\C[t_1,\ldots,t_{#1}]}
\newcommand{\poly}{\C[t_\bullet]}
\newcommand{\supp}{\mathrm{supp}}
\newcommand{\adj}{\mathrm{N}}
\newcommand{\cut}{\mathfrak{c}}
\newcommand{\Cut}{\mathfrak{C}}

\newcommand{\IB}{\mathrm{IB}}
\newcommand{\LB}{\mathrm{LB}}
\newcommand{\IC}{\mathrm{IC}}

 \newcommand{\splines}[1]{\mathcal{M}_{#1}}
 \newcommand{\lqot}[1]{\mathrm{L}_{#1}}
 \newcommand{\rqot}[1]{\mathrm{R}_{#1}}
 \newcommand{\lrep}[1]{\mathbf{ch}\left(\mathrm{L}_{#1}\right)}
 \newcommand{\rrep}[1]{\mathbf{ch}\left(\mathrm{R}_{#1}\right)}
 \newcommand{\choos}[2]{\genfrac(){0pt}{0}{#1}{#2}}

\title{Splines on Cayley Graphs of the Symmetric Group}
\author{Nathan R.\,T. Lesnevich}
\thanks{The author was partially supported by NSF grant DMS-1954001, and would like to thank Martha Precup and John Shareshian for their continued advice and support.}
\address{Department of Mathematics, Washington University in St.\ Louis, One Brookings Drive, St.\ Louis, MO 63130, USA}
\email{\href{mailto:nlesnevich@wustl.edu}{nlesnevich@wustl.edu}}

\begin{document}
\begin{abstract}
A spline is an assignment of polynomials to the vertices of a graph whose edges are labeled by ideals, where the difference of two polynomials labeling adjacent vertices must belong to the corresponding ideal. The set of splines forms a ring. We consider spline rings where the underlying graph is the Cayley graph of a symmetric group generated by a collection of transpositions. These rings generalize the GKM construction for equivariant cohomology rings of flag, regular semisimple Hessenberg, and permutohedral varieties. These cohomology rings carry two actions of the symmetric group $S_n$ whose graded characters are both of general interest in algebraic combinatorics. In this paper, we generalize the graded $S_n$-representations from the cohomologies of the above varieties to splines on Cayley graphs of $S_n$, then (1) give explicit module and ring generators for whenever the $S_n$-generating set is minimal, (2) give a combinatorial characterization of when graded pieces of one $S_n$-representation is trivial, and (3) compute the first degree piece of both graded characters for all generating sets.
\end{abstract}

\maketitle

\tableofcontents
\addtocontents{toc}{\protect\setcounter{tocdepth}{1}}

\section{Introduction}\label{sec:intro}
Let $\cg$ be a graph with edges labeled by ideals in $\poly \coloneqq \polyr{n}$. A \emph{spline} on $\cg$ is an assignment of polynomials to vertices such that the difference of two polynomials labeling adjacent vertices must be in the corresponding ideal. The \emph{Cayley graph} for a group $G$ and generating set $S \subseteq G$ has vertex set $G$ and edge set $\{(g,gs) \mid g \in G, s \in S\}$. When the group $G$ is a symmetric group $S_n$ and the generating set $S$ consists of inversions, there is a natural edge labeling for the corresponding Cayley graph. This labeled Cayley graph, and thereby the splines on it, are entirely determined by the data of the inversion graph $\Gamma = ([n],S)$. This paper determines algebraic structures of splines on Cayley graphs of symmetric groups using the combinatorial data of the inversion graph $\Gamma$. \par 
 To discuss the results below, we begin with some notation. Let $\Gamma$ be a connected simple graph with vertex set $[n] \coloneqq \{1,\ldots,n\}$, and identify the edges in its edge set $E(\Gamma)$ with transpositions in $S_n$. This paper studies how properties of $\Gamma$ determine the algebraic structure of splines on the Cayley graph $\cg_{\Gamma}$ of $S_n$ with generating set $E(\Gamma)$ and edge label $(w,w(i,j)) \mapsto \llangle t_{w(i)} - t_{w(j)} \rrangle$. Formally, the \emph{ring of splines} is defined as
    \[
        \splines{\Gamma} \coloneqq \left\{ \left.\bar{\rho} \in \prod_{w \in S_n} \poly\; \right|\; \bar{\rho}(w) - \bar{\rho}(w(i,j)) \in \llangle t_{w(i)}-t_{w(j)} \rrangle\text{ when } (i,j) \in E(\Gamma) \right\},
   \]
   with (graded) $S_n$-module structure $w\cdot \bar{\rho}(v) = w\bar{\rho}(w^{-1}v)$ and (graded) $\poly$-module structure given by multiplication.  \par 
    This definition of the ring of splines generalizes the case where $\cg_\Gamma$ is the moment graph of a geometric object called a regular semisimple Hessenberg variety and the ring of splines is isomorphic to the equivariant cohomology of that variety \cite{DMPS1992hessenbergvarieties,GKM_theory,tymoczko2008permutation}. We call this the \emph{geometric case}, and in this case the corresponding graph $\Gamma$ is a \emph{Hessenberg graph}, commonly characterized in algebraic combinatorics as being the indifference graph of a $3+1$- and $2+2$-free poset. The more general setting considered in this paper allows one to spot patterns in rich algebraic structure that would otherwise be restricted for geometric reasons. For example, in the geometric case $\splines{\Gamma}$ is always a free module over the polynomial ring, whereas for general $\Gamma$ it is not. \par 
   The $S_n$-module structure on $\splines{\Gamma}$ was first defined in the geometric case as the \emph{dot action} on equivariant cohomology by Tymoczko in \cite{tymoczko2008permutation}. There are two natural $S_n$-equivariant quotients, $\lqot{\Gamma}$ and $\rqot{\Gamma}$, of $\splines{\Gamma}$ that are in fact graded $\C$-vector spaces. The graded $S_n$-module structure of $\splines{\Gamma}$ induces graded $S_n$-representations on the quotients, admitting (via the Frobenius character map $\mathbf{ch}$) two different graded symmetric functions: \[\lrep{\Gamma} \coloneqq \bigoplus_{i} \lrep{\Gamma}_i\;\text{ and }\;\rrep{\Gamma} \coloneqq \bigoplus_{i} \rrep{\Gamma}_i.\] These are manifestly Schur-positive symmetric function invariants of any simple graph. \par 
   The graded symmetric functions $\lrep{\Gamma}$ and $\rrep{\Gamma}$ are historically of interest to algebraic combinatorists because of their connections to chromatic symmetric functions \cite{guaypaquet2016shar_wachs_conj,SW2016chromaticquasisymmetric,brosnan_chow_dotactn_is_chromsym} and LLT polynomials \cite{guaypaquet2016shar_wachs_conj,Ayzenberg2018isospectral, ALEXANDERSSON2018LLTchromsym} in the geometric case. The two bases of symmetric functions we consider here are \emph{Schur functions} $\{s_\lambda\}$ and \emph{homogeneous symmetric functions} $\{h_\lambda\}$. In the geometric case, two major open problems seek (1) a homogeneous basis expansion of $\lrep{\Gamma}$ (\cite{STANLEY1995chromsym, Gasharov96,GuayPaquet13, SW2016chromaticquasisymmetric,  harada2017cohomology, brosnan_chow_dotactn_is_chromsym,Dahlberg19, Abreu_Nigro20}, and many others), and (2) a Schur basis expansion of $\rrep{\Gamma}$ (\cite{Leclerc2000LLT_LR_KL,Blasiak2016LLTMacdonald,guaypaquet2016shar_wachs_conj,Huh2020LLTLollipop,Alexandersson2020Lollipop,Lee2021LLTlinear} and many others). Again, our object of study is more general, and because of this we can identify patterns otherwise masked by geometric structure. For example, the Stanley--Stembridge conjecture \cite{SS1993immanantsofJTmatrices} claims that the homogeneous basis expansion of $\lrep{\Gamma}$ has only non-negative integer coefficients ($h$-positivity) in the geometric case. We observe below that this is not the case for general $\Gamma$, but $h$-positivity seems to occur whenever $\splines{\Gamma}$ is a free module over $\poly$. \par 

 This paper begins with several fundamental properties of $\splines{\Gamma}$. First, we establish the algebraic structure of $\splines{\Gamma}$ as an invariant of the graph $\Gamma$.
\begin{customlemma}{A}\label{intlem:isoms}
    An isomorphism of graphs $\Gamma \cong \Gamma'$ induces a ring isomorphism of splines $\splines{\Gamma} \cong \splines{\Gamma'}$, and equality of graded symmetric functions: $\lrep{\Gamma} = \lrep{\Gamma'}$ and $\rrep{\Gamma} = \rrep{\Gamma'}$.
\end{customlemma} 
 In particular, Lemma \ref{intlem:isoms} shows that the graded symmetric functions $\lrep{\Gamma}$ and $\rrep{\Gamma}$ are (Schur-positive) invariants of unlabeled simple graphs. Lemma \ref{intlem:isoms} is  proved via Propositions \ref{prop:isom_rings} and \ref{prop:isom_reps} below. \par  
 Then when $\Gamma$ is a tree, we determine explicit ring and module generators of $\splines{\Gamma}$ called \emph{coset splines} (Def. \ref{def:coset_spline}). 
 \begin{customthm}{B}\label{intthm:coset_splines}
     If $\Gamma$ is a tree, then the set of coset splines is a $\poly$-module generating set of $\splines{\Gamma}$, and the set of linear and constant coset splines is a ring generating set of $\splines{\Gamma}$.
 \end{customthm}
Since they generate, one can compute $\splines{\Gamma}$ explicitly with coset splines using a computer algebra system. Theorem \ref{intthm:coset_splines} is Theorem \ref{thm:coset_splines_spann} and Corollary \ref{cor:ring_gen} below. \par 
    We use Theorem \ref{intthm:coset_splines} to show that $\splines{\Gamma}$ is \emph{not} always a free $\poly$-module, and $\lrep{\Gamma}$ is \emph{not} always $h$-positive. One example is if $\Gamma = \left([4],\{(1,4),(2,4),(3,4)\}\right)$, then $\splines{\Gamma}$ is not a free module and $\lrep{\Gamma}_2$ is not $h$-positive. This also confirms that $\splines{\Gamma}$ is not always the equivariant cohomology of an (equivariantly formal) algebraic variety as in \cite{GKM_theory}, since in that case $\splines{\Gamma}$ is a free $\poly$-module. \par 
Our next main results, Theorems \ref{intthm:connectedness} and \ref{intthm:linear} below, explicitly compute certain graded pieces of the symmetric functions $\lrep{\Gamma}$ and $\rrep{\Gamma}$. Specifically, we determine when graded pieces of $\lrep{\Gamma}$ and $\rrep{\Gamma}$ are equal to $\lrep{K_n}$ and $\rrep{K_n}$ where $K_n$ is the complete graph ($\Gamma = K_n$ is a very special geometric case), and we compute $\lrep{\Gamma}_1$ and $\rrep{\Gamma}_1$ for all connected graphs $\Gamma$. \par 
For a variety of reasons, for example by formulae in \cite{SW2016chromaticquasisymmetric} or by some geometric observations, in the geometric case it is straightforward to tell from a Hessenberg graph  $H$ whether the symmetric function $\lrep{H}_d$ corresponds to a trivial representation. We achieve an analogous result for arbitrary graphs. The $k$-connectivity (Def. \ref{def:connectivity}) of a graph is a combinatorial invariant that measures how many vertices can be removed from a graph before it might become disconnected.
\begin{customthm}{C}\label{intthm:connectedness}
    Let $\Gamma$ be a connected simple graph. The following are equivalent:
    \begin{itemize}
        \item[1)] The graph $\Gamma$ is $k$-connected.
        \item[2)] For all $d<k$, the symmetric function $\lrep{\Gamma}_d$ corresponds to a trivial representation.  
        \item[3)] For all $d<k$, the $d$-th graded piece of $\splines{\Gamma}$ is isomorphic to the $d$-th graded piece of $\splines{K_n}$, where $K_n$ is the complete graph on $n$ vertices.
    \end{itemize}
\end{customthm}
Geometrically, the $d$-th graded piece of $\splines{K_n}$ is isomorphic to the $2d$-th equivariant cohomology of the full flag variety, and is thus spanned by equivariant Schubert classes whose spline formula is given in \cite{billey1999schub_formula}. Theorem \ref{intthm:connectedness} is a consequence of Theorem \ref{thm:kconn} below.\par 
When $\Gamma$ is a Hessenberg graph, the first graded piece of $\lrep{\Gamma}$ has been computed in a variety of ways. The Schur expansion is computed by counting $P$-tableaux \cite{SW2016chromaticquasisymmetric}. Expansions in the homogeneous basis have been computed with $P$-tableaux \cite{chow_linearp}, geometrically \cite{chohonglee_second_cohom}, as well as with splines \cite{ayzenberg2022second}. Our methods here most directly generalize those in \cite{ayzenberg2022second}. 
\begin{customthm}{D}\label{intthm:linear}
    Let $\Gamma$ be any connected simple graph. The first degree pieces of the graded symmetric functions $\lrep{\Gamma}$ and $\rrep{\Gamma}$ can be computed in both the Schur and homogeneous bases of symmetric functions from the data of (1) cut edges of $\Gamma$ and (2) cut vertices of $\Gamma$ and the number of connected components those vertices separate. \par 
    Formally, there exist a subset $E_1$ of cut edges, a subset $E_2$ of 2-connected subgraphs, a non-negative integer $k \in \N$, and a function $e \mapsto \lambda_e$ from $E_1$ to the set of partitions of $n$, such that 
    \[
    \lrep{\Gamma}_1 = \sum_{e \in E_1} h_{\lambda_e} + \left(\abs{E_2}-1\right) h_{n-1,1} + kh_n
    \]
    and 
    \[
    \rrep{\Gamma}_1 = \sum_{e \in E_1}\left(h_{\lambda_e} - s_n\right) +\abs{E_2} s_{n-1,1}.
    \]
\end{customthm}
Theorem \ref{intthm:linear} is Theorem \ref{thm:linear_reps} and Corollary \ref{cor:nonlabel_reps} below. The subsets $E_1$ and $E_2$ are defined using a combinatorial construction from the block-cut tree (Def. \ref{def:block_cut}) of $\Gamma$ in Section \ref{sec:lin_dim}.\par 
The paper is structured as follows. Section \ref{sec:background} constructs $\splines{\Gamma}$ and proves some of the fundamental algebraic properties, including the isomorphism of Lemma \ref{intlem:isoms}. Section \ref{sec:structure} builds tools for computing spline conditions from paths in $\Gamma$ and $\cg_\Gamma$. It also contains the construction of coset splines for trees and the proof that coset splines generate $\splines{\Gamma}$, Theorem \ref{intthm:coset_splines} above. Section \ref{sec:connectedness} leverages the tools in Section \ref{sec:structure} to prove our result on $k$-connectedness, Theorem \ref{intthm:connectedness} above. Sections \ref{sec:lin_gen_rels}, \ref{ssec:cliqued}, \ref{sec:lin_span}, \ref{sec:lin_dim}, and \ref{sec:lin_reps} are all to compute the representations in Theorem \ref{intthm:linear} above. Section \ref{sec:lin_gen_rels} defines a set of linear splines on the graph $\cg_\Gamma$ and proves some linear relations within that set. Section \ref{ssec:cliqued} reduces the computation to a subclass of graphs $\Gamma$ that will be used in all of the remaining sections. Section \ref{sec:lin_span} proves that the set of splines from Section \ref{sec:lin_gen_rels} is in fact a $\C$-spanning set for linear splines, and Section \ref{sec:lin_dim} computes the $\C$-dimension of this space. Finally Section \ref{sec:lin_reps} computes the first graded piece of $\lrep{\Gamma}$ and $\rrep{\Gamma}$, Theorem \ref{intthm:linear}. 

\section{Background}\label{sec:background}
There is a natural action of the symmetric group $S_n$ on the polynomial ring $\poly$ by 
\begin{equation}\label{eqn:sn_action_poly}
    w f\left(t_1,\ldots,t_n\right) \mapsto f\left(t_{w(1)},\ldots,t_{w(n)}\right).
\end{equation}
We use both one-line and cycle notation for elements of $S_n$. We denote a permutation's cycle notation with parentheses and commas, and its one-line notation without, so that $(1,2,3)=231$.
\subsection{Graphs: Simple and Cayley}
This subsection establishes the basic definitions, results, and notation from graph theory needed below. A \emph{graph} is a tuple $\Gamma = (V,E)$ where $V$ is the set of vertices and $E \subset V\times V$ is the set of edges. Graphs here are understood to be undirected and simple, i.e. finite, loopless, and without multiple edges. We will always take $\Gamma$ to be connected, and may remind the reader of this assumption where particularly important. Write $E(\Gamma)$ for the edge set of a graph $\Gamma$ and $V(\Gamma)$ for the vertex set. Inclusion $v \in \Gamma$ means $v \in V(\Gamma)$. \par 
If the vertex set $V$ has some natural linear order (in particular when $V = [n]$), then an edge between vertices $i < j$ will always be written with the lower vertex first $(i,j)$, unless explicitly stated otherwise. Note that these edges are undirected, so an edge $(i,j)$ is the same as an edge $(j,i)$.\par 
We denote graphs pictorially with circles as vertices and lines as edges between them, for example we would display a particular graph $\Gamma$ on $9$ vertices as 
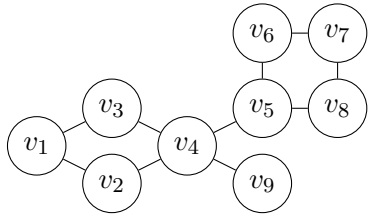
\begin{center}
    \begin{tikzpicture}
        \begin{scope}
            \draw (0,0) node[draw,circle] (V11) {$v_1$};
            \draw (1,-0.5) node[draw,circle] (V12) {$v_2$};
            \draw (1,0.5) node[draw,circle] (V13) {$v_3$};
            \draw (2,0) node[draw,circle] (V14) {$v_4$};
            \draw (3,0.5) node[draw,circle] (V21) {$v_5$};
            \draw (3,1.5) node[draw,circle] (V22) {$v_6$};
            \draw (4,1.5) node[draw,circle] (V23) {$v_7$};
            \draw (4,0.5) node[draw,circle] (V24) {$v_8$};
            \draw (3,-0.5) node[draw,circle] (V31) {$v_9$};

            \draw (V11)--(V12);
            \draw (V11)--(V13);
            \draw (V12)--(V14);
            \draw (V13)--(V14);
            \draw (V14)--(V21);
            \draw (V14)--(V31);

            \draw (V21)--(V22);
            \draw (V22)--(V23);
            \draw (V23)--(V24);
            \draw (V24)--(V21);

        \end{scope}
        \end{tikzpicture}
        \end{center}
 The \emph{induced subgraph} of $\Gamma$ with vertex set $V \setminus A$ is $\Gamma \setminus A \coloneqq (V \setminus A,E')$ where $E' = E \cap \left( V \setminus A \times V \setminus A\right)$. We write $\Gamma - v$ for $\Gamma\setminus \{v\}$. When collapsing a subgraph in drawing, we reference the subgraph in a square to distinguish that there are multiple vertices being referenced, and double lines connecting to acknowledge the possibility of multiple edges. For example, we may display $\Gamma$ above as \par 
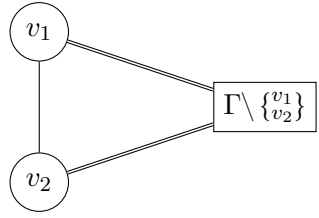
\begin{center}
    \begin{tikzpicture}
            \draw (2,1) node[draw,circle] (V2) {$v_1$};
            \draw (2,-1) node[draw,circle] (V3) {$v_2$};
            \draw (5,0) node[draw,rectangle] (V4) {$\Gamma \setminus \left\{\substack{v_1 \\ v_2}\right\}$};     

            \draw (V2)--(V3);
            \draw[double] (V2)--(V4);
            \draw[double] (V3)--(V4);
    \end{tikzpicture}
\end{center}
if the structure within $\Gamma \setminus \left\{v_1  ,v_2\right\}$ is not needed.
\begin{definition}
    For a graph $\Gamma$, a set $A \subset V(\Gamma)$ is a \emph{cut set} if $\Gamma \setminus A$ is disconnected. Similarly $v \in V(\Gamma)$ is a \emph{cut vertex} of $\Gamma$ (denoted $v \vdash \Gamma$) if $\Gamma - v$ is disconnected.\par 
    An edge $e \in E(\Gamma)$ is a \emph{cut edge} if the graph $(V(\Gamma),E(\Gamma)\setminus\{e\})$ is disconnected.
\end{definition}
A \emph{path} in $\Gamma$ from vertex $v_0$ to vertex $v_\ell$ of length $\ell$ is a sequences of vertices $(v_0,v_1,...,v_\ell)$ where $(v_k,v_{k+1}) \in E(\Gamma)$ for $k=0,...,\ell-1$. Define the \emph{distance} $d(v,w)$ between $v$ and $w$ as the minimum length over all paths from $v$ to $w$, and let $d(v,w) \coloneqq \infty$ if no such path exists.
\begin{definition}\label{def:connectivity}
    A graph $\Gamma = (V,E)$ is \emph{$k$-connected} if $\Gamma \setminus A$ is connected for all $A \subset V$ such that $\abs{A} \leq k-1$.
\end{definition}
In other words, a graph is $k$-connected if there exists no cut-set $A$ where $\abs{A} <k$. The following is an equivalent characterization used in $\S$\ref{sec:connectedness}.
\begin{theorem}[Menger's Theorem]\label{thm:menger}
    A graph $\Gamma$ is $k$-connected if and only if for every pair of vertices $i,j \in \Gamma$, there exist at least $k$ vertex-disjoint paths from $i$ to $j$.
\end{theorem}

An \emph{$R$-labeled graph} is a tuple $(V,E,L)$, where $(V,E)$ is a graph and $L$ is a function $L \colon E \to R$ for some set $R$. 
A \emph{Cayley graph} of a group $G$ and a set of generators $S$ is the graph $\left(G,\{(g,h) \mid g^{-1}h \in S\}\right)$. 
Cayley graphs are usually directed graphs, but all generators considered here will be involutions and so the Cayley graphs will be undirected simple graphs. Note that $g^{-1}h \in S$ if and only if $h=gs$ for $s \in S$, so edges in a Cayley graph correspond to right multiplication by generators.\par 
This paper concerns graphs $\Gamma$ on vertex set $[n]$, and labelled Cayley graphs of the symmetric group with generators being some set of transpositions. The edge labels are principle ideals in $\poly$. 
\begin{definition}
    Let $\Gamma$ be a graph on $[n]$. Identify each edge $(i,j) \in E(\Gamma)$ with the transposition $(i,j) \in S_n$. The \emph{labelled Cayley graph associated to $\Gamma$} is $\cg_\Gamma := (\cv,\ce,\cl)$ where 
    \begin{itemize}
        \item $\cv = S_n$,
        \item $\ce = \{(w,v) \mid w^{-1}v \in E(\Gamma)\}$, and 
        \item $\cl(w,v) = \llangle t_i-t_j \rrangle$ where $(i,j) = wv^{-1}$.
    \end{itemize}
\end{definition}
Note $w^{-1}v$ is conjugate to $wv^{-1}$, so if $w = v(i,j)$ then $\cl(w,v) = \langle t_{w(i)}-t_{w(j)}\rangle = \langle t_{v(i)}-t_{v(j)}\rangle.$ Note also that $\cl$ is defined whenever $wv^{-1}$ is a transposition.  \par 
\begin{example}\label{ex:labelled_Cayley}
Let $\Gamma = \left([3],\{(1,2),(2,3)\}\right)$. Then $\cg_\Gamma$ has vertex set $S_3$, edges $\{(w,v) \mid w^{-1}v \in \{(1,2),(2,3)\}$, and labels of the form $\langle t_i-t_j\rangle$ where $i,j \in [3]$. 
Below is $\cg_\Gamma$, with labeling ideals denoted by generators.

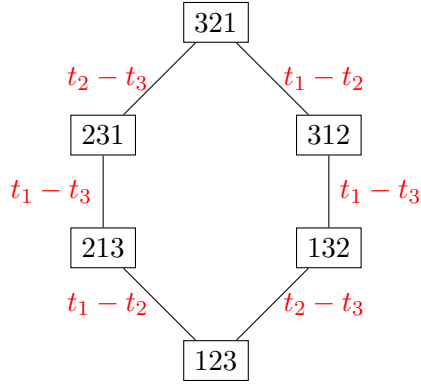
\begin{center}
    \begin{tikzpicture}[scale=1.5]
\tikzstyle{every node}=[shape=rectangle, inner sep=4pt];
\draw  (0,0) node[draw] (v0) {$123$};
\draw  (1,1) node[draw]  (v1) {$132$};
\draw  (-1,1) node[draw] (v2) {$213$};
\draw  (-1,2) node[draw] (v3) {$231$};
\draw  (1,2) node[draw] (v4) {$312$};
\draw (0,3)  node[draw]  (v5) {$321$};

\draw (v0)--(v1) node[midway,right] {\textcolor{red}{ $t_2-t_3$}};
\draw (v0)--(v2) node[midway,left] {\textcolor{red}{ $t_1-t_2$}};
\draw (v1)--(v4)node[midway,right] {\textcolor{red}{ $t_1-t_3$}};
\draw (v2)--(v3) node[midway,left] {\textcolor{red}{ $t_1-t_3$}};
\draw (v3)--(v5)node[midway,left] {\textcolor{red}{ $t_2-t_3$}};
\draw (v4)--(v5)node[midway,right] {\textcolor{red}{ $t_1-t_2$}};

\end{tikzpicture}
\end{center}
Consider the edge $(132,312)$. These permutations have their first and second positions swapped, corresponding to right multiplication by $(1,2) \in E(\Gamma)$. The edge is labelled $\langle t_1-t_3\rangle$ because these permutations have the \emph{entries} $1$ and $3$ swapped, corresponding to left multiplication by $(1,3)$.  
\end{example}
The \emph{$\Gamma$-length} of a permutation $w \in S_n$ is 
\begin{equation}\label{eqn:gamma_length}
    \ell_\Gamma(w) \coloneqq\min\left\{ \ell \mid w = s_1\cdots s_\ell,\;\; \{s_1,\ldots,s_\ell\} \subseteq E(\Gamma)\right\}.
\end{equation}
This is also the value of $d(e,w)$ in $\cg_\Gamma$. When $\Gamma$ is the path graph, $\Gamma$-length is the traditional length function on permutations.
\subsection{Splines}\label{ssec:background_splines}
This section introduces the ring of splines on a labelled Cayley graph. The lemmas in this subsection are well known and straightforward, but we include proofs for completeness.
\begin{definition}
    Let $\Gamma$ be a graph on $[n]$. A \emph{spline} on $\cg_\Gamma$ is a function $\bar{\rho} \colon S_n \to \poly$ such that $\bar{\rho}(w) - \bar{\rho}(v) \in \cl(w,v)$ whenever $(w,v) \in E\left(\cg_\Gamma\right)$. The \emph{support} of the spline $\bar{\rho}$ is the set $\supp(\bar{\rho})\coloneqq \{w \mid \bar{\rho}(w) \neq 0\}$.
\end{definition}
 To distinguish from polynomials, we always denote a spline with a bar.
\begin{example}\label{ex:splines}
Again consider $\Gamma = \left([3],\{(1,2),(2,3)\}\right)$. Drawn below (omitting edge-labels) are three examples of splines on $\cg_\Gamma$.
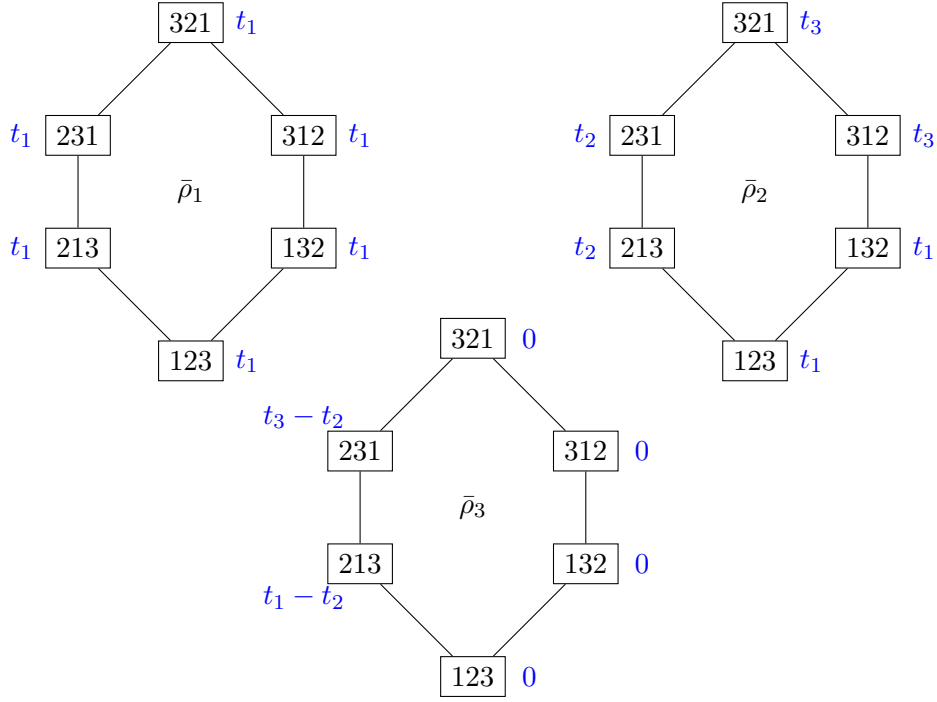
\begin{center}
    \begin{tikzpicture}[scale=1.5]
\tikzstyle{every node}=[shape=rectangle, inner sep=4pt];
\begin{scope}[xshift = 0]
\draw  (0,0) node[draw] (v0) {$123$};
\draw  (1,1) node[draw]  (v1) {$132$};
\draw  (-1,1) node[draw] (v2) {$213$};
\draw  (-1,2) node[draw] (v3) {$231$};
\draw  (1,2) node[draw] (v4) {$312$};
\draw (0,3)  node[draw]  (v5) {$321$};
\draw (0,1.5) node (name) {$\bar{\rho}_1$};

\draw  (0.5,0) node (v00) {\textcolor{blue}{$t_1$}};
\draw  (1.5,1) node  (v10) {\textcolor{blue}{$t_1$}};
\draw  (-1.5,1) node (v20) {\textcolor{blue}{$t_1$}};
\draw  (-1.5,2) node (v30) {\textcolor{blue}{$t_1$}};
\draw  (1.5,2) node (v40) {\textcolor{blue}{$t_1$}};
\draw (0.5,3)  node  (v50) {\textcolor{blue}{$t_1$}};

\draw (v0)--(v1);
\draw (v0)--(v2);
\draw (v1)--(v4);
\draw (v2)--(v3);
\draw (v3)--(v5);
\draw (v4)--(v5);
\end{scope}

\begin{scope}[xshift = 5cm]
\draw  (0,0) node[draw] (v0) {$123$};
\draw  (1,1) node[draw]  (v1) {$132$};
\draw  (-1,1) node[draw] (v2) {$213$};
\draw  (-1,2) node[draw] (v3) {$231$};
\draw  (1,2) node[draw] (v4) {$312$};
\draw (0,3)  node[draw]  (v5) {$321$};
\draw (0,1.5) node (name) {$\bar{\rho}_2$};

\draw  (0.5,0) node (v00) {\textcolor{blue}{$t_1$}};
\draw  (1.5,1) node  (v10) {\textcolor{blue}{$t_1$}};
\draw  (-1.5,1) node (v20) {\textcolor{blue}{$t_2$}};
\draw  (-1.5,2) node (v30) {\textcolor{blue}{$t_2$}};
\draw  (1.5,2) node (v40) {\textcolor{blue}{$t_3$}};
\draw (0.5,3)  node  (v50) {\textcolor{blue}{$t_3$}};

\draw (v0)--(v1);
\draw (v0)--(v2);
\draw (v1)--(v4);
\draw (v2)--(v3);
\draw (v3)--(v5);
\draw (v4)--(v5);
\end{scope}

\begin{scope}[xshift = 2.5cm, yshift=-2.8cm]
\draw  (0,0) node[draw] (v0) {$123$};
\draw  (1,1) node[draw]  (v1) {$132$};
\draw  (-1,1) node[draw] (v2) {$213$};
\draw  (-1,2) node[draw] (v3) {$231$};
\draw  (1,2) node[draw] (v4) {$312$};
\draw (0,3)  node[draw]  (v5) {$321$};
\draw (0,1.5) node (name) {$\bar{\rho}_3$};

\draw  (0.5,0) node (v00) {\textcolor{blue}{$0$}};
\draw  (1.5,1) node  (v10) {\textcolor{blue}{$0$}};
\draw  (-1.5,0.7) node (v20) {\textcolor{blue}{$t_1-t_2$}};
\draw  (-1.5,2.3) node (v30) {\textcolor{blue}{$t_3-t_2$}};
\draw  (1.5,2) node (v40) {\textcolor{blue}{$0$}};
\draw (0.5,3)  node  (v50) {\textcolor{blue}{$0$}};

\draw (v0)--(v1);
\draw (v0)--(v2);
\draw (v1)--(v4);
\draw (v2)--(v3);
\draw (v3)--(v5);
\draw (v4)--(v5);
\end{scope}
\end{tikzpicture}
\end{center}
So $\bar{\rho}_1(w) = t_1$ for all $w \in S_3$, $\bar{\rho}_2(w) = t_{w(1)}$ for all $w \in S_3$, and ${\ds \bar{\rho_3}(w) = \begin{cases}
    t_1-t_2 & \text{if }w = 213 \\ t_3-t_2 & \text{if }w = 231 \\ 0 & \text{otherwise.}
\end{cases}}$
\end{example}
The set of splines is closed under addition, as well as multiplication.
\begin{lemma}\label{lem:right_action_works}
    Let $\Gamma$ be a graph on $[n]$. If $\bar{\rho}$ and $\bar{\sigma}$ are splines on $\cg_\Gamma$, then so is $\bar{\sigma}\bar{\rho}$, the spline constructed via point-wise multiplication.
\end{lemma}
\begin{proof}
    Let $(w,v) \in E(\cg_\Gamma)$. By assumption $\bar{\rho}(w)-\bar{\rho}(v) \in \cl(w,v)$ and $\bar{\sigma}(w)-\bar{\sigma}(v) \in \cl(w,v)$. We have
    \begin{align*}
        \bar{\rho}(w)\bar{\sigma}(w) - \bar{\rho}(v)\bar{\sigma}(v) &= \bar{\rho}(w)\bar{\sigma}(w) - \bar{\sigma}(v)\bar{\rho}(w) +\bar{\sigma}(v)\bar{\rho}(w)   - \bar{\rho}(v)\bar{\sigma}(v) \\
        &= \bar{\rho}(w)\left(\bar{\sigma}(w) - \bar{\sigma}(v)\right) +\bar{\sigma}(v)\left(\bar{\rho}(w)   - \bar{\rho}(v)\right),
    \end{align*}
    and sum is clearly in $\cl(w,v)$.
\end{proof}
\begin{definition}
    The \emph{ring of splines} on $\cg_\Gamma$ is the subring 
    \[
        \splines{\Gamma} \coloneqq\left\{ \left.\bar{\rho} \in \prod_{w \in S_n} \poly\; \right|\; \bar{\rho}(w) - \bar{\rho}(v) \in \cl(w,v) \text{ for all } (w,v) \in E(\cg_\Gamma)\right\}.
   \]
   of ${\ds \prod_{w \in S_n} \poly}$ with pointwise addition and multiplication. 
\end{definition}
\begin{lemma}\label{lem:ring_is_graded}
    The ring $\splines{\Gamma}$ is graded by degree, so ${\ds \splines{\Gamma} = \bigoplus_{i \geq 0} \splines{\Gamma}^i}$.
\end{lemma}
\begin{proof}
    Let $\bar{\rho}$ be a spline in $\splines{\Gamma}$ and let $\bar{\rho}_k(w)$ be the $k$-th graded piece of the polynomial $\bar{\rho}(w)$. We aim to show that $\bar{\rho}_k$ is a spline as well. For each $(w,v) \in E(\cg_\Gamma)$, the ideal $\cl(w,v)$ is a homogeneous ideal. Thus $\bar{\rho}(w)-\bar{\rho}(v) \in \cl(w,v)$ and it follows that $\bar{\rho}_k(w)-\bar{\rho}_k(v) \in \cl(w,v)$, so $\bar{\rho}_k$ is a spline. For two homogeneous splines $\bar{\rho}$ and $\bar{\sigma}$ of degrees $p$ and $q$ respectively, the product $\bar{\rho}\bar{\sigma}$ is homogeneous of degree $p+q$ on its support. 
    \end{proof}
We now construct two sets of splines and the \emph{identity spline}, each are elements of $\splines{\Gamma}$ for all $\Gamma$. Let 
\begin{center} 
\begin{tabular}{ll}
    $\bar{\mathbb{1}} \colon S_n \to \poly$ be $\bar{\mathbb{1}}(w)\coloneqq 1$ & for all $w \in S_n$,\\
    $\bar{t}_i \colon S_n \to \poly$ be $\bar{t}_i(w) \coloneqq t_i$ & for all $w\in S_n$, $i\in\{1,\ldots,n\}$, and\\
    $\bar{x}_i \colon S_n \to \poly$ be $\bar{x}_i(w) \coloneqq t_{w(i)}$ & for all $w\in S_n$, $i\in\{1,\ldots,n\}$.
\end{tabular}
\end{center}
The ring $\splines{\Gamma}$ is an infinite-dimensional $\C$-vector space in the natural way, and can also be viewed as a finitely generated graded $\poly$-module in two ways via the following module actions:
\begin{equation}\label{eqn:left_action}
    f\left(t_1,\ldots,t_n\right).\bar{\rho} = f\left(\bar{t}_1,\ldots,\bar{t}_n\right)\bar{\rho}
\end{equation}
and
\begin{equation}\label{eqn:right_action}
    f\left(t_1,\ldots,t_n\right).\bar{\rho} = f\left(\bar{x}_1,\ldots,\bar{x}_n\right)\bar{\rho},
\end{equation}
where the right-hand side of both (\ref{eqn:left_action}) and (\ref{eqn:right_action}) work by substituting splines for variables in to the polynomial $f$ then multiplying as in the ring structure of $\splines{\Gamma}$. For both actions the constant $f(0,\ldots,0)$ is naturally mapped to $f(0,\ldots,0)\bar{\mathbb{1}}$. Since $\splines{\Gamma}$ is a $\poly$-submodule of $\prod_{w \in S_n} \poly$ for either module action, it is finitely generated. We call the module action (\ref{eqn:left_action}) the \emph{left action} and the module action (\ref{eqn:right_action}) the \emph{right action} of $\poly$ on $\splines{\Gamma}$. Given any $\omega \in S_n$ both actions may be twisted by sending $f \mapsto \omega f$ first in the polynomial ring. Both the left and right actions are naturally compatible with the grading on $\splines{\Gamma}$. \par 
\begin{example}\label{ex:poly_action}
    Let $\bar{\rho} \in \splines{\Gamma}$ and let $f(t_\bullet) = t_1^3 + t_2^2 + t_3$. Let $\omega=(1,2,3) \in S_n$. \\
    The left action of $f$ on $\bar{\rho}$ evaluated at any $v \in S_n$ is: \[ f(t_\bullet).\bar{\rho}(v) = \left[((\bar{t}_1)^3 + (\bar{t}_2)^2 + \bar{t}_3)\bar{\rho}\right](v) = (t_1^3 + t_2^2 + t_3)\bar{\rho}(v),\]
    the right action of $f$ on $\bar{\rho}$ evaluated at any $v \in S_n$ is: \[ f(t_\bullet).\bar{\rho}(v) =\left[((\bar{x}_1)^3 + (\bar{x}_2)^2 + \bar{x}_3)\bar{\rho}\right](v) = (t_{v(1)}^3 + t_{v(2)}^2 + t_{v(3)})\bar{\rho}(v),\]
    the $\omega$-twisted left action of $f$ on $\bar{\rho}$ evaluated at any $v \in S_n$ is: \[ f(t_\bullet).\bar{\rho}(v) = \left[((\bar{t}_{\omega(1)})^3 + (\bar{t}_{\omega(2)})^2 + \bar{t}_{\omega(3)})\bar{\rho}\right](v) = (t_2^3 + t_3^2 + t_1)\bar{\rho}(v) ,\]
    and the $\omega$-twisted right action of $f$ on $\bar{\rho}$ evaluated at any $v \in S_n$ is: \[ f(t_\bullet).\bar{\rho}(v) = \left[((\bar{x}_{\omega(1)})^3 + (\bar{x}_{\omega(2)})^2 + \bar{x}_{\omega(3)})\bar{\rho}\right](v)= (t_{v(2)}^3 + t_{v(3)}^2 + t_{v(1)})\bar{\rho}(v).\]
\end{example}
The ring of splines has a $S_n$-module structure, originally defined for Hessenberg graphs in \cite{tymoczko2008permutation,tymoczko2008schubertreps}. 
\begin{definition}
    Let $\bar{\rho} \in \splines{\Gamma}$. The \emph{dot action} of $S_n$ on $\splines{\Gamma}$ is given by 
    \[
    w\cdot \bar{\rho}(v) \coloneqq w\bar{\rho}(w^{-1}v)
    \]
    for $w,v \in S_n$.
    Any $\omega \in S_n$ may twist the dot action by first sending $v \to \omega v \omega^{-1}$ (conjugating by $\omega$). Since conjugation is an inner automorphism of $S_n$, the standard and $\omega$-twisted $S_n$-module structures on $\splines{\Gamma}$ are isomorphic.
\end{definition}
Using our standard for visualizing splines, the dot action by $w$ moves polynomials around $\cg_\Gamma$ by sending the polynomial at $v$ to $wv$ (for all $v \in S_n$), then acts on every polynomial by $w$ as in Equation (\ref{eqn:sn_action_poly}). \begin{example}\label{ex:dot_action}
The dot action of the transposition $(1,2)$ on the spline $\bar{\rho}_3$ from Example \ref{ex:splines} is computed below. 
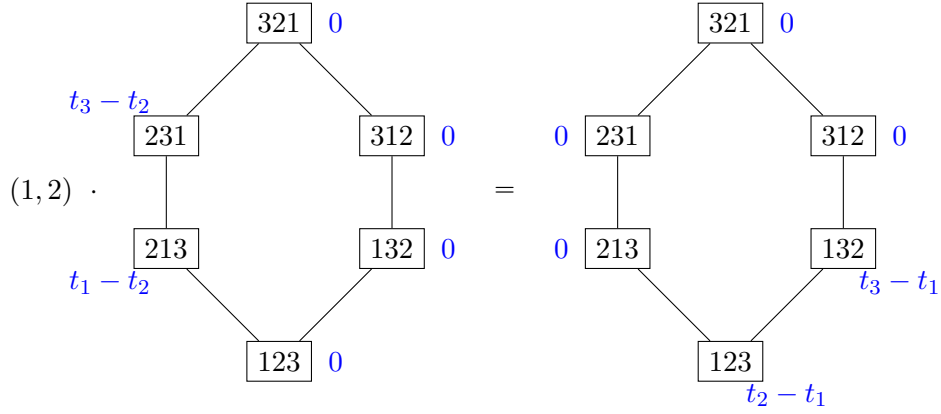
\begin{center}
    
    \begin{tikzpicture}[scale=1.5]
\tikzstyle{every node}=[shape=rectangle, inner sep=4pt];
\begin{scope}[xshift = 0cm]
\draw (0,3.2) node (h) {};
\draw (-2,1.5) node (dot) {$(1,2)\;\; \cdot$};
\draw  (0,0) node[draw] (v0) {$123$};
\draw  (1,1) node[draw]  (v1) {$132$};
\draw  (-1,1) node[draw] (v2) {$213$};
\draw  (-1,2) node[draw] (v3) {$231$};
\draw  (1,2) node[draw] (v4) {$312$};
\draw (0,3)  node[draw]  (v5) {$321$};

\draw  (0.5,0) node (v00) {\textcolor{blue}{$0$}};
\draw  (1.5,1) node  (v10) {\textcolor{blue}{$0$}};
\draw  (-1.5,0.7) node (v20) {\textcolor{blue}{$t_1-t_2$}};
\draw  (-1.5,2.3) node (v30) {\textcolor{blue}{$t_3-t_2$}};
\draw  (1.5,2) node (v40) {\textcolor{blue}{$0$}};
\draw (0.5,3)  node  (v50) {\textcolor{blue}{$0$}};

\draw (v0)--(v1);
\draw (v0)--(v2);
\draw (v1)--(v4);
\draw (v2)--(v3);
\draw (v3)--(v5);
\draw (v4)--(v5);
\end{scope}
\begin{scope}[xshift = 4cm]
\draw (-2,1.5) node (dotact) {$=$};
\draw  (0,0) node[draw] (v0) {$123$};
\draw  (1,1) node[draw]  (v1) {$132$};
\draw  (-1,1) node[draw] (v2) {$213$};
\draw  (-1,2) node[draw] (v3) {$231$};
\draw  (1,2) node[draw] (v4) {$312$};
\draw (0,3)  node[draw]  (v5) {$321$};

\draw  (0.5,-0.3) node (v00) {\textcolor{blue}{$t_2-t_1$}};
\draw  (1.5,0.7) node  (v10) {\textcolor{blue}{$t_3-t_1$}};
\draw  (-1.5,1) node (v20) {\textcolor{blue}{$0$}};
\draw  (-1.5,2) node (v30) {\textcolor{blue}{$0$}};
\draw  (1.5,2) node (v40) {\textcolor{blue}{$0$}};
\draw (0.5,3)  node  (v50) {\textcolor{blue}{$0$}};

\draw (v0)--(v1);
\draw (v0)--(v2);
\draw (v1)--(v4);
\draw (v2)--(v3);
\draw (v3)--(v5);
\draw (v4)--(v5);
\end{scope}
\end{tikzpicture}
\end{center}
Computed below is the $\omega = (1,2,3)$-twisted action of the  transposition $(1,2)$ on the spline $\bar{\rho}_3$ from Example \ref{ex:splines}. Note this is the same as the untwisted action of $(1,2,3)(1,2)(1,2,3)^{-1} = (2,3)$.
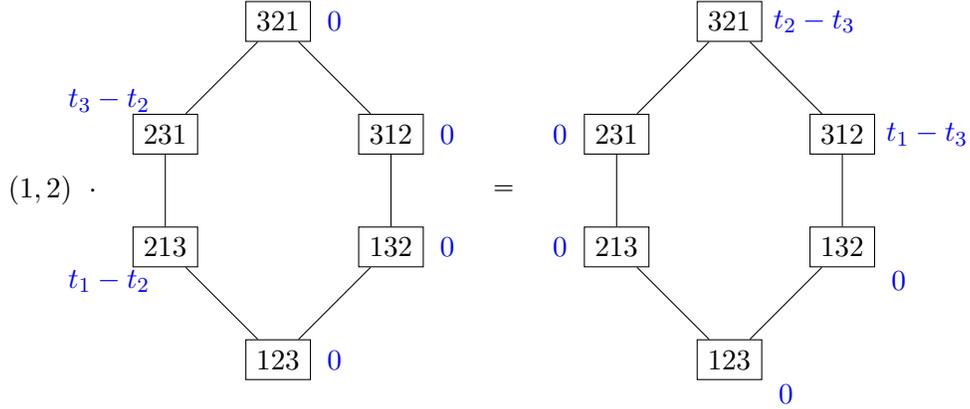
\begin{center}
    
    \begin{tikzpicture}[scale=1.5]
\tikzstyle{every node}=[shape=rectangle, inner sep=4pt];
\begin{scope}[xshift = 0cm]
\draw (0,3.2) node (h) {};
\draw (-2,1.5) node (dot) {$(1,2)\;\; \cdot$};
\draw  (0,0) node[draw] (v0) {$123$};
\draw  (1,1) node[draw]  (v1) {$132$};
\draw  (-1,1) node[draw] (v2) {$213$};
\draw  (-1,2) node[draw] (v3) {$231$};
\draw  (1,2) node[draw] (v4) {$312$};
\draw (0,3)  node[draw]  (v5) {$321$};

\draw  (0.5,0) node (v00) {\textcolor{blue}{$0$}};
\draw  (1.5,1) node  (v10) {\textcolor{blue}{$0$}};
\draw  (-1.5,0.7) node (v20) {\textcolor{blue}{$t_1-t_2$}};
\draw  (-1.5,2.3) node (v30) {\textcolor{blue}{$t_3-t_2$}};
\draw  (1.5,2) node (v40) {\textcolor{blue}{$0$}};
\draw (0.5,3)  node  (v50) {\textcolor{blue}{$0$}};

\draw (v0)--(v1);
\draw (v0)--(v2);
\draw (v1)--(v4);
\draw (v2)--(v3);
\draw (v3)--(v5);
\draw (v4)--(v5);
\end{scope}
\begin{scope}[xshift = 4cm]
\draw (-2,1.5) node (dotact) {$=$};
\draw  (0,0) node[draw] (v0) {$123$};
\draw  (1,1) node[draw]  (v1) {$132$};
\draw  (-1,1) node[draw] (v2) {$213$};
\draw  (-1,2) node[draw] (v3) {$231$};
\draw  (1,2) node[draw] (v4) {$312$};
\draw (0,3)  node[draw]  (v5) {$321$};

\draw  (0.5,-0.3) node (v00) {\textcolor{blue}{$0$}};
\draw  (1.5,0.7) node  (v10) {\textcolor{blue}{$0$}};
\draw  (-1.5,1) node (v20) {\textcolor{blue}{$0$}};
\draw  (-1.5,2) node (v30) {\textcolor{blue}{$0$}};
\draw  (1.75,2) node (v40) {\textcolor{blue}{$t_1-t_3$}};
\draw (0.75,3)  node  (v50) {\textcolor{blue}{$t_2-t_3$}};

\draw (v0)--(v1);
\draw (v0)--(v2);
\draw (v1)--(v4);
\draw (v2)--(v3);
\draw (v3)--(v5);
\draw (v4)--(v5);
\end{scope}
\end{tikzpicture}
\end{center}
\end{example}
\begin{remark} The dot action is well defined. We have for all $(v_1,v_2) \in E(\cg_\Gamma)$ that 
\begin{align*}
w \cdot \bar{\rho}(v_1) - w \cdot \bar{\rho}(v_2) &= w\bar{\rho}(w^{-1}v_1)- w\bar{\rho}(w^{-1}v_2) \\
&= w(\bar{\rho}(w^{-1}v_1)- \bar{\rho}(w^{-1}v_2)) \\
&\in w\cl(w^{-1}v_1,w^{-1}v_2).
\end{align*}
If $v_1v_2^{-1} = (i,j)$ then $w^{-1}v_1v_2^{-1}w = (w^{-1}(i),w^{-1}(j))$. So 
\[w\cl(w^{-1}v_1,w^{-1}v_2) = \llangle w(t_{w^{-1}(i)}-t_{w^{-1}(j)}) \rrangle = \llangle t_i-t_j \rrangle  = \cl(v_1,v_2).\] 
Thus $w \cdot \bar{\rho}(v_1) - w \cdot \bar{\rho}(v_2) \in \cl(v_1,v_2)$, and $w \cdot \bar{\rho} \in \splines{\Gamma}$.
\end{remark}
Finally, consider the quotients 
\begin{equation}\label{eqndef:left_quotient}
    \mathrm{L}_\Gamma \coloneqq \faktor{\splines{\Gamma}}{\llangle \bar{t}_1,\ldots,\bar{t}_n\rrangle}
\end{equation}
and
\begin{equation}\label{eqndef:right_quotient}
    \mathrm{R}_\Gamma \coloneqq \faktor{\splines{\Gamma}}{\llangle \bar{x}_1,\ldots,\bar{x}_n\rrangle}
\end{equation}
Call $\mathrm{L}_\Gamma$ and $\mathrm{R}_\Gamma$ the left and right quotients of $\splines{\Gamma}$ respectively. As $\poly$-modules for the left and right action, both quotients are $\faktor{\splines{\Gamma}}{I\splines{\Gamma}}$ where $I$ is the ``irrelevant ideal" $\llangle t_1,...,t_n\rrangle$ of $\polyr{n}$. Thus, $\lqot{\Gamma}$ and $\rqot{\Gamma}$ each inherit the structure of a finite-dimensional graded $\C$-vector space from the left- and right-module structure of $\splines{\Gamma}$ respectively. Any homogeneous module-generating set over $\poly$ projects to a spanning set over $\C$ in the quotient.\par 
The ideals $\llangle \bar{t}_1,\ldots,\bar{t}_n\rrangle$ and $\llangle \bar{x}_1,\ldots,\bar{x}_n\rrangle$ are homogeneous and $S_n$-equivariant, and so the graded $S_n$-module structure on $\splines{\Gamma}$ projects to graded $S_n$-representations on both $\lqot{\Gamma}$ and $\rqot{\Gamma}$. \emph{Symmetric functions} are formal power series in $\{x_1,x_2,...\}$ invariant under permuting the variables. The Frobenius character map gives an isomorphism from the algebra of representations of symmetric groups to the algebra of symmetric functions. The two bases of symmetric functions we consider are \emph{Schur functions} $\{s_\lambda\}$, which correspond to irreducible representations, and \emph{homogeneous symmetric functions} \{$h_\lambda\}$, which correspond to induced representations of trivial representations on Young subgroups to symmetric groups. Both Schur and homogeneous symmetric functions are indexed by integer partitions. Denote the Frobenius character of these ($q$-graded) $S_n$-representations as $\lrep{\Gamma}$ and $\rrep{\Gamma}$ respectively. Since both $\lrep{\Gamma}$ and $\rrep{\Gamma}$ correspond to graded representations, and all representations are sums of irreducible representations, both $\lrep{\Gamma}$ and $\rrep{\Gamma}$ are manifestly Schur-positive graded symmetric functions.
\begin{example}
    Again consider $\Gamma = \left([3],\{(1,2),(2,3)\}\right)$. Then 
    \[\lrep{\Gamma} = s_3 + (s_{2,1} + 2s_3)q + s_3q^2=h_3 + (h_{2,1}+h_3)q + h_3q^2 \] and 
    \[\rrep{\Gamma} = s_3 + 2s_{2,1}q + s_3q^2. \]
\end{example}
The following Lemma \ref{lem:unions} is useful for computer calculations.
\begin{lemma}\label{lem:unions}
    Let $\Gamma$ and $\Gamma'$ be two graphs on $[n]$, and $\Gamma \cup \Gamma' \coloneqq \left([n], E(\Gamma) \cup E(\Gamma')\right)$. Then 
    \[
    \splines{\Gamma \cup \Gamma'} = \splines{\Gamma} \cap \splines{\Gamma'}
    \]
\end{lemma}
\begin{proof}
    This easily follows from the set-theoretic definition \[\splines{\Gamma} = \left\{ \left.\bar{\rho} \in \prod_{w \in S_n} \poly\; \right|\; \bar{\rho}(w) - \bar{\rho}(v) \in \cl(w,v) \text{ for all } (w,v) \in E(\cg_\Gamma)\right\}.\]
\end{proof}
\subsection{Isomorphisms}\label{ssec:isoms}
It is natural to expect that if two graphs $\Gamma$ and $\Gamma'$ on $[n]$ are isomorphic, that the resulting algebraic structures on  $\splines{\Gamma}$ and $\splines{\Gamma'}$ should also have meaningful isomorphisms between them. This section shows that an isomorphism $\Gamma \to \Gamma'$ induces a labeled-graph isomorphism $\cg_\Gamma \to \cg_{\Gamma'}$, a ring isomorphism $\splines{\Gamma} \to \splines{\Gamma'}$, a collection of different $\poly$-module isomorphisms $\splines{\Gamma} \to \splines{\Gamma}$, and an $S_n$-module isomorphism $\splines{\Gamma} \to \splines{\Gamma}$ that leads to equalities $\lrep{\Gamma} = \lrep{\Gamma'}$ and $\rrep{\Gamma} = \rrep{\Gamma'}$. \par 
Throughout this subsection, let $\Gamma$ and $\Gamma'$ be graphs on $[n]$ and say that $\omega \colon \Gamma \to \Gamma'$ is a graph isomorphism. Then $\omega$ is also naturally an element of $S_n$, viewed as a bijection from $[n]$ to itself. Let $\omega$ denote both the graph isomorphism and associated permutation. \par 
Our first construction is an isomorphism between the corresponding labelled Cayley graphs. The following Lemma \ref{lem:isom_cayley_graph} states that $\cg_\Gamma$ and $\cg_{\Gamma'}$ are related as graphs by conjugation, and the associated labels are related via the action on ideals induced by the action on polynomials in Equation (\ref{eqn:sn_action_poly}).
\begin{lemma}\label{lem:isom_cayley_graph}
    Let $\omega \colon \Gamma \to \Gamma'$ be a graph isomorphism. Then $v \mapsto \omega v \omega^{-1}$ is a graph isomorphism  $\cg_\Gamma \to \cg_{\Gamma'}$. Additionally, if $\mathcal{L}$ is the label on $\cg_{\Gamma}$, $\mathcal{L'}$ the label on $\cg_{\Gamma'}$, and $(v_1,v_2) \in E(\cg_\Gamma)$, then  $\mathcal{L}'(\omega v_1 \omega^{-1},\omega v_2 \omega^{-1}) = \omega \mathcal{L}(v_1,v_2)$.
\end{lemma}
\begin{proof}
    Conjugation is a group automorphism of $S_n$. Say $(v_1,v_2) \in E(\Gamma)$, and in particular that $v_1^{-1}v_2 = (i,j) \in E(\Gamma)$. Then 
    \begin{align*}
        \left(\omega v_1 \omega^{-1}\right)^{-1}\left(\omega v_2\omega^{-1}\right) &= \omega v_1^{-1} v_2\omega^{-1} \\
        &= \omega(i,j)\omega^{-1} \\
        &= \left(\omega(i),\omega(j)\right) \in E(\Gamma').
    \end{align*}
    Thus conjugation by $\omega$ defines a graph isomorphism $\cg_\Gamma \to \cg_{\Gamma'}$. For the labels on $\cg_\Gamma$ and $\cg_{\Gamma'}$, the computation above also shows that if $v_1v_2^{-1} = (p,q)$, then $\left(\omega v_1 \omega^{-1}\right)\left(\omega v_2\omega^{-1}\right)^{-1} = (\omega(p),\omega(q))$. It follows that $\left(\omega v_1 \omega^{-1},\omega v_2\omega^{-1}\right) \in E(\cg_{\Gamma'})$ is labeled $\llangle t_{\omega(p)} - t_{\omega(q)}\rrangle = \omega\llangle t_p-t_q\rrangle$. The claim follows.
    \end{proof}
Define $\Omega \colon \splines{\Gamma} \to \splines{\Gamma'}$ by $\Omega(\bar{\rho})(v) \coloneqq \omega\bar{\rho}\left(\omega^{-1} v \omega\right)$. The following Proposition \ref{prop:isom_rings} proves $\Omega$ is a ring isomorphism, and is actually a consequence of Lemma \ref{lem:isom_cayley_graph} and a more general Proposition of Gilbert, Tymoczko, and Viel \cite[Prop 2.7]{tymoczko2016generalized_splines}. We include the proof here for completeness.
\begin{proposition}\label{prop:isom_rings}
    The map $\Omega \colon \splines{\Gamma} \to \splines{\Gamma'}$ is a ring isomorphism.
\end{proposition}
\begin{proof}
    Let $\bar{\rho} \in \splines{\Gamma}$. First, we show that $\Omega(\bar{\rho}) \in \splines{\Gamma'}$. Let $(v_1,v_2) \in E(\cg_{\Gamma'})$. By Lemma \ref{lem:isom_cayley_graph} there is an edge $(\omega^{-1}v_1\omega,\omega^{-1}v_2\omega) \in E(\cg_\Gamma)$, and so $\bar{\rho}(\omega^{-1}v_1\omega)-\bar{\rho}(\omega^{-1}v_2\omega) \in \mathcal{L}(\omega^{-1}v_1\omega,\omega^{-1}v_2\omega)$. Now we have 
    \begin{align*}
        \Omega(\bar{\rho})(v_1) - \Omega(\bar{\rho})(v_2) &= \omega\bar{\rho}\left(\omega^{-1} v_1 \omega\right) - \omega\bar{\rho}\left(\omega^{-1} v_2 \omega\right) \\
        &= \omega\left(\bar{\rho}\left(\omega^{-1} v_1 \omega\right) - \bar{\rho}\left(\omega^{-1} v_2 \omega\right)\right) \\
        &\in \omega \mathcal{L}\left(\omega^{-1} v_1 \omega,\omega^{-1} v_2 \omega\right) = \mathcal{L}'(v_1,v_2).
    \end{align*}
    Thus $\Omega(\bar{\rho}) \in \splines{\Gamma'}$. It is easy to verify that this map is a ring homomorphism, and the inverse from $\splines{\Gamma'}$ to $\splines{\Gamma}$ is constructed in the same manner with the map $\omega^{-1} \colon \Gamma' \to \Gamma$.
\end{proof}
The following lemma gives three instances in which $\Omega$ is also a module isomorphism between $\splines{\Gamma}$ and $\splines{\Gamma'}$.
\begin{lemma}\label{lem:mod_isom}
    The ring isomorphism $\Omega$ is a module isomorphism from $\splines{\Gamma}$ to $\splines{\Gamma'}$ with respect to the following actions:
    \begin{enumerate}
        \item the left $\poly$-action on $\splines{\Gamma}$ to the $\omega$-twisted left $\poly$-action on $\splines{\Gamma'}$, 
        \item the right $\poly$-action on $\splines{\Gamma}$ to the $\omega$-twisted right $\poly$-action on $\splines{\Gamma'}$, and 
        \item the dot action of $S_n$ on  $\splines{\Gamma}$ to the $\omega$-twisted dot action of $S_n$ on $\splines{\Gamma'}$.
    \end{enumerate}
\end{lemma}
\begin{proof}
    Both $\poly$-module statements follow from two straightforward computations,
    \[
        \Omega(\bar{t}_i)(v) =\bar{t}_{\omega(i)}(v)\; \text{ and }\;
        \Omega(\bar{x}_i)(v) = \bar{x}_{\omega(i)}(v).
    \]
    Say for the left action, if $f \in \poly$ and $\bar{\rho} \in \splines{\Gamma}$, then $\Omega(f(t_1,\ldots,t_n).\bar{\rho}) =  f(\bar{t}_{\omega(1)},\ldots,\bar{t}_{\omega(n)})\Omega(\bar{\rho})$, precisely the twisted action. The same holds for the right $\poly$-action to the $\omega$-twisted right $\poly$-action. It is easy to show that ring isomorphism $\Omega^{-1}$ is the inverse for $\Omega$ as a $\poly$-module morphism for both pairs of actions, and so $\Omega$ is a $\poly$-module isomorphism as in (1) and (2).\par 
    Given the dot action on $\splines{\Gamma}$, the induced action of $u \in S_n$ on $\bar{\rho} \in \splines{\Gamma'}$ is 
    \[ (u,\bar{\rho}) \coloneqq \Omega\left( u \cdot \Omega^{-1}(\bar{\rho})\right).
    \]
    To check that the action is compatible with multiplication of elements $v,u \in S_n$, compute
    \begin{align*}
    (v,(u,\bar{\rho})) &= \left(v,\Omega\left( u \cdot \Omega^{-1}(\bar{\rho})\right)\right) \\ 
    &= \Omega\left( v \cdot \Omega^{-1}\Omega\left( u \cdot \Omega^{-1}(\bar{\rho})\right)\right) \\ 
    &= \Omega\left( v \cdot \left( u \cdot \Omega^{-1}(\bar{\rho})\right)\right)\\
    &= \Omega\left( vu \cdot \Omega^{-1}(\bar{\rho})\right)\\ 
    &= (vu,\bar{\rho}).
    \end{align*}
    Now compute for $u,v \in S_n$ that
    \begin{align*}
        (u,\bar{\rho})(v) &=  \Omega\left( u \cdot \Omega^{-1}(\bar{\rho})\right) (v) \\ 
        &= \omega \left( u\cdot \Omega^{-1}(\bar{\rho})\right)(\omega^{-1}v\omega) \\ 
        &= \omega u \left(\Omega^{-1}(\bar{\rho})\right)(u^{-1}\omega^{-1}v\omega) \\ 
        &= \omega u \omega^{-1}\bar{\rho}(\omega u^{-1}\omega^{-1}v)\\ 
        &= \omega u \omega^{-1} \cdot \bar{\rho}(v).
    \end{align*}
    This is precisely the $\omega$-twisted dot action of $u$ on $\splines{\Gamma'}$. Again by computing with $\Omega^{-1}$ it follows that $\Omega$ is an $S_n$-module isomorphism.
\end{proof}
Note that any generating set for the left or right $\poly$-module structures on $\splines{\Gamma}$ must necessarily be generators for the $\omega$-twisted versions as well. As such, when searching for generators we may choose any graph isomorphic to $\Gamma$ for explicit calculations.
\begin{proposition}\label{prop:isom_reps}
    If $\Gamma$ and $\Gamma'$ are isomorphic, then $\lrep{\Gamma} =\lrep{\Gamma'}$ and $\rrep{\Gamma} = \rrep{\Gamma'}$.
\end{proposition}
\begin{proof}
    Let $\omega$ be an isomorphism from $\Gamma$ to $\Gamma'$. The $S_n$-isomorphism $\Omega$ from Lemma \ref{lem:mod_isom} (3) preserves the ideal $\llangle \bar{t}_1,\ldots,\bar{t}_n\rrangle$ from $\splines{\Gamma}$ to $\splines{\Gamma'}$. Thus we have a $S_n$-module isomorphism from $\lqot{\Gamma}$ to the $\omega$-twisted $\lqot{\Gamma'}$. Twisting by $\omega$ is an inner automorphism of $S_n$, so the $\omega$-twisted $\lqot{\Gamma'}$ is in turn isomorphic to the untwisted $\lqot{\Gamma'}$ as an $S_n$-representation. The exact same argument holds for $\rqot{\Gamma}$ and $\rqot{\Gamma'}$. Isomorphic representations have identical traces (i.e. equal characters), and the equalities follow.
\end{proof}
By Proposition \ref{prop:isom_reps} we may consider any graph isomorphic to $\Gamma$ when calculating $\lrep{\Gamma}$ and $\rrep{\Gamma}$.
\begin{corollary}
    The graded symmetric functions $\lrep{\Gamma}$ and $\rrep{\Gamma}$ are invariants of simple graphs. 
\end{corollary}

\section{Module Structure of \texorpdfstring{$\splines{\Gamma}$}{Splines-Gamma}}\label{sec:structure}
This section establishes some algebraic properties of $\splines{\Gamma}$ as a module over the polynomial ring $\poly$. It begins with a two results, one that establishes the size of a minimal homogeneous $\poly$-module generating set as an invariant of $\Gamma$, and a second that proves the module generated by constant and linear splines is a free module over $\poly$. This section continues with subsection \ref{ssec:structure:implied_conditions}, which establishes an algebraic relation that must be satisfied by elements of $\splines{\Gamma}$. This section ends with subsection \ref{ssec:structure_trees}, which gives an explicit and combinatorially meaningful generating set of $\splines{\Gamma}$ as a $\poly$-module when $\Gamma$ is a tree (proving Theorem \ref{intthm:coset_splines}).\par 
Before continuing, we will briefly describe what is already known in the geometric case. If $\Gamma$ is a Hessenberg graph, then 
\begin{itemize}
    \item $\splines{\Gamma}$ is a free $\poly$-module with a combinatorial formula for its rank-generating function \cite{DMPS1992hessenbergvarieties}, and furthermore
    \item $\splines{\Gamma}$ has explicit upper-triangular generators are achieved from a Bia{\l}ynicki-Birula decomposition of the corresponding variety \cite{DMPS1992hessenbergvarieties,ChoHongLee_Bases}.
\end{itemize}
In the geometric case, the rank-generating function is equivalent (substituting $q \mapsto q^2$) to the Poincar{\'e} polynomial of the corresponding variety. If $\Gamma$ is not in the geometric case, Then $\splines{\Gamma}$ is not always a free module. We now prove that the number of generators in each degree of a homogeneous generating set is still an invariant of $\Gamma$. We compute the minimal number of linear generators for $\splines{\Gamma}$ in Section \ref{sec:lin_dim}. \par 
A generating set $F$ of a finitely generated $\poly$-module $M$ is minimal if there exists a collection of polynomials $\{c_f \mid f \in F\} \subset \poly$ such that $\sum_{f \in F} c_f.f = 0$, then $c_f \notin \C \setminus\{0\}$ for all $f \in F$ (i.e. no $c_f$ is a unit). In other words, no proper subset of $F$ generates $M$. If $M$ is graded then a set $F$ is homogeneous if every element $f \in F$ is homogeneous.\par 
The following lemma is known, essentially as a corollary to the graded Nakayama lemma, and holds in greater generality (i.e. for other graded rings over a field). We include a proof for completeness.

\begin{lemma}\label{lem:graded_basis_size}
    Let $M$ be a finitely generated $\N$-graded module over $\poly$. Then every minimal homogeneous generating set has the same number of elements of each degree.
\end{lemma}
\begin{proof}
    Let $I = \llangle t_\bullet \rrangle$ be the irrelevant ideal. As $\faktor{\poly}{I} \cong \C$, the quotient $\faktor{M}{IM}$ is a graded $\C$-module. 
    In particular, $\faktor{M}{IM}$ is a graded $\C$-vector space of dimension $(d_0,\ldots,d_n)$, and we will prove that any homogeneous minimal generating set for $M$ projects to a graded basis in $\faktor{M}{IM}$. Let $F \coloneqq \left\{f_k^i \mid 0\leq i\leq N,\; k \in [K_i],\; \deg(f_k^i) = i\right\}$ be a minimal homogeneous generating set of $M$ with $K_i$ elements of degree $i$. It is easy to reason that $n=N$ since an element in $M$ of degree greater than $N$ is in $IF = IM$ (so $n \leq N$) and if $f_1^N \in IM$ then $F$ is not minimal (so $n \geq N$). In fact, if $f \in F$ and $f \in IM$ then $F$ is not minimal ($f$ would be in the $\poly$-span of lower degree elements of $F$), and moreover since we are assuming that $F$ is minimal we know that the image of $f \in F$ in $\faktor{M}{IM}$ is nonzero. We will show that $K_i = d_i$ for all $i=1,...,N$. Let $\pi \colon M \to \faktor{M}{IM}$ be the quotient map. \par 
    We know $\pi(F)$ is a homogeneous spanning set for the graded vector space $\faktor{M}{IM}$. Since $\faktor{M}{IM}$ is a graded vector space, we may prove linear independence degree-by-degree. Say for $c_1,\ldots,c_{K_i} \in \C$ that ${\ds \sum_{k=1}^{K_i} c_k\pi(f_k^i) = 0}$. We will show that $c_1 = \cdots = c_{K_i} = 0$. It follows that  ${\ds 
 \pi \left( \sum_{k=1}^{K_i} c_kf_k^i\right)} = 0$, and so  ${\ds 
  \sum_{k=1}^{K_i} c_kf_k^i \in IM}$. So there exists some finite set $P$ that indexes two subsets $\{r_p \mid p \in P\} \subset I$ and $\{h_p \mid P \in P\} \subset M$ such that ${\ds \sum_{k=1}^{K_i} c_kf_k^i = \sum_{p \in P}r_p.h_p}$. Since ${\ds \sum_{k=1}^{K_i} c_kf_k^i}$ is homogeneous of degree $i$, it suffices to consider only the $i$-th graded piece of each element $r_ph_p$.\par  
     Say $i=0$. Since each $r_p \in I$ has no degree $0$ component neither does $r_ph_p$, so ${\ds \sum_{k=1}^{K_0} c_kf_k^0 = 0}$. Since $F$ is a minimal generating set for $M$ it follows that $c_1=\cdots=c_{k_0} = 0$. \par 
     Now say $i>0$. Each $h_p$ is degree at most $i-1$, so $\sum_{ p \in P} r_p.h_p \in \poly\{f_q^j \mid 0\leq j < i,\; q \in [K_j]\}$. So \[\sum_{k=1}^{K_i} c_kf_k^i = \sum_{p \in P} r_ph_p = \sum_{\substack{0\leq j < i\\ 1\leq q \leq K_j}} c_{j,q}(t_\bullet) f_q^j.\]
    This is a relation in $M$ of elements from $F$ and thus cannot have any nonzero constant coefficients, so $c_1 = \cdots =c_{k_i} = 0$. Thus $\{\pi(f_k^i) \mid  k \in [K_i]\}$ is a basis of the $i$-th graded piece of the vector space $\faktor{M}{IM}$, and so $K_i = d_i$ is independent of the choice of $F$.
\end{proof}
The proof of Lemma \ref{lem:graded_basis_size} also ensures that a minimal graded generating set of $\splines{\Gamma}$ with respect to either the left or right module structure projects to a basis of $\lqot{\Gamma}$ or $\rqot{\Gamma}$ respectively. \par 
Lemma \ref{lem:linear_freeness} below shows that the first graded piece of the $\poly$-module is free. Note Lemma \ref{lem:linear_freeness} is independent of the polynomial action chosen, e.g. left, right, and twisted alternatives.
\begin{lemma}\label{lem:linear_freeness}
    The $\poly$-submodule $\splines{\Gamma}^{\leq 1}$ generated by the constant and linear splines on $\cg_\Gamma$ is a free module.
\end{lemma}
\begin{proof}
    Let $(e,w_2,w_3,\ldots,w_{n!})$ be a linear order on $S_n$ where $\ell_\Gamma(v) < \ell_\Gamma(w)$ implies that $v<w$. Since $\Gamma$ is connected, if $w \neq e$, there exists $(i,j) \in E(\Gamma)$ such that $w(i,j) < w$. \par 
    Let $F\coloneqq \{\bar{\mathbb{1}},\bar{f}_1,\ldots,\bar{f}_k\}$ be a minimal generating set of $\splines{\Gamma}^{\leq 1}$ where each $\bar{f}_1,\ldots,\bar{f}_k$ is a linear spline. Then $\{\bar{\mathbb{1}}, \bar{f}_1-\bar{f}_1(e)\bar{\mathbb{1}},\ldots,\bar{f}_n-\bar{f}_n(e)\bar{\mathbb{1}}\}$ is a homogeneous generating set of the same size and is therefore minimal by Lemma \ref{lem:graded_basis_size}. This new generating set has the property that the $\bar{\mathbb{1}}$ is the unique spline whose minimal element is $e$, so assume that $\bar{f}(e) = 0$ for all $f \in  F\setminus\{\bar{\mathbb{1}}\}$. \par 
     Let $F_v \coloneqq \{f \in F \mid \min(\supp(f)) = v\}$. So $F_e = \{\bar{\mathbb{1}}\}$. We iteratively construct a minimal generating set such that $\abs{F_v} \in \{0,1\}$ for all $v \in S_n$. Say that $\abs{F_v} \in \{0,1\}$ for all $v < w$, and $\abs{F_w} \geq 2$. Let $F_w = \{\bar{g}_1,\ldots,\bar{g}_r\}$. Since the linear order on $S_n$ is an extension of $\Gamma$-length, there exists $(a,b) \in E(\Gamma)$ such that $w(a,b) < w$, and so $\bar{g}(w(a,b)) = 0$ for all $\bar{g} \in F_w$. Thus there exist $c_1,\ldots,c_r \in \C^*$ such that $\bar{g}_i(w) = c_i(t_{w(a)}-w_{w(b)})$. For $j = 2,...,r$, the spline $\bar{g}_j - \frac{c_j}{c_1}\bar{g}_1$ is supported strictly above $w$. Let \[F' = \left(F \setminus F_w\right) \cup \left\{\bar{g}_1, \bar{g}_2 - \frac{c_2}{c_1}\bar{g}_1,\ldots,\bar{g}_r - \frac{c_r}{c_1}\bar{g}_1 \right\} \]
    is still a minimal generating set, and $\abs{F'_v} \in \{0,1\}$ for all $v \leq w$. Iterate this process letting $F = F'$. Eventually $\abs{F_v} \in \{0,1\}$ for all $v \in S_n$. In particular, this $F$ is upper triangular with respect to our total order (the minimal element in the support of each spline is unique to that spline), and so $F$ generates a free $\poly$-module.
    \end{proof}
We note that this submodule is precisely where this paper proves $h$-positivity in Theorem \ref{thm:linear_reps} and Corollary \ref{cor:nonlabel_reps}.
\subsection{Implied Conditions on Splines}\label{ssec:structure:implied_conditions}
This subsection gives algebraic conditions that an element $\bar{\rho} \in \splines{\Gamma}$ must satisfy that are not explicitly in the definition. Specifically, given $w,v \in S_n$, we want to infer conditions on $\bar{\rho}(w) - \bar{\rho}(v)$ when $(w,v)$ is not necessarily an edge in $\cg_\Gamma$. Let $w,v \in S_n$ and $(w=v_0,v_1,\ldots,v_m=v)$ be a path from $w$ to $v$ in $\cg_\Gamma$. Say for each edge $(v_{k-1},v_k)$ that $v_kv_{k-1}^{-1} = (i_k,j_k)$, so that $\cl(v_{k-1},v_k) = \llangle t_{i_k} - t_{j_k}\rrangle$ for each $k=1,\ldots ,m$. Then
    \begin{equation}\label{eqn:path_to_ideal}
        \bar{\rho}(w)-\bar{\rho}(v) = \sum_{k=1}^{m} \left(\bar{\rho}(v_{k-1})-\bar{\rho}(v_k)\right) \in \llangle t_{i_k} - t_{j_k} \mid k \in [m] \rrangle.
    \end{equation}
    Define 
    \begin{equation}\label{eqn:distance_ideals}
        I_B^w := \llangle t_{w(i)}-t_{w(j)} \mid (i,j) \in B \subseteq E(T) \rrangle.
    \end{equation} 
Lemma \ref{lem:diff_ideal} below is particularly useful when $\bar{\rho} \in \splines{\Gamma}$ satisfies $\bar{\rho}(v) = 0$ for some $v \in S_n$. Recall that we identify $B \subset E(\Gamma)$ with a subset of tranpositions, and write $w\langle B \rangle$ for the left coset at $w$ of the reflection subgroup generated by the transpositions in $B$.
\begin{lemma}\label{lem:diff_ideal}
    Let $T$ be a spanning tree of $\Gamma$. Let $v \in w\langle B \rangle$ where $B \subseteq E(T)$. If $\bar{\rho} \in \splines{\Gamma}$, then $\bar{\rho}(w)-\bar{\rho}(v) \in I_B^w$.
\end{lemma}
\begin{proof}
    Let $w^{-1}v = b_1\cdots b_m$ where $b_1,\ldots,b_m \in B$. Let $(w = v_0, v_1,\ldots, v_m=v)$ be the path from $w$ to $v$ where $v_{k}^{-1}v_{k-1} = b_k \in B$ for all $k \in [m]$. Say  that $v_{k}v_{k-1}^{-1} = (i_k,j_k)$, so that $\cl(v_{k-1},v_{k})= t_{i_k} - t_{j_k}$. By Equation (\ref{eqn:path_to_ideal}), 
    \[    \bar{\rho}(w)-\bar{\rho}(v) \in \llangle t_{i_k} - t_{j_k} \mid k \in [m] \rrangle. \]
    For each $k \in [m]$, since $v_kv_{k-1}^{-1}= (i_k,j_k)$, we have that $b_k = v_k^{-1}v_{k-1} = \left(v_k^{-1}(i_k),v_k^{-1}(j_k)\right)$. Each edge $b_k \in B$, so the integers $v_k^{-1}(i_k) = (wb_1\cdots b_k)^{-1}(i_k)$ and $v_{k}^{-1}(j_k) = (wb_1\cdots b_k)^{-1}(j_k)$ must be in the same connected component of $([n],B)$.  \par 
    Since $(wb_1\cdots b_k)^{-1} = (b_k\cdots b_1)w^{-1}$, it follows that $w^{-1}(i_k)$ and $w^{-1}(j_k)$ are vertices in the same connected component of $([n],B)$ for all $k \in [m]$. If $(q_0,...,q_\ell)$ is a path in $([n],B)$ from $q_0 =  w^{-1}(i_k)$ to $q_{\ell} =  = w^{-1}(j_k)$, then $t_{q_0} - t_{q_\ell} = \sum_{r=1}^\ell t_{q_{r-1}} - t_{q_r}$ and thus $t_{w^{-1}(i_k)}-t_{w^{-1}(j_k)} \in I_B^e$. It follows that $t_{i_k} - t_{j_k} \in I_B^w$ for all $k \in [m]$, and so $\bar{\rho}(w)-\bar{\rho}(v)\in I_B^w$.
\end{proof}
A monomial ideal in $\poly$ is an ideal $I$ generated by monomials. Monomial ideals are particularly nice when computing intersections; if $I_1 = \langle m_1,\ldots,m_k\rangle$ and $I_2 = \langle n_1,\ldots,n_\ell \rangle$ are both monomial ideals, then $I_1 \cap I_2 = \langle \mathrm{lcm}(m_i,n_j) \mid i \in [k],\; j\in [\ell]\rangle$.\par 
Let $T$ be a spanning tree of $\Gamma$, where $E(T) = \{(a_1,b_1),\ldots,(a_{n-1},b_{n-1})\}$. Ideals of the form $\langle t_{a_i}-t_{b_i} \mid (a_i,b_i) \in B \subset E(T) \rangle$ can be considered monomial ideals, via the graded automorphism 
\[
\C[t_1,\ldots,t_n] \cong \C[t_{a_1}-t_{b_1},\ldots,t_{a_{n-1}}-t_{b_{n-1}},t_n]
\]
defined by $t_i \mapsto \begin{cases}t_{a_i}-t_{b_i} &\text{if } i \in [n-1] \\ t_n & \text{if }i=n \end{cases}$. Since $t_i \mapsto t_{w(i)}$ is also a graded automorphism of $\C[t_\bullet]$, the ideals $I_B^w$ 
from Equation \ref{eqn:distance_ideals} can also be considered as monomial ideals (taking care to \textbf{fix} $T$ and $w \in S_n$). We will fix $T$ and $w$ then treat ideals of the form $I_B^w$ as a monomial ideals to compute intersections in the proofs of Theorem \ref{thm:coset_splines_spann} and Lemma \ref{lem:kconn_better} in the following two sections.
\subsection{Coset Splines and Trees}\label{ssec:structure_trees}
This subsection establishes a set of splines called \emph{coset splines} (Def. \ref{def:coset_spline}) that generate $\splines{\Gamma}$ as a module over the polynomial ring when $\Gamma$ is a tree. This subsection also identifies a subset of those coset splines that generate $\splines{\Gamma}$ as a ring when $\Gamma$ is a tree.
\begin{definition}\label{def:coset_spline}
    Let $\Gamma$ be a tree, $E \coloneqq E(\Gamma)$ and $B \subseteq E$. The \emph{coset spline at the identity} $\bar{f}_e^B \colon S_n \to \poly$ is
    \[
\bar{f}_e^B(w) := \begin{cases}
    {\ds \prod_{(i,j) \in E \setminus B} \left(t_{w(i)} - t_{w(j)}\right)} & w \in \langle B \rangle \\
    0 & \text{otherwise}
\end{cases}
    \]
    The \emph{coset spline at $w$} is $\bar{f_w^B} := w \cdot \bar{f}_e^B$. We adopt the conventions that a product over the empty set $\emptyset$ is $1$ (so $f_w^\emptyset = \bar{\mathbb{1}}$) and that the subgroup generated by the empty set is the identity (so $\langle \emptyset \rangle = \{e\}$).
\end{definition}
\begin{example}\label{ex:coset_splines}
Again consider $\Gamma = \left([3],\{(1,2),(2,3)\}\right)$. Drawn below are three examples of coset splines on $\cg_\Gamma$.
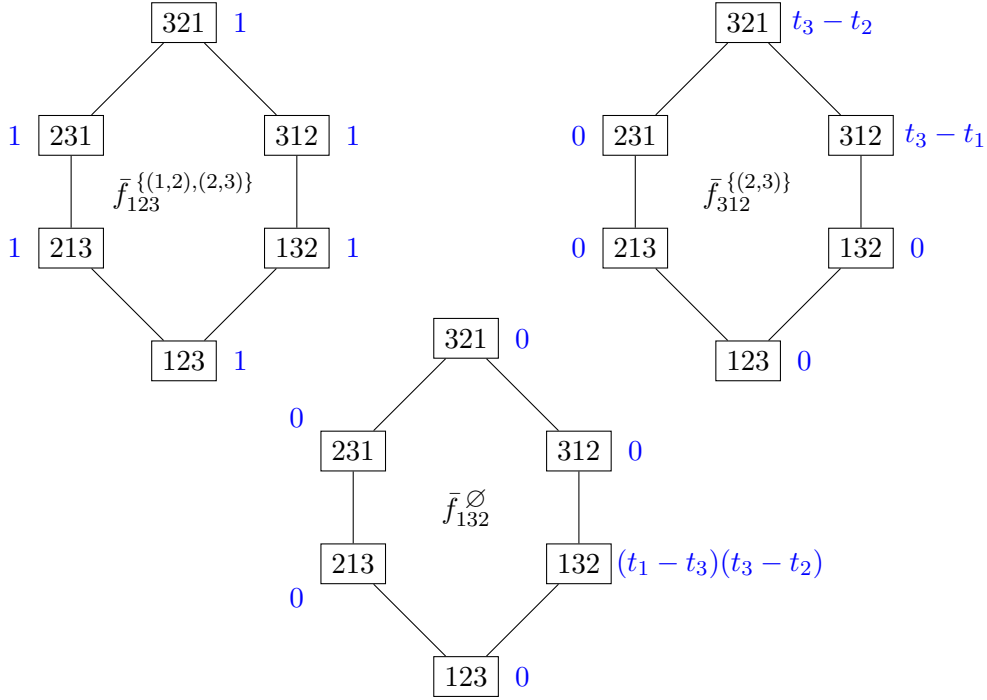
\begin{center}
    \begin{tikzpicture}[scale=1.5]
\tikzstyle{every node}=[shape=rectangle, inner sep=4pt];
\begin{scope}[xshift = 0]
\draw  (0,0) node[draw] (v0) {$123$};
\draw  (1,1) node[draw]  (v1) {$132$};
\draw  (-1,1) node[draw] (v2) {$213$};
\draw  (-1,2) node[draw] (v3) {$231$};
\draw  (1,2) node[draw] (v4) {$312$};
\draw (0,3)  node[draw]  (v5) {$321$};
\draw (0,1.5) node (name) {$\bar{f}_{123}^{\;\{(1,2),(2,3)\}}$};

\draw  (0.5,0) node (v00) {\textcolor{blue}{$1$}};
\draw  (1.5,1) node  (v10) {\textcolor{blue}{$1$}};
\draw  (-1.5,1) node (v20) {\textcolor{blue}{$1$}};
\draw  (-1.5,2) node (v30) {\textcolor{blue}{$1$}};
\draw  (1.5,2) node (v40) {\textcolor{blue}{$1$}};
\draw (0.5,3)  node  (v50) {\textcolor{blue}{$1$}};

\draw (v0)--(v1);
\draw (v0)--(v2);
\draw (v1)--(v4);
\draw (v2)--(v3);
\draw (v3)--(v5);
\draw (v4)--(v5);
\end{scope}

\begin{scope}[xshift = 5cm]
\draw  (0,0) node[draw] (v0) {$123$};
\draw  (1,1) node[draw]  (v1) {$132$};
\draw  (-1,1) node[draw] (v2) {$213$};
\draw  (-1,2) node[draw] (v3) {$231$};
\draw  (1,2) node[draw] (v4) {$312$};
\draw (0,3)  node[draw]  (v5) {$321$};
\draw (0,1.5) node (name) {$\bar{f}_{312}^{\;\{(2,3)\}} $};

\draw  (0.5,0) node (v00) {\textcolor{blue}{$0$}};
\draw  (1.5,1) node  (v10) {\textcolor{blue}{$0$}};
\draw  (-1.5,1) node (v20) {\textcolor{blue}{$0$}};
\draw  (-1.5,2) node (v30) {\textcolor{blue}{$0$}};
\draw  (1.75,2) node (v40) {\textcolor{blue}{$t_3-t_1$}};
\draw (0.75,3)  node  (v50) {\textcolor{blue}{$t_3-t_2$}};

\draw (v0)--(v1);
\draw (v0)--(v2);
\draw (v1)--(v4);
\draw (v2)--(v3);
\draw (v3)--(v5);
\draw (v4)--(v5);
\end{scope}

\begin{scope}[xshift = 2.5cm, yshift=-2.8cm]
\draw  (0,0) node[draw] (v0) {$123$};
\draw  (1,1) node[draw]  (v1) {$132$};
\draw  (-1,1) node[draw] (v2) {$213$};
\draw  (-1,2) node[draw] (v3) {$231$};
\draw  (1,2) node[draw] (v4) {$312$};
\draw (0,3)  node[draw]  (v5) {$321$};
\draw (0,1.5) node (name) {$\bar{f}_{132}^{\;\emptyset}$};

\draw  (0.5,0) node (v00) {\textcolor{blue}{$0$}};
\draw  (2.25,1) node  (v10) {\textcolor{blue}{$(t_1-t_3)(t_3-t_2)$}};
\draw  (-1.5,0.7) node (v20) {\textcolor{blue}{$0$}};
\draw  (-1.5,2.3) node (v30) {\textcolor{blue}{$0$}};
\draw  (1.5,2) node (v40) {\textcolor{blue}{$0$}};
\draw (0.5,3)  node  (v50) {\textcolor{blue}{$0$}};

\draw (v0)--(v1);
\draw (v0)--(v2);
\draw (v1)--(v4);
\draw (v2)--(v3);
\draw (v3)--(v5);
\draw (v4)--(v5);
\end{scope}
\end{tikzpicture}
\end{center}
\end{example}
\begin{lemma}
    When $\Gamma$ is a tree, coset splines are elements of $\splines{\Gamma}$. Additionally, if $w,v \in S_n$ are in the same coset of $\langle B \rangle$, then $\bar{f}_w^B = \bar{f}_v^B$.
\end{lemma}
\begin{proof}
    It suffices to show $\bar{f}_e^B$ is a spline. Let $w \in \langle B \rangle$ and $v \in S_n$, where $vw^{-1} = (i,j) \in E$. \par 
    If $(i,j) \in E \setminus B$, then $v \notin \langle B \rangle$ and so $\bar{f}_e^B(v) = 0$. Thus $\bar{f}_e^B(w) - \bar{f}_w^B(v) = t_{w(i)}-t_{w(j)} \in \cl(w,v)$ as desired.  \par 
    If $(i,j) \in B$, then 
    \begin{align*}
        \bar{f}_e^B(w) - \bar{f}_e^B(v) &=  \prod_{(r_1,s_1) \in E \setminus B} \left(t_{w(r_1)} - t_{w(s_1)}\right) -  \prod_{(r_2,s_2) \in E \setminus B} \left(t_{w(i,j)(r_2)} - t_{w(i,j)(s_2)}\right) \\
        &= w\left(\prod_{(r_1,s_1) \in E \setminus B} \left(t_{r_1} - t_{s_1}\right) -  (i,j)\left(\prod_{(r_2,s_2) \in E \setminus B} \left(t_{r_2} - t_{s_2}\right)\right)\right)\\
        &= w\left(\sum_{0\leq p,q} g_{pq}(t_\bullet) t_i^p t_j^q  -(i,j)\left(\sum_{0\leq r,s} g_{rs}(t_\bullet) t_i^r t_j^s\right)\right)\\
        &= w\left(\sum_{0\leq p,q} g_{pq}(t_\bullet) (t_i^p t_j^q-  t_i^q t_j^p)\right)\\
    \end{align*}
As $t_i^p t_j^q-  t_i^q t_j^p \in \llangle t_i-t_j\rrangle$, it follows that $\bar{f}_e^B(w) - \bar{f}_e^B(v) \in \llangle t_{w(i)} - t_{w(j)} \rrangle =  \cl(w,v)$. Thus $\bar{f}_e^B$ is a spline. \par 
Now we prove that coset splines are uniquely determined by the coset. For all $u \in \langle B \rangle$ we have
\begin{align*}
    u \cdot \bar{f}_e^B(w) &=  \begin{cases}{\ds 
     u\left(\prod_{(i,j) \in E - B} \left(t_{u^{-1}w(i)} - t_{u^{-1}w(j)}\right)\right)} & u^{-1}w \in \langle B \rangle \\
    0 & \text{otherwise}
    \end{cases}\\
    &= \begin{cases}
     {\ds \prod_{(i,j) \in E - B} \left(t_{w(i)} - t_{w(j)}\right)} & w \in \langle B \rangle \\
    0 & \text{otherwise} \\
    \end{cases}\\
    &= \bar{f}_e^B(w).
\end{align*} 
If $w\langle B \rangle = v\langle B \rangle$ then $w=vu$, for some $u \in \langle B \rangle$ and so
\[
\bar{f}_w^B = w\cdot \bar{f}_e^B = (vu)\cdot \bar{f}_e^B = v\cdot (u\cdot \bar{f}_e^B) = v\cdot \bar{f}_e^B = \bar{f}_v^B.
\]
\end{proof}
Note that the following Theorem \ref{thm:coset_splines_spann} is independent of the left or right module structure.
\begin{theorem}\label{thm:coset_splines_spann}
    Let $\Gamma $ be a tree. The set of coset splines $\{\bar{f}_w^B \mid w \in S_n,\; B \subseteq E(\Gamma)\}$ is a $\poly$-generating set of $\splines{\Gamma}$.
\end{theorem}
\begin{proof}
    Let $\bar{\rho}\in \splines{\Gamma}$, we will show that $\bar{\rho} \in \poly\{\bar{f}_B^w \mid w \in S_n, B \subseteq E(\Gamma)\}$ by induction on containment of the support $\supp(\bar{\rho})$. If $\bar{\rho} \equiv 0$, this is clearly in the span of the coset splines and the base case $\supp(\bar{\rho}) = \emptyset$ is done. Otherwise $\supp(\bar{\rho}) \neq \emptyset$, and we assume all splines $\bar{\kappa}$ where $\mathrm{supp}(\bar{\kappa}) \subsetneq \mathrm{supp}(\bar{\rho})$ are in $\poly\{\bar{f}_B^w \mid w \in S_n, B \subseteq E(\Gamma)\}$. Replacing $\bar{\rho}$ by $ \bar{\rho} - \bar{\rho}(e)\bar{\mathbb{1}}$ if necessary, we assume $\bar{\rho}(e) = 0$. This also handles the case where $\supp(\bar{\rho}) = S_n$.\par 
    Fix $w \in S_n$ such that $\bar{\rho}(w) \neq 0$ and $w$ is adjacent in $\cg_\Gamma$ to some $w' \in S_n$ where $\bar{\rho}(w') = 0$. Define
    \[
        \cb_w \coloneqq \{B \mid B \subset E(\Gamma), \;\; \exists v \in w\langle B\rangle  \text{ such that } \bar{\rho}(v) = 0\}.
    \]
    Since $\bar{\rho}(w') = 0$, this set is nonempty. Each element in $\cb_w$ is a generating set for a reflection subgroup whose left coset at $w$ contains an element not in $\supp(\bar{\rho})$. Note if $B \subset B'$ and $B \in \cb_w$, then $B' \in \cb_w$.\par 
    By Lemma \ref{lem:diff_ideal}, 
    \begin{equation}
        \bar{\rho}(w) \in \bigcap_{B \in \cb_w } I_B^w = \bigcap_{B \in \cb_w} \llangle t_{w(i)}-t_{w(j)} \mid (i,j) \in B \subset E(\Gamma) \rrangle \eqqcolon \ci_\rho^w.
    \end{equation}
    Following the logic of Subsection \ref{ssec:structure:implied_conditions} (i.e. treating $\{t_{w(i)} - t_{w(j)} \mid (i,j) \in E(\Gamma)\}$ as variables), $\ci_\rho^w$ is  a monomial ideal generated by the monomials that are contained within \emph{every} element of the intersection.\par 
    A monomial $\mathfrak{m} = {\ds \prod_{(i,j) \in E} \left(t_{w(i)}-t_{w(j)}\right)^{\alpha_{ij}}}$ is contained within the ideal $\ci_\rho^w$ if and only if for every $B \in \cb_w$, there is at least one $(i,j) \in B$ such that $\alpha_{ij} > 0$. For generators of $\ci_\rho^w$, it suffices to consider only those monomials such that $\alpha_{ij} \in \{0,1\}$ for all $(i,j) \in E(\Gamma)$.
    Since $\alpha_{ij} \in \{0,1\}$, the monomials that generate $\ci_\rho^w$ are a subset of $\{\bar{f}_w^B(w) \mid B \subseteq E(\Gamma)\}$. In particular, we have the equality $\llangle \bar{f}_w^D(w) \mid D \subseteq E(\Gamma),\;\bar{f}_w^D(w) \in \ci_\rho^w \rrangle = \ci_\rho^w$. \par 
    Consider the coset splines $\{\bar{f}_w^D \mid \bar{f}_w^D(w) \in \ci_\rho^w\}$. By definition, for any $D \subset E(\Gamma)$,
    \[
        \bar{f}_w^{D}(w) = \prod_{(i,j) \in E(\Gamma)\setminus D} t_{w(i)} - t_{w(j)}.
    \]
    Now $\bar{f}_w^{D}(w)\in \ci_\rho^w$ if and only if $\left(E(\Gamma)\setminus D\right) \cap B \neq \emptyset$ for all $B \in \cb_w$. Thus $\bar{f}_w^{D}(w)\in \ci_\rho^w$ if and only if $B \not\subset D$ for all $B \in \cb_w$. Since $\cb_w$ is closed under supersets, $\bar{f}_w^{D}(w)\in \ci_\rho^w$ if and only if $D \notin \cb_w$. Thus 
    \begin{equation*}\label{eqn:cos_splines_thm}
        \bar{\rho}(w) \in \ci_\rho^w = \llangle \bar{f}_w^D(w) \mid \text{ for all } v \in w\langle D \rangle, \;  \bar{\rho}(v) \neq 0 \rrangle.
    \end{equation*}
    Let $\bar{f} \in \poly\{\bar{f}_w^D \mid \text{ for all } v \in w\langle D \rangle ,\;  \bar{\rho}(v) \neq 0 \}$ such that $\bar{\rho}(w) = \bar{f}(w)$ (a different $\bar{f}$ may be chosen for the left and right module structure, but either way such a $\bar{f}$ exists since it is only required to agree with $\bar{\rho}$ at $w$). Since $\mathrm{supp}(\bar{f}) \subseteq \mathrm{supp}(\bar{\rho})$ and $\bar{f}(w) = \bar{\rho}(w) \neq 0$, it follows that $\mathrm{supp}(\bar{\rho}) \supsetneq \mathrm{supp}(\bar{\rho}-\bar{f})$. Thus $\bar{\rho}-\bar{f} \in \poly\{\bar{f}_B^w \mid w \in S_n, B \subseteq E(\Gamma)\}$. Since $\bar{f}$ is also a sum of coset splines, $\bar{\rho} \in \poly\{\bar{f}_B^w \mid w \in S_n, B \subseteq E(\Gamma)\}$. 
    \end{proof}
    The collection of all coset splines is not a minimal generating set. One might significantly decrease the size of this set by fixing the linear order on $S_n$ in the proof of Lemma \ref{lem:linear_freeness}, and only considering the largest (by support) coset splines supported ``above" a permutation. There is no guarantee that these generators are minimal for all degrees, but it is easy to reason that this collection is minimal for the module $\splines{\Gamma}^{\leq 2}$ generated by the constant, linear, and quadratic splines. \par 
    We also achieve a generating set for $\splines{\Gamma}$ as a ring in Corollary \ref{cor:ring_gen} below.    \begin{corollary}\label{cor:ring_gen}
        Let $\Gamma$ be a tree. The constant and linear coset splines along with either $\{\bar{t}_i \mid i \in [n]\}$ or $\{\bar{x}_i \mid i \in [n]\}$ generate $\splines{\Gamma}$ as a ring.
    \end{corollary}
    \begin{proof}
        It follows immediately from the definition that 
        \[
        f_w^B = \prod_{s \in E(\Gamma) \setminus B}f_w^{E(\Gamma) \setminus\{s\}}.
        \]
        So every coset spline except $\bar{\mathbb{1}}$ is a product of linear coset splines, which generate $\splines{\Gamma}$ together with either $\{\bar{t}_i \mid i \in [n]\}$ or $\{\bar{x}_i \mid i \in [n]\}$ by Theorem \ref{thm:coset_splines_spann}.
    \end{proof}
    We can leverage Theorem \ref{thm:coset_splines_spann} to compute $\splines{\Gamma}$ for all graphs $\Gamma$.  Any graph $\Gamma$ can be expressed as the union of spanning trees $\Gamma = T_1 \cup \cdots \cup T_k$. Lemma \ref{lem:unions} says that  ${\ds \splines{\Gamma} = \bigcap_{i=1}^k \splines{T_i}}$, and Theorem \ref{thm:coset_splines_spann} gives explicit generators for each $\splines{T_i}$. This is most useful in computer calculations, where the task of constructing modules from generators and intersecting them can be completed by a computer algebra system.

\section{Connectedness and \texorpdfstring{$\splines{\Gamma}^k$}{Splines-Gamma-k}}\label{sec:connectedness}
This section proves an equivalence between between the $k$-connectivity of $\Gamma$ and which graded pieces of the representation $\lrep{\Gamma}$ are trivial.  \par 
Lemma \ref{lem:kconn_better} below infers new conditions on $\splines{\Gamma}$ from collections of vertex-disjoint paths in $\Gamma$.
\begin{lemma}\label{lem:kconn_better}
    Say that there exist $k$ vertex-disjoint paths from $i$ to $j$ in $\Gamma$. Let $\Gamma' = ([n],E(\Gamma) \cup \{(i,j)\})$. Then $\splines{\Gamma}^{k-1} = \splines{\Gamma'}^{k-1}$.
\end{lemma}
\begin{proof}
    We will show both directions of containment. Clearly $\splines{\Gamma}^{k-1} \supseteq \splines{\Gamma'}^{k-1}$. \par 
    Say that $\bar{\rho} \in \splines{\Gamma}^{k-1}$, and let $w,v \in S_n$ such that $w^{-1}v = (i,j)$. Let $p_r = (i,s_{r,1},\ldots,s_{r,\ell_r},j)$ for $r=1,\ldots,k$ be the $k$ vertex-disjoint paths from $i$ to $j$ in $\Gamma$. By Lemma \ref{lem:diff_ideal}, 
    \begin{align*}
    \bar{\rho}(w)- \bar{\rho}(v)&\in \llangle t_{w(i)}-t_{w(s_{r,1})},t_{w(s_{r,1})}-t_{w(s_{r,2})},\ldots,t_{w(s_{r,\ell_r})}-t_{w(j)}\rrangle \\
    &= \llangle t_{w(i)}-t_{w(j)}, t_{w(i)}-t_{w(s_{r,1})},\ldots,t_{w(s_{r,\ell_r-1})}-t_{w(s_{r,\ell_r})}\rrangle
    \end{align*}
    for all $r=1,\ldots,k$. Since the paths $p_1,\ldots,p_k$ are vertex independent, the set of edges
    \[
        A=\{(i,j)\} \cup \bigcup_{r=1}^k \{(i,s_{r,1}),(s_{r,1},s_{r,2}),\ldots,(s_{r,\ell_r-1},s_{r,\ell_r})\}
    \]
    contain no cycles, and thus forms a tree. In particular, we may consider $\{t_{a}-t_{b} \mid (a,b) \in A\}$ as monomials in $\C[t_\bullet]$. Let $s_{r,0} := i$ when it is convenient for indexing. It remains to compute
    \begin{align*}
        \bar{\rho}(w) &\in \bigcap_{r=1}^k \llangle t_{w(i)}-t_{w(j)}, t_{w(s_{r,0})}-t_{w(s_{r,1})},\ldots,t_{w(s_{r,\ell_r-1})}-t_{w(s_{r,\ell_r})}\rrangle \\
        &= \llangle t_{w(i)}-t_{w(j)},\left. \prod_{r=1}^k t_{w(s_{r,m_r-1})}-t_{w(s_{r,m_r})}\right| 0 < m_r \leq \ell_r  \rrangle.
    \end{align*}
    Each generator of this ideal is a homogeneous polynomial, one of degree $1$ and all others of degree $k$. Since $\bar{\rho}(w)-\bar{\rho}(v)$ is degree $k-1$, it follows that $\bar{\rho}(w)-\bar{\rho}(v) \in \langle t_{w(i)} - t_{w(j)}\rangle$.\par 
    Since $v,w$ were arbitrary such that $w^{-1}v = (i,j)$, we know that $\bar{\rho} \in \splines{\Gamma'}^{k-1}$. Since $\bar{\rho}$ was arbitrary, $\splines{\Gamma}^{k-1} \subseteq \splines{\Gamma'}^{k-1}$.
\end{proof}
\begin{theorem}\label{thm:kconn}
    Let $\Gamma = ([n],E)$ be a connected graph on $n$ vertices. The following are equivalent:
    \begin{enumerate}
        \item $\Gamma$ is $k$-connected.
        \item $\splines{\Gamma}^d = \splines{K_n}^d$ for all $d<k$, where $K_n$ is the complete graph.
        \item $\lrep{\Gamma}_{d}$ is trivial for all $d<k$.
    \end{enumerate} 
\end{theorem}
\begin{proof}
    (1) $\Rightarrow$ (2). If $\Gamma$ is $k$-connected, by Menger's theorem every $(i,j) \in [n]\times [n]$ has $k$ vertex-disjoint paths connecting them in $\Gamma$. By Lemma \ref{lem:kconn_better}, $\splines{\Gamma}^{k-1} = \splines{K_n}^{k-1}$. \par 
    (2) $\Rightarrow$ (3). The ring $\splines{K_n}$ corresponds to the equivariant cohomology of the full flag variety, where the dot action is known to be trivial \cite{tymoczko2008permutation}.\par
    (3) $\Rightarrow$ (1). Assume that $\Gamma$ is \emph{not} $k$-connected. Let $d$ be the integer such that $\Gamma$ is $d$-connected but not $(d+1)$-connected (so $0 < d < k$). We will show that $\lrep{\Gamma}_d$ is not trivial. Then $\Gamma$ has a cut-set of size $d$, and so $\Gamma$ is (isomorphic to) a sub-graph of the graph $H = ([n],\{(i,j) \mid 1 \leq i<j<\ell+d\text{ or } \ell < i < j \leq n\})$ drawn below.
    
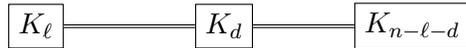
\begin{center}
        \begin{tikzpicture}[scale=2.5]
            \draw (-1,0) node[draw, rectangle] (L1) {$K_{\ell}$};
            \draw (0,0) node[draw,rectangle] (C) {$K_{d}$};
            \draw (1,0) node[draw,rectangle] (L2) {$K_{n-\ell-d}$};
            \draw[double] (L1)--(C);
            \draw[double] (C)--(L2);
        \end{tikzpicture}
    \end{center}
    (the center $K_{d}$ is the cut set). The graph $H$ is also $d$-connected. By (1) $\Rightarrow$ (2), the graded pieces $\splines{\Gamma}^p = \splines{H}^p = \splines{K_n}^p$ for all $0 \leq p < d$. Since $\Gamma$ is an edge-subgraph of $H$, it follows directly from the definitions that $\splines{\Gamma} \supseteq \splines{H}$. \par 
    If $I = \llangle t_1,...,t_n \rrangle$, then for any graded $\poly$-module $M \coloneqq \oplus_{p \geq 0} M^p$ the following equality is by definition 
    \[ \left(\faktor{M}{IM}\right)^p = \faktor{M^p}{IM \cap M^p}.\]
    Since multiplication by elements in $I$ must increase degree, the $d$-th degree component of $I\splines{\Gamma}$ and the $d$-th degree component of $I\splines{H}$ are equal. In particular, \[I\splines{\Gamma} \cap \splines{\Gamma}^d = I\left(\splines{\Gamma}^{\leq d-1}\right) \cap \splines{\Gamma}^d = I\left(\splines{H}^{\leq d-1}\right) \cap \splines{\Gamma}^d = I\splines{H} \cap \splines{\Gamma}^d. \] 
    It follows that for the quotients 
    \[
    \left( \faktor{\splines{\Gamma}}{I\splines{\Gamma}}\right)^d = \faktor{\splines{\Gamma}^d}{I\splines{\Gamma}\cap \splines{\Gamma}^d} = \faktor{\splines{\Gamma}^d}{I\splines{H}\cap \splines{\Gamma}^d} = \left( \faktor{\splines{\Gamma}}{I\splines{H}}\right)^d,
    \]
    and we get containment in the vector spaces
    \[
    (\lqot{\Gamma})_d = \left( \faktor{\splines{\Gamma}}{I\splines{\Gamma}}\right)^d = \left( \faktor{\splines{\Gamma}}{I\splines{H}}\right)^d \supseteq \left( \faktor{\splines{H}}{I\splines{H}}\right)^d = (\lqot{H})_d.
    \]
    In particular, the representation with character $\lrep{H}_d$ is a sub-representation of the representation with character $\lrep{\Gamma}_d$.
    The graph $H$ is in fact a Hessenberg graph, and it is easy to compute with $P$-tableaux from \cite{SW2016chromaticquasisymmetric} that the $d$-th graded piece of $\lrep{H}$ is non-trivial, so the $d$-th graded piece of $\lrep{\Gamma}$ contains a non-trivial sub-representation and is thus nontrivial.
\end{proof}
\begin{remark}
    The graph $H$ in the proof of Theorem \ref{thm:kconn} is the Hessenberg graph associated to the vector 
    \[
    h = (\overbrace{\ell+d,\ldots,\ell+d}^{\ell \text{ times}},n,\ldots,n).
    \]
    The $3+1$-- and $2+2$--free poset $P$ on $[n]$ for which $H$ is the indifference graph has relations $\{i <_P j \mid i \in [\ell],\; j \in \{d+\ell+1,\ldots,n\}\}$.
\end{remark}
The following corollary is a consequence of Theorem \ref{thm:kconn}.
\begin{corollary}\label{cor:conn_rrep}
    If $\Gamma$ is $k$-connected, then $\rrep{\Gamma}_d$ is equal to $\rrep{K_n}_d$, which is the $d$-th degree piece of the graded regular representation.
\end{corollary}

\section{Generators for Linear Splines}\label{sec:lin_gen_rels}
The remaining sections are devoted to computing the first degree piece of the graded symmetric functions $\lrep{\Gamma}$ and $\rrep{\Gamma}$ for all connected graphs $\Gamma$. We show that the first degree piece of $\splines{\Gamma}$ is computable from the data of cut vertices and cut edges, in particular the block-cut tree of $\Gamma$ (Def. \ref{def:block_cut}). \par 
This section defines a set $\cf_\Gamma$ that we will eventually show is a $\C$-spanning set for $\splines{\Gamma}^1$. Subsection \ref{ssec:lin_relations} proves several $\C$-linear relations within the set $\cf_\Gamma$, that will turn out to be sufficient for reducing to a basis.\par 
First we will introduce (in fact re-introduce) a collection of linear splines that depend on cut edges in $\Gamma$. Let $s = (i,j)$ be a cut edge of $\Gamma$, and let $G_s$ be one of the two connected components of $([n],E(\Gamma) \setminus\{s\})$. We are free to choose either component, see Remark \ref{rem:cutedge_component_start} below. For each subset $A \subset [n]$ such that $\abs{A} = \abs{G_s}$, we define $\bar{f}_A^s \colon S_n \to \poly$ by
\[\bar{f}_A^s(w) \coloneqq \begin{cases}
    t_{w(i)}-t_{w(j)} &\text{if }w^{-1}(A) = V(G_s)  \\
    0 &\text{otherwise,} 
\end{cases} \]
 for all $w \in S_n$. We associate to $\Gamma$ the collection \[\cc_\Gamma \coloneqq \left\{\bar{f}_{A}^s \mid s \text{ is a cut edge of }\Gamma,\; A \subset [n],\; \abs{A} = \abs{G_s}\right\}.\]
Note that these splines $\bar{f}_A^s$ are actually the linear coset splines from Definition \ref{def:coset_spline}. We make the change in notation for several reasons, one being that the subset $A$ uniquely determines the coset whose support is $\bar{f}_A^s$ (as opposed to many elements $w$ defining the same $\bar{f}_w^B$). Since $G_s$ is one of two connected components in the graph $([n],E(\Gamma)\setminus\{s\})$, it follows that $v \in w\llangle E(\Gamma)\setminus\{s\} \rrangle$ if and only if $w(V(G_s)) = v(V(G_s))$.  In particular, we have equality $\bar{f}_A^s = \bar{f}_w^B$ precisely when $B = E(\Gamma) \setminus \{s\}$ and $w^{-1}(A) = V(G_s)$. \par 
\begin{remark}\label{rem:cutedge_component_start}
    When defining $\bar{f}_{A}^s$, we chose $G_s$ to be one of the two connected components in the graph $([n],E(\Gamma)\setminus\{s\})$. This choice does not affect the set of splines in $\cc_\Gamma$. More precisely, if $H$ is the other connected component in $([n],E(\Gamma)\setminus\{s\})$, then $\abs{H} = n-\abs{G_s}$, and we have that
    \[\bar{f}_{A}^s(w) = \begin{cases}
    t_{w(i)}-t_{w(j)} &\text{if }w^{-1}(A) = V(G_s)  \\
    0 &\text{otherwise} 
\end{cases} = \begin{cases}
    t_{w(i)}-t_{w(j)} &\text{if }w^{-1}(A^c) = V(H)  \\
    0 &\text{otherwise.} 
\end{cases}\]
    In particular, for a fixed cut edge $s$ the set of linear coset splines $\{\bar{f}_{A}^s\}$ associated to that cut edge is unaffected by the choice of $G_s$.
\end{remark}
Now we will introduce a (truly new) collection of linear splines that depend on cut vertices $j$ and the connected components of $\Gamma - j$, as well as an integer $k$. Let $j \vdash \Gamma$ be a cut vertex, $G$ be a connected component of $\Gamma-j$, and $k \in [n]$. We define $\bar{y}_{G,k}^j \colon S_n \to \poly$ by \[ \bar{y}_{G,k}^j(w) \coloneqq \begin{cases}
            t_k - t_{w(j)} & \text{if }w^{-1}(k) \in G \\ 0 &\text{otherwise}
        \end{cases}\]  
for all $w \in S_n$. We associate to $\Gamma$ the collection 
\[ \cy_\Gamma \coloneqq \left\{\left.\bar{y}_{
G,k}^j \right| j \vdash \Gamma,\; G\text{ a connected component of }\Gamma-j,\; k \in [n]  \right\}.\]
Finally recall the splines $\ct_n \coloneqq \{\bar{t}_i \mid i \in [n]\}$ and $\cx_n \coloneqq \{\bar{x}_i \mid i \in [n]\}$ from Subsection \ref{ssec:background_splines}. Now we define
\begin{equation}\label{eqn:linear_generators}
    \cf_\Gamma \coloneqq \ct_n \cup \cx_n \cup \cc_\Gamma \cup \cy_\Gamma.
\end{equation}
We will eventually show that $\cf_\Gamma$ is a $\C$-spanning set of $\splines{\Gamma}^1$. 
\begin{example}\label{ex:linear_generating_splines}
    Let $\Gamma$ be the graph drawn below.
    
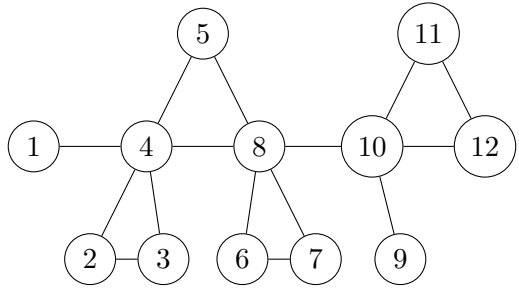
\begin{center} 
    \begin{tikzpicture}
            \begin{scope}[scale=1.5]
                \draw (0,0) node[draw,circle] (V1) {1};
                \draw (0.5,-1) node[draw,circle] (V2) {2};
                \draw (1.15,-1) node[draw,circle] (V3) {3};
                \draw (1,0) node[draw,circle] (V4) {4};
                \draw (1.5,1) node[draw,circle] (V5) {5};
                \draw (1.85,-1) node[draw,circle] (V6) {6};
                \draw (2.5,-1) node[draw,circle] (V7) {7};
                \draw (2,0) node[draw,circle] (V8) {8};
                \draw (3,0) node[draw,circle] (V10) {10};
                \draw (3.25,-1) node[draw,circle] (V9) {9};
                \draw (3.5,1) node[draw,circle] (V11) {11};
                \draw (4,0) node[draw,circle] (V12) {12};

                \draw (V1)--(V4);
                \draw (V2)--(V4);
                \draw (V3)--(V4);
                \draw (V2)--(V3);
                \draw (V4)--(V5);
                \draw (V5)--(V8);
                \draw (V4)--(V8);
                \draw (V6)--(V7);
                \draw (V6)--(V8);
                \draw (V7)--(V8);
                \draw (V8)--(V10);
                \draw (V9)--(V10);
                \draw (V10)--(V11);
                \draw (V10)--(V12);
                \draw (V11)--(V12);
            \end{scope}
        \end{tikzpicture}
    \end{center}
    Since $\Gamma$ has three cut edges $(1,4)$, $(8,10)$, and $(9,10)$, we have that  
           \[ \cc_\Gamma =\left\{\bar{f}_A^{(1,4)} \mid A \subset [12],\, \abs{A} =1 \right\} \cup \left\{\bar{f}_A^{(8,10)} \mid A \subset [12],\, \abs{A} =8 \right\} \cup \left\{\bar{f}_A^{(9,10)} \mid A \subset [12],\, \abs{A} =1 \right\}\] 
        One such element $\bar{f}_{\{6\}}^{(9,10)} \in \cc_\Gamma$ takes the form 
        \[\bar{f}_{\{6\}}^{(9,10)}(w) \coloneqq \begin{cases}
    t_{w(9)}-t_{w(10)} & w^{-1}(\{6\}) = \{9\}  \\
    0 &\text{otherwise,}
\end{cases}\]
and is supported on the coset $\{w \in S_n \mid w(9) = 6\}$. \par 
    Since $\Gamma$ has three cut vertices $4$, $8$, and $10$, we have that 
    \begin{align*}
        \cy_\Gamma=\left\{\bar{y}_{G,k}^4 \left| V(G) \in \left\{ \begin{matrix}
        \{1 \}, \\ \{2,3 \}, \\ \{5,...,12 \}
    \end{matrix} \right\},\; k \in [12] \right.\right\}&\cup \left\{\bar{y}_{G,k}^8 \left| V(G) \in \left\{ \begin{matrix}
        \{1,..,,5 \}, \\ \{6,7 \}, \\ \{9,...,12 \}
    \end{matrix} \right\},\; k \in [12] \right.\right\} \\&\cup \left\{\bar{y}_{G,k}^{10} \left| V(G) \in \left\{ \begin{matrix}
        \{1,...,8 \}, \\ \{9 \}, \\ \{11,12 \}
    \end{matrix} \right\},\; k \in [12] \right.\right\}
    \end{align*}
    One such element $\bar{y}_{G,3}^8 \in \cy_\Gamma$ takes the form \[\bar{y}_{\{6,7\},3}^8(w) \coloneqq \begin{cases}
            t_3 - t_{w(8)} & w^{-1}(3) \in \{6,7\} \\ 0 &\text{otherwise}
        \end{cases},\]
    and is supported on the set $\{w \in S_n \mid w(6) = 3 \text{ or } w(7) = 3\}$.
\end{example}

Now Lemma \ref{lem:linear_generators} below shows that $\cf_\Gamma$ is in fact a subset of $\splines{\Gamma}^1$.\par 
\begin{lemma}\label{lem:linear_generators}
    Let $\Gamma$ be a graph on $[n]$. The four sets $\ct_n$, $\cx_n$, $\cc_\Gamma$, and $\cy_\Gamma$ are subsets of $\splines{\Gamma}$.
\end{lemma}
\begin{proof}
    We already know that $\bar{t}_i$ and $\bar{x}_i$ are elements of $\splines{\Gamma}$ for all $i \in [n]$, so $\ct_n$ and $\cx_n$ are subsets.\par 
    Now we show that each element of $\cc_\Gamma$ is a well-defined spline. Recall that these are coset splines, and so are well defined for trees. If $s$ is a cut edge of $\Gamma$, then every spanning tree $T$ of $\Gamma$ must have $s$ as an edge. Fix $A \subset [n]$ where $\abs{A} = \abs{G_s}$, and for all spanning trees $T$ choose $T_s$ to be the connected component where $V(T_s) = V(G_s)$. It follows that $\bar{f}_A^s \in \splines{T}$ for every spanning tree $T$, and so $\bar{f}_A^s \in \splines{\Gamma}$ by Lemma \ref{lem:unions}. \par 
    Finally we show that every element $\bar{y}_{G,k}^j \in \cy_\Gamma$ is a linear spline on $\cg_\Gamma$. We will verify this from the definition, edge by edge. Let $(w,v) \in E(\cg_\Gamma)$ where $w = v(p,q)$. We prove that $\bar{y}_{G,k}^j(w)-\bar{y}_{G,k}^j(v) \in \cl(w,v) = \llangle t_{v(p)} - t_{v(q)}\rrangle$ in three cases, depending on the values of $w^{-1}(k)$ and $v^{-1}(k)$. \par 
    Case 1: $w^{-1}(k),v^{-1}(k) \notin G$. Then by definition both $\bar{y}_{G,k}^j(w) = 0$ and $\bar{y}_{G,k}^j(v) = 0$, so the difference is clearly in $\cl(w,v)$. \par 
    Case 2: $w^{-1}(k),v^{-1}(k) \in G$. So $\bar{y}_{G,k}^j$ contains both $w$ and $v$ in its support. We compute from the definition that 
    \begin{align*}
        \bar{y}_{G,k}^j(w)-\bar{y}_{G,k}^j(v) &= t_k-t_{w(j)}-t_k+t_{v(j)} \\
        &= t_{v(j)} - t_{w(j)}  = \begin{cases}
            \pm(t_{v(p)} - t_{v(q)}) & j\in \{p,q\} \\            
            0 & j \notin \{p,q\}.
        \end{cases}
    \end{align*}
    In either case this difference is in the ideal $\cl(w,v)$. \par 
    Case 3: $w^{-1}(k) \in G$, $v^{-1}(k) \notin G$. In particular, $w$ is in the support of $\bar{y}_{G,k}^j$ whereas $v$ is not. Since $w^{-1}(k) = (p,q)v^{-1}(k)$, we know that one of either $p$ or $q$ is in $G$ and the other is not. Without loss of generality, say $p \in G$ and $q \notin G$. In particular, $w^{-1}(k) = p$ and $v^{-1}(k) = q$. Since $(p,q) \in E(\Gamma)$ and the only element in $[n] \setminus V(G)$ that elements of $G$ are connected to is the vertex $j$, it follows that $v^{-1}(k) = q = j$. So $v(q) = k$ and $w(j) = v(p,q)(j) = v(p)$. Compute that 
    \[
    \bar{y}_{G,k}^j(w)-\bar{y}_{G,k}^j(v) = t_k-t_{w(j)} = t_{v(q)} - t_{v(p,q)(j)} = t_{v(q)} - t_{v(p)} \in \cl(w,v).
    \]
    Thus $\bar{y}_{G,k}^j$ is an element of $\splines{\Gamma}$.
\end{proof}
The splines in $\cf_\Gamma$ are defined from graph properties that are intrinsic to the isomorphism class of $\Gamma$. Lemma \ref{lem:linear_isoms} below makes this precise.
\begin{lemma}\label{lem:linear_isoms}
    Let $\omega \colon \Gamma \to \Gamma'$ be a graph isomorphism and $\Omega$ be as in Subsection \ref{ssec:isoms}. Then $\cf_{\Gamma'} = \Omega\left(\cf_{\Gamma}\right)$.
\end{lemma}
\begin{proof}
    It follows directly from the definitions that $\Omega(\cx_n) = \cx_n$ and and $\Omega(\ct_n) = \ct_n$. \par 
    The image of the coset spline $\bar{f}_A^{(i,j)} \in \splines{\Gamma}$ can be computed to be the coset spline $\bar{f}_{w^{-1}(A)}^{(w(i),w(j))} \in \splines{\Gamma'}$, where we consistently choose the connected component $\omega(G_s)$. From this, it is straightforward from the definitions to verify that $\cc_{\Gamma'} = \Omega(\cc_\Gamma)$ \par 
    Similarly, it is easy to verify that $\bar{y}_{G,k}^j \mapsto \bar{y}_{w(G),w(k)}^{w(j)}$, and so $\cy_{\Gamma'} = \Omega(\cy_\Gamma)$.
\end{proof}
By Lemma \ref{lem:linear_isoms} it suffices to prove that $\cf_\Gamma$ spans $\splines{\Gamma}^1$ for any particular graph in isomorphism class of $\Gamma$. 
\subsection{Some Relations}\label{ssec:lin_relations}
This set $\cf_\Gamma$ is not a $\C$-basis of $\splines{\Gamma}^1$. Indeed, the following Lemmas \ref{lem:linear_relations}, \ref{lem:strong_y_spline}, and \ref{lem:hard_relation} give relations between elements of $\cf_\Gamma$. \par 
The first set of relations in the Lemma \ref{lem:linear_relations} are relatively straightforward.
\begin{lemma}\label{lem:linear_relations}
    For $\bar{t}_i \in \ct_n$, $\bar{x}_i\in \cx_n$, $\bar{f}_A^s\in \cf_\Gamma$, and $\bar{y}_{G,k}^j \in \cy_\Gamma$, the following relations hold:
    \begin{enumerate}
        \item[(1)] ${\ds \sum_{r=1}^n \bar{x}_r = \sum_{r=1}^n \bar{t}_r }$,
        \item[(2)] if $(i,j)$ is a cut edge and $G_{(i,j)}$ is the component containing the vertex $i$, then \\${\ds \sum_{A} \bar{f}_A^{(i,j)} = \bar{x}_i - \bar{x}_j}$, where the sum is over all $A \subset [n]$ such that $\abs{A} = \abs{G_{(i,j)}}$, 
        \item[(3)] if $j \vdash \Gamma$, then ${\ds \sum_{k=1}^n \bar{y}_{G,k}^j = \left(\sum_{r\in G} \bar{x}_r\right) - \abs{G}\bar{x}_j}$ for any connected component $G$ of $\Gamma - j$, and
        \item[(4)] if $j \vdash \Gamma$ and $k \in [n]$ is fixed, ${\ds \sum_{G} \bar{y}_{G,k}^i = \bar{t}_k - \bar{x}_j}$, where the sum is over all connected components $G$ of $\Gamma - j$.
    \end{enumerate}
\end{lemma}
\begin{proof}
    Relation $(1)$ is easy, as is relation $(2)$ once it is noted that the support of each $\bar{f}_A^{(i,j)}$ for a fixed $(i,j)$ is disjoint. Fix $w \in S_n$.\par 
    For relation (3),  compute that
    \begin{align*}
        \sum_{k=1}^n \bar{y}_{G,k}^j(w) &= \sum_{\substack{k \in [n]\\w^{-1}(k) \in G}} t_k - t_{w(j)} = \left(\sum_{\substack{k\in [n]\\w^{-1}(k) \in G}} t_k\right) - \abs{G}t_{w(j)} \\ 
        &= \left(\sum_{r \in G} t_{w(r)}\right) - \abs{G}t_{w(j)} = \left(\left(\sum_{r\in G} \bar{x}_r\right) - \abs{G}\bar{x}_j\right)(w).
    \end{align*}
    For relation (4), note that either $w^{-1}(k) = j$ or $w^{-1}(k) \in G$ for one and only one connected component $G$ of $\Gamma - j$. As such, if $w^{-1}(k) = j$, in which case $w$ is not in the support of any of the $\bar{y}_{G,k}^j$ and likewise $\bar{t}_k(w)-\bar{x}_j(w) = 0$. Otherwise, $w^{-1}(k) \in G$ for some particular connected component $G$ and $\bar{y}_{G,k}^j(w) = t_k-t_{w(j)}$. This is precisely $\bar{t}_k(w) - \bar{x}_j(w)$, and we have (4).
\end{proof}
 Lemma \ref{lem:strong_y_spline} below shows that if $j$ is a cut vertex and a connected component $G$ of $\Gamma - j$ is connected to $j$ by a cut edge $(i,j)$, then the spline $\bar{y}_{G,k}^j$ can be written as a sum of other splines from $\cf_\Gamma$. In particular, it will allow us to remove from $\cf_\Gamma$ the splines in $\cy_\Gamma$ that correspond to components connected by cut edges.
\begin{lemma}\label{lem:strong_y_spline}
    Let $\Gamma$ be a graph, $j$ a cut vertex and $(i,j)$ a cut edge, choosing $G_{(i,j)}$ to be the component containing the vertex $i$. Let $C$ be connected component of the forest with vertex set $V(G_{(i,j)}) \cup \{j\}$ and edge set $\{s \mid s\text{ is a cut edge of }G_{(i,j)}\} \cup \{(i,j)\}$ that contains the vertex $j$. Then for all $k \in [n]$, 
    \[
    \sum_{\substack{v \in C-j \\v \vdash \Gamma}} \sum_{\substack{ G\\ G \cap C = \emptyset}} \bar{y}_{G,k}^v + \sum_{a \in E(C)} \sum_{\substack{A \ni k\\ \abs{A} = \abs{G_a}}} \bar{f}_A^{a} = \bar{y}_{G_{(i,j)},k}^j,
    \]
    where $G$ in the first double-sum is a connected component of $\Gamma - v$, where $G_a$ is the connected component that is a subset of $G_{(i,j)}$, and where $A$ in the second double-sum is a subset of $[n]$.
\end{lemma}
Since Lemma \ref{lem:strong_y_spline} is rather technical, we will walk through an example before seeing the full proof. Let $\Gamma$ be the graph on $13$ vertices below, 
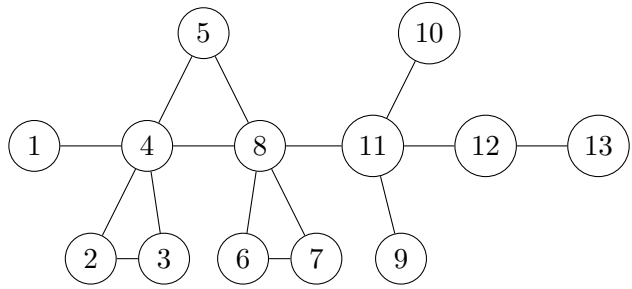
\begin{center} 
    \begin{tikzpicture}
            \begin{scope}[scale=1.5]
                \draw (0,0) node[draw,circle] (V1) {1};
                \draw (0.5,-1) node[draw,circle] (V2) {2};
                \draw (1.15,-1) node[draw,circle] (V3) {3};
                \draw (1,0) node[draw,circle] (V4) {4};
                \draw (1.5,1) node[draw,circle] (V5) {5};
                \draw (1.85,-1) node[draw,circle] (V6) {6};
                \draw (2.5,-1) node[draw,circle] (V7) {7};
                \draw (2,0) node[draw,circle] (V8) {8};
                \draw (3,0) node[draw,circle] (V10) {11};
                \draw (3.25,-1) node[draw,circle] (V9) {9};
                \draw (3.5,1) node[draw,circle] (V11) {10};
                \draw (4,0) node[draw,circle] (V12) {12};
                \draw (5,0) node[draw,circle] (V13) {13};

                \draw (V1)--(V4);
                \draw (V2)--(V4);
                \draw (V3)--(V4);
                \draw (V2)--(V3);
                \draw (V4)--(V5);
                \draw (V5)--(V8);
                \draw (V4)--(V8);
                \draw (V6)--(V7);
                \draw (V6)--(V8);
                \draw (V7)--(V8);
                \draw (V8)--(V10);
                \draw (V9)--(V10);
                \draw (V10)--(V11);
                \draw (V10)--(V12);
                \draw (V12)--(V13);
            \end{scope}
        \end{tikzpicture}
    \end{center}
    where $(11,12)$ is cut edge, and $G_{(11,12)}$ is the component that contains $11$. The the forest of cut edges with vertex set $V(G_{(11,12)}) \cup \{12\}$ has edge set $\{(1,4),(8,11),(9,11),(10,11),(12,11)\}$. If we mark in $\Gamma$ the component of this forest that contains the vertex $11$ with double lines, we get the following graph.
    
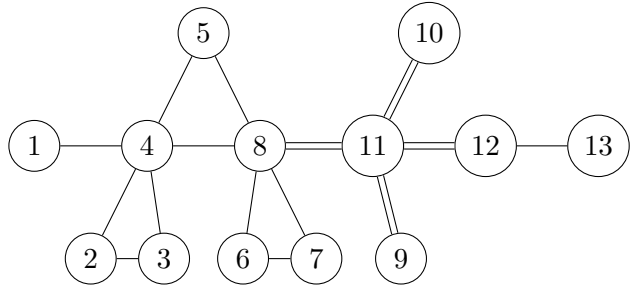
\begin{center} 
    \begin{tikzpicture}
            \begin{scope}[scale=1.5]
                \draw (0,0) node[draw,circle] (V1) {1};
                \draw (0.5,-1) node[draw,circle] (V2) {2};
                \draw (1.15,-1) node[draw,circle] (V3) {3};
                \draw (1,0) node[draw,circle] (V4) {4};
                \draw (1.5,1) node[draw,circle] (V5) {5};
                \draw (1.85,-1) node[draw,circle] (V6) {6};
                \draw (2.5,-1) node[draw,circle] (V7) {7};
                \draw (2,0) node[draw,circle] (V8) {8};
                \draw (3,0) node[draw,circle] (V10) {11};
                \draw (3.25,-1) node[draw,circle] (V9) {9};
                \draw (3.5,1) node[draw,circle] (V11) {10};
                \draw (4,0) node[draw,circle] (V12) {12};
                \draw (5,0) node[draw,circle] (V13) {13};

                \draw (V1)--(V4);
                \draw (V2)--(V4);
                \draw (V3)--(V4);
                \draw (V2)--(V3);
                \draw (V4)--(V5);
                \draw (V5)--(V8);
                \draw (V4)--(V8);
                \draw (V6)--(V7);
                \draw (V6)--(V8);
                \draw (V7)--(V8);
                \draw[double distance=2pt] (V8)--(V10);
                \draw[double distance=2pt] (V9)--(V10);
                \draw[double distance=2pt] (V10)--(V11);
                \draw[double distance=2pt] (V10)--(V12);
                \draw (V12)--(V13);
            \end{scope}
        \end{tikzpicture}
    \end{center}
    In essence, Lemma \ref{lem:strong_y_spline} says that for all $k \in [13]$, the spline $\bar{y}_{G_{(11,12)},k}^{12}$ can be written as a sum of some splines in $\cc_\Gamma$ associated to the cut edges $(8,11)$, $(9,11)$, and $(10,11)$ (since they are a part of that marked tree), and some splines in $\cy_\Gamma$ associated to cut vertex $8$, since $\Gamma - 8$ has components that ``hang off of" that marked tree.\par 
    For each of the cut edges $a \in \{(8,11),(9,11),(10,11)\}$, we must choose $G_a$ to be the component contained within $G_{(11,12)}$, so $V(G_{(9,11)}) = \{1,...,11\}$, $V(G_{(9,11)}) = \{9\}$, and $V(G_{(10,11)}) = \{10\}$.\par 
    More formally, Lemma \ref{lem:strong_y_spline} states that for all $k \in [13]$, the following equality holds (we will denote the specific subgraph by their vertex set).
    \[\bar{y}_{[5],k}^8 + \bar{y}_{\{6,7\},k}^8 + \sum_{\substack{\abs{A} = 8 \\ k \in A}} \bar{f}_A^{(8,11)}+
    \sum_{\substack{\abs{A} = 1 \\ k \in A}} \bar{f}_A^{(9,11)}+
    \sum_{\substack{\abs{A} = 1 \\ k \in A}} \bar{f}_A^{(10,11)}+
    \sum_{\substack{\abs{A} = 11 \\ k \in A}} \bar{f}_A^{(11,12)} = \bar{y}_{[11],k}^{12}.\]
    The proof proceeds in cases by the value of $w^{-1}(k)$. For example, say we wished to evaluate both sides of the above expression at a $w \in S_{13}$ such that $w^{-1}(k) = 3$. This makes it easier to determine which splines in the sum above on the left are supported at $w$. In particular,
    \begin{enumerate}[itemsep=4pt]
        \item ${\ds \bar{y}_{[5],k}^8(w) = t_k - t_{w(8)}}$ since $3 \in [5]$,
        \item ${\ds \bar{y}_{\{6,7\},k}^8(w) = 0}$ since $3 \notin \{6,7\}$,
        \item ${\ds \sum_{\substack{\abs{A} = 8 \\ k \in A}} \bar{f}_A^{(8,11)}(w) = t_{w(8)} - t_{w(11)}}$ since the set $A = w(V(G_{(8,11)})$ contains $w(3) = k$,
        \item ${\ds \sum_{\substack{\abs{A} = 1 \\ k \in A}} \bar{f}_A^{(9,11)}(w) = 0}$ and ${\ds \sum_{\substack{\abs{A} = 1 \\ k \in A}} \bar{f}_A^{(10,11)}(w) = 0}$ since $k \notin w(\{9\})$ and $k\notin w(\{10\})$, and finally 
        \item ${\ds \sum_{\substack{\abs{A} = 11 \\ k \in A}} \bar{f}_A^{(11,12)}(w) = t_{w(11)} - t_{w(12)}}$ since the set $A = w(V(G_{(11,12)})$ contains $w(3) = k$.
    \end{enumerate}
    If we add these all up, the sum telescopes and the evaluation is 
    \[(t_k - t_{w(8)}) + (t_{w(8)} - t_{w(11))} + (t_{w(11)} - t_{w(12)}) = t_k - t_{w(12)},\]
    which is precisely ${\ds \bar{y}_{[11]}^{12}(w)}$. The essence of the proof, which we will now provide, is that the splines in the sum with support at a particular $w \in S_n$ can be determined from any simple path from $w^{-1}(k)$ to $j$, and that the sum always telescopes as it did in the example above.
\begin{proof}[Proof of Lemma \ref{lem:strong_y_spline}]
    Let $w \in S_n$. We will show that both sides of the claimed equality are equal when evaluated at $w$. Let $k \in [n]$, and we will proceed in cases based off of the value $w^{-1}(k)$. \par 
    First, say $w^{-1}(k) \notin V(G_{(i,j)})$. All paths that begin in $G_{(i,j)}$ and leave must contain the edge $(i,j)$ and thus visit the vertex $j$. In particular, for any $v \in C$ where $c \vdash \Gamma$, the connected component of $\Gamma - v$ that contains $w^{-1}(k)$ also contains $j \in C$. So $w^{-1}(k) \notin G$ for any $G$ in the first double-sum, and so 
    \[\sum_{\substack{v \in C-j \\v \vdash \Gamma}} \sum_{\substack{ G\\ G \cap C = \emptyset}} \bar{y}_{G,k}^v(w) = 0\] 
    If $a \in E(C)$ then $G_a$ was chosen to be contained within $G_{(i,j)}$, so $w^{-1}(k) \notin G_a$ as well. In particular, if $k \in A$ then $w^{-1}(A) \neq G_a$, and so 
    \[\sum_{a \in E(C)} \sum_{\substack{A \ni k\\\abs{A} = \abs{G_a}}} \bar{f}_A^{a}(w) = 0.\]
    It is direct from the definition that $\bar{y}_{G_{(i,j)},k}^v(w) = 0$, and so the claim holds when $w^{-1}(k) \notin G_{(i,j)}$. \par 
    Now assume that  $w^{-1}(k) \in G_{(i,j)}$.  
    Let $P = \left(p_0,\ldots,p_\ell, p_{\ell+1} \right)$ be a simple path from $p_0 \coloneqq w^{-1}(k)$ to $p_{\ell+1} \coloneqq j$. Note that $p_\ell$ must be the vertex $i$. This path $P$ may start outside of $C$, but must eventually enter the tree $C$. Say that $m \in \{0,...,\ell\}$ is the lowest index such that $p_m \in C$. Since $i \in C$ and $p_\ell = i$, this integer $m$ does exist. Simple paths from vertex to vertex within trees are unique, so there is a unique simple path from $p_m$ to $j$ in $C$. This path is $P^C \coloneqq (p_m,...,p_{\ell+1})$. \par 
    First we will determine value of $\bar{f}_A^{a}(w)$ for $a \in E(C)$. Say the edge $a \in C$ is not an edge in the path $P^C$. Then the vertices $w^{-1}(k)$ and $j$ are in the same connected component of the graph $([n],E(\Gamma)-a)$. In particular $w^{-1}(k) \notin G_a$, so if $k \in A$ then $w^{-1}(A) \neq V(G_a)$. It follows that for all $A$ such that $k \in A$, if $a \notin P^C$ then  $\bar{f}_A^{a}(w) = 0$. \par
    On the other hand, if $a \in P^C$ then we may let $A \coloneqq w(V(G_a))$, and then $k \in A$. So for each $a \in P^C$ there exists a single spline $\bar{f}_A^{a}$ in the sum that is supported at $w$. In particular, if $a = (p,q)$ then $\bar{f}_A^{q}(w) = t_{w(p)} - t_{w(q)}$.\par 
    At this point, we have that 
    \[
    \sum_{\substack{v \in C-j \\v \vdash \Gamma}} \sum_{\substack{ G\\ G \cap C = \emptyset}} \bar{y}_{G,k}^v(w) + \sum_{a \in E(C)} \sum_{\substack{A \ni k\\\abs{A} = \abs{G_a}}} \bar{f}_A^{a}  = \sum_{\substack{v \in C-j \\v \vdash \Gamma}} \sum_{\substack{ G\\ G \cap C = \emptyset}} \bar{y}_{G,k}^v(w) + \sum_{(p,q) \in P^C} t_{w(p)} - t_{w(q)}
    \]
    Now we will have two cases, if $w^{-1}(k) \in C$ and if $w^{-1}(k) \notin C$. \par 
    Case 1: $w^{-1}(k) \in C$. So $m = 0$. Then for all $v \in G_{(i,j)}$, if a connected component $G$ of $\Gamma - v$ contains $w^{-1}(k)$ then so does $G \cap C$. In particular, \[\sum_{\substack{v \in C-j \\v \vdash \Gamma}} \sum_{\substack{ G\\ G \cap C = \emptyset}} \bar{y}_{G,k}^v(w) = 0.\] 
Now we compute  
    \begin{align*}
         \sum_{(r,v) \in P^C} t_{w(v)} - t_{w(r)}   &=\sum_{i=0}^\ell t_{w(p_i)} - t_{w(p_{i+1})} \\
        &= t_{w(p_0)} - t_{w(p_{\ell+1})}\\
        &= t_k - t_{w(j)}.
    \end{align*}
    This is precisely $\bar{y}_i^j(w)$, and so the equality holds if $w^{-1}(k) \in C$.\par 
    Case 2: $w^{-1}(k) \notin C$. Then $m \neq 0$, and consider the vertex $p_{m-1}$. Since $p_m \in C$ and $\Gamma - p_m$ separates $w^{-1}(k)$ from $j$, the vertex $p_m$ is a cut vertex of $\Gamma$. If $v' \in C$ is any vertex other than $p_m$, then $p_m$ and $w^{-1}(k)$ are in the same connected component of $\Gamma - v'$ (connected via the path $(p_0,...,p_m)$). In particular, any connected component of $\Gamma - v'$ that contains $w^{-1}(k)$ intersects non-trivially with $C$. So for $v=p_m \in C-j$ there is precisely one component $G$ of $\Gamma - v$ that contains $w^{-1}(k)$. If this component $G$ intersected nontrivially with $C$, then $(p_{m-1},p_m)$ would have to be an edge in $C$, but $(p_{m-1},p_m)$ was explicitly assumed not to be a cut edge. In particular, $w$ is supported on one and only one spline in the first double-sum (the one where $v = p_m$ and $G \ni p_{m-1})$), and so 
    \[\sum_{\substack{v \in C-j \\v \vdash \Gamma}} \sum_{\substack{ G\\ G \cap C = \emptyset}} \bar{y}_{G,k}^v(w) = t_k - t_{w(p_m)}. \]
    It follows that 
    \begin{align*}
        \sum_{\substack{v \in C-j \\v \vdash \Gamma}} \sum_{\substack{ G\\ G \cap C = \emptyset}} \bar{y}_{G,k}^v(w) + \sum_{(p,q) \in P^C} t_{w(p)} - t_{w(q)} &= t_k - t_{w(p_m)} + \sum_{r=m}^\ell t_{w(p_r)} - t_{w(p_{r+1})} \\
        &= t_{k} - t_{w(p_m)} +  t_{w(p_m)} - t_{w(p_{\ell+1})}\\
        &= t_k - t_{w(j)}.
    \end{align*}
    So in either case, the sum evaluates to $\bar{y}_{G_{(i,j)},k}^j(w)$.
\end{proof}
    Now Lemma \ref{lem:strong_y_spline} two very important consequences. First, as mentioned, it will allow us to disregard those splines $\bar{y}_{G,k}^j$ where $G$ is connected to $j$ via a cut edge. Second, observe we only required $j$ to be a cut vertex so that $\bar{y}_{G,k}^j$ is defined. We may however remove this restriction and ``force through" the argument as follows. If $(p,q)$ is a cut edge and $q$ is not a cut vertex then $q$ is a leaf in $\Gamma$. Let $G_{(p,q)}$ be the connected component containing $p$ (and thereby all of $[n] \setminus\{q\}$). We might abuse notation and let for all $w \in S_n$,
    \begin{align*}
    \bar{y}_{G_{(p,q)},k}^q(w) &= \begin{cases}
            t_k - t_{w(q)} & \text{if }w^{-1}(k) \in [n] \setminus\{q\} \\ 0 &\text{otherwise}
        \end{cases} \\
        &= \bar{t}_k(w) - \bar{x}_q(w),
    \end{align*}
    and get another relation from Lemma \ref{lem:strong_y_spline}. A consequence of this is Lemma \ref{lem:hard_relation} below.
\begin{lemma}\label{lem:hard_relation}
    Let $(i,j)$ be a leaf edge in $\Gamma$ with $j$ the leaf vertex, and $G_{(i,j)}$ the connected component of $([n],E(\Gamma) \setminus\{s\})$ that contains $i$. Then for all $A \subset [n]$ such that $\abs{A} = \abs{G_{(i,j)}} = n-1$, we have that 
    \[ \bar{f}_A^{(i,j)} \in \C\left\{\bar{\rho} \in \cf_\Gamma\left| \bar{\rho} \neq \bar{f}_B^{(i,j)} \text{ for any }B \subset [n] \right. \right\}.\]
\end{lemma}
\begin{proof}
    Let $C$ be connected component of the forest with vertex set $V(G_{(i,j)}) \cup \{j\} = [n]$ and edge set $\{s \mid s\text{ is a cut edge of }\Gamma)\}$ that contains the vertex $j$. By Lemma \ref{lem:strong_y_spline} and the discussion above, for all $k \in [n]$, 
    \[
    \sum_{\substack{v \in C-j \\v \vdash \Gamma}} \sum_{\substack{ G\\ G \cap C = \emptyset}} \bar{y}_{G,k}^v + \sum_{a \in E(C)} \sum_{\substack{A \ni k\\\abs{A} = \abs{G_a}}} \bar{f}_A^{a} = \bar{t}_k - \bar{x}_j,
    \]
    where $G$ in the first double-sum is a connected component of $\Gamma - v$, $G_{a}$ is the connected component of $([n],E(\Gamma)\setminus\{(i,j)\})$ that is a subgraph of $G_{(i,j)}$ (i.e. doesn't contain $j$), and $A$ in the second double-sum is a subset of $[n]$.  In particular,
    \[
    \sum_{\substack{\abs{A} = n-1 \\ k \in A}} \bar{f}_A^{(i,j)}= \bar{t}_k - \bar{x}_n-  \sum_{\substack{v \in C-j \\v \vdash \Gamma}} \sum_{\substack{ G\\ G \cap C = \emptyset}} \bar{y}_{G,k}^v - \sum_{\substack{a \in E(C) \\ a\neq (i,j)}} \sum_{\substack{A \ni k\\\abs{A}=\abs{G_a} }} \bar{f}_A^{a}.
    \]
    Now the right hand side is in ${\ds \C\left\{\bar{\rho} \in \cf_\Gamma\left| \bar{\rho} \neq \bar{f}_B^{(i,j)} \text{ for any }B \subset [n] \right. \right\}}$. Let ${\ds \bar{\sigma}_k :=\sum_{\substack{\abs{A} = n-1 \\ k \in A}} \bar{f}_A^{(i,j)},}$
    so the above relation says ${\ds \bar{\sigma}_k \in \\C\left\{\bar{\rho} \in \cf_\Gamma\left| \bar{\rho} \neq \bar{f}_B^{(i,j)} \text{ for any }B \subset [n] \right. \right\}}$. For each $p \in [n]$ we have that
    \[ \bar{f}_{[n] \setminus \{p\}}^{(i,j)}  = \left(\frac{1}{n-1} \sum_{k \in [n]} \bar{\sigma}_k\right) - \bar{\sigma}_p.  \]
    Thus ${\ds } \bar{f}_{[n] \setminus \{p\}}^{(i,j)} \in \C\left\{\bar{\rho} \in \cf_\Gamma\left| \bar{\rho} \neq \bar{f}_B^{(i,j)} \text{ for any }B \subset [n] \right. \right\}$, and as all subsets of size $\abs{A} = n-1$ take the form $A = [n] \setminus \{p\}$, the claim follows.
\end{proof} 

\begin{example}
    If $\Gamma$ is the graph below where $(12,13)$ is a leaf, 
    
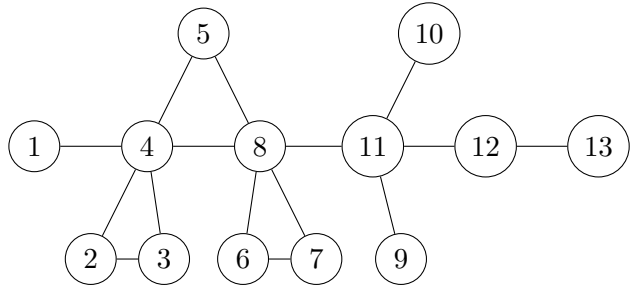
\begin{center} 
    \begin{tikzpicture}
            \begin{scope}[scale=1.5]
                \draw (0,0) node[draw,circle] (V1) {1};
                \draw (0.5,-1) node[draw,circle] (V2) {2};
                \draw (1.15,-1) node[draw,circle] (V3) {3};
                \draw (1,0) node[draw,circle] (V4) {4};
                \draw (1.5,1) node[draw,circle] (V5) {5};
                \draw (1.85,-1) node[draw,circle] (V6) {6};
                \draw (2.5,-1) node[draw,circle] (V7) {7};
                \draw (2,0) node[draw,circle] (V8) {8};
                \draw (3,0) node[draw,circle] (V10) {11};
                \draw (3.25,-1) node[draw,circle] (V9) {9};
                \draw (3.5,1) node[draw,circle] (V11) {10};
                \draw (4,0) node[draw,circle] (V12) {12};
                \draw (5,0) node[draw,circle] (V13) {13};

                \draw (V1)--(V4);
                \draw (V2)--(V4);
                \draw (V3)--(V4);
                \draw (V2)--(V3);
                \draw (V4)--(V5);
                \draw (V5)--(V8);
                \draw (V4)--(V8);
                \draw (V6)--(V7);
                \draw (V6)--(V8);
                \draw (V7)--(V8);
                \draw (V8)--(V10);
                \draw (V9)--(V10);
                \draw (V10)--(V11);
                \draw (V10)--(V12);
                \draw (V12)--(V13);
            \end{scope}
        \end{tikzpicture}
    \end{center}
    then Lemma \ref{lem:hard_relation} states that $\C\cf_\Gamma$ is identical to ${\ds \C\left\{\bar{\rho} \in \cf_\Gamma\left| \bar{\rho} \neq \bar{f}_B^{(12,13)} \text{ for any }B \subset [n] \right. \right\}}$. \par 
    Since $(1,4)$ is also a leaf, we have that $\C\cf_\Gamma$ is equal to ${\ds \C\left\{\bar{\rho} \in \cf_\Gamma\left| \bar{\rho} \neq \bar{f}_B^{(1,4)} \text{ for any }B \subset [n] \right. \right\}}$ as well. Note we cannot remove the coset splines for $(1,4)$ and $(12,13)$ from $\cf_\Gamma$ at the same time and maintain the $\C$-span, we have to pick a particular leaf to remove and stick with it.
\end{example}

\section{Natural Labels and Cliqued Graphs}\label{ssec:cliqued}
In this section we reduce the computation for arbitrary $\Gamma$ in two ways. First, we show that $\Gamma$ may be replaced by a cliqued graph (defined below) without altering $\splines{\Gamma}^1$. Second, we replace $\Gamma$ with a particular representative of the isomorphism class we call naturally labeled. Subsection \ref{ssec:lin_tech_lemmas} gives three technical lemmas on splines that hold for these constructions. \par 
First, if $\Gamma$ does not have a cut vertex, then it is $2$-connected. Thus $\lrep{\Gamma}_1$ is trivial and $\rrep{\Gamma}_1$ is the first degree piece of the graded regular representation by Theorem \ref{thm:kconn} and Corollary \ref{cor:conn_rrep}. So we may assume that $\Gamma$ has a cut vertex, in particular we may assume that $\Gamma$ has at least three vertices. \par 
A \emph{clique} is a subgraph isomorphic to a complete graph. Let $\Gamma$ by any (connected) graph on $[n]$. Call $\Gamma$ \emph{cliqued} if two vertices are connected by an edge in $\Gamma$ whenever there exists two vertex-disjoint paths between them.  Define 
\[\Gamma' = \left([n], E(\Gamma) \cup \{(i,j) \mid \; \text{exists two vertex-disjoint paths from $i$ to $j$ in $\Gamma$}\}\right).\]
 Now $\Gamma'$ is cliqued, and we call $\Gamma'$ the \emph{cliqued version} of $\Gamma$. By Lemma \ref{lem:kconn_better}, the first degree pieces of $\splines{\Gamma}$ and $\splines{\Gamma'}$ are equal. Therefore it suffices to consider cliqued graphs $\Gamma$ when proving results on the structure of $\splines{\Gamma}^1$.
\begin{example}\label{ex:cliqued}
    Below is an example of a graph $\Gamma$ and the cliqued graph $\Gamma'$ such that $\splines{\Gamma}^1 = \splines{\Gamma'}^1$.
    
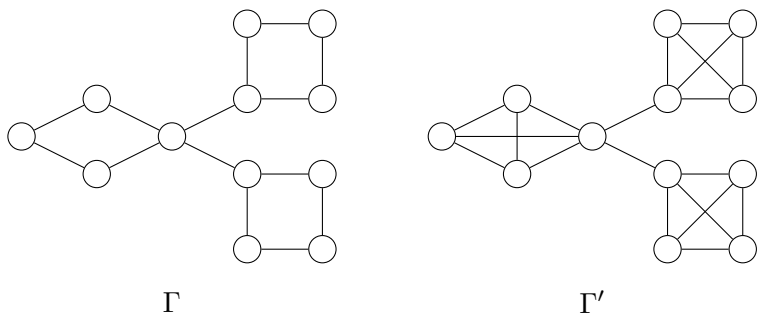
\begin{center}
    \begin{tikzpicture}
        \begin{scope}
            \draw (0,0) node[draw,circle] (V11) {};
            \draw (1,-0.5) node[draw,circle] (V12) {};
            \draw (1,0.5) node[draw,circle] (V13) {};
            \draw (2,0) node[draw,circle] (V14) {};
            \draw (3,0.5) node[draw,circle] (V21) {};
            \draw (3,1.5) node[draw,circle] (V22) {};
            \draw (4,1.5) node[draw,circle] (V23) {};
            \draw (4,0.5) node[draw,circle] (V24) {};
            \draw (3,-0.5) node[draw,circle] (V31) {};
            \draw (4,-0.5) node[draw,circle] (V32) {};
            \draw (4,-1.5) node[draw,circle] (V33) {};
            \draw (3,-1.5) node[draw,circle] (V34) {};

            \draw (2,-2.2) node (label) {$\Gamma$};
            
            \draw (V11)--(V12);
            \draw (V11)--(V13);
            \draw (V12)--(V14);
            \draw (V13)--(V14);
            \draw (V14)--(V21);
            \draw (V14)--(V31);

            \draw (V21)--(V22);
            \draw (V22)--(V23);
            \draw (V23)--(V24);
            \draw (V24)--(V21);

            \draw (V31)--(V32);
            \draw (V32)--(V33);
            \draw (V33)--(V34);
            \draw (V34)--(V31);
        \end{scope}
        \begin{scope}[xshift=2.2in]
            \draw (0,0) node[draw,circle] (V11) {};
            \draw (1,-0.5) node[draw,circle] (V12) {};
            \draw (1,0.5) node[draw,circle] (V13) {};
            \draw (2,0) node[draw,circle] (V14) {};
            \draw (3,0.5) node[draw,circle] (V21) {};
            \draw (3,1.5) node[draw,circle] (V22) {};
            \draw (4,1.5) node[draw,circle] (V23) {};
            \draw (4,0.5) node[draw,circle] (V24) {};
            \draw (3,-0.5) node[draw,circle] (V31) {};
            \draw (4,-0.5) node[draw,circle] (V32) {};
            \draw (4,-1.5) node[draw,circle] (V33) {};
            \draw (3,-1.5) node[draw,circle] (V34) {};
            
            \draw (2,-2.2) node (label) {$\Gamma'$};
            
            \draw (V11)--(V12);
            \draw (V11)--(V13);
            \draw (V12)--(V14);
            \draw (V13)--(V14);
            \draw (V11)--(V14);
            \draw (V12)--(V13);
            
            \draw (V14)--(V21);
            \draw (V14)--(V31);

            \draw (V21)--(V22);
            \draw (V22)--(V23);
            \draw (V23)--(V24);
            \draw (V24)--(V21);
            \draw (V21)--(V23);
            \draw (V22)--(V24);
            
            \draw (V31)--(V32);
            \draw (V32)--(V33);
            \draw (V33)--(V34);
            \draw (V34)--(V31);
            \draw (V31)--(V33);
            \draw (V32)--(V34);
        \end{scope}
    \end{tikzpicture}        
    \end{center}
\end{example}
A \emph{2-connected component} of a graph $\Gamma$ is a subgraph of $\Gamma$ that is $2$-connected. A \emph{block} is a maximal $2$-connected component. In a cliqued graph, every block is a clique, and the process of cliquing a graph simply converts every block to a clique.
\begin{definition}\label{def:block_cut}
    Let $\Gamma$ be a graph. The \emph{block-cut tree} of $\Gamma$ is the tree with vertex set 
    \[ \{v \mid v \vdash \Gamma\} \cup \{B \mid B\text{ is a block in } \Gamma\}\]
    consisting of cut vertices and blocks in $\Gamma$, and edge set $\{(v,B) \mid v \in B\}.$
\end{definition}
\begin{example}\label{ex:block_cut}
    The graph on the left below is $\Gamma$ from Example \ref{ex:block_cut} with the blocks and cut vertices labeled. The graph on the right below is the associated the block-cut tree. Note that the block cut tree for the cliqued version $\Gamma'$ of $\Gamma$ in Example \ref{ex:block_cut} would be the same.
    
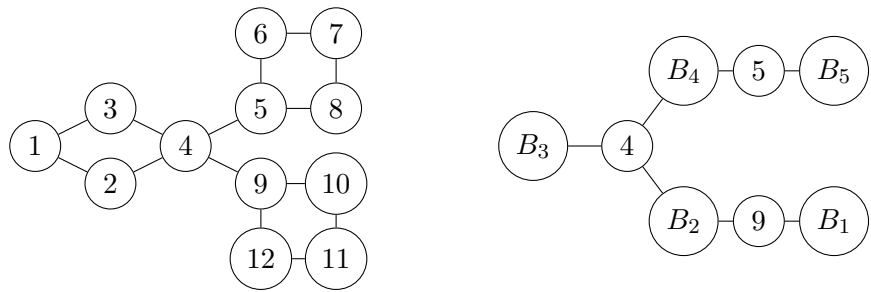
\begin{center}
        \begin{tikzpicture}
        \begin{scope}[scale=1]
            \draw (0,0) node[draw,circle] (V11) {$1$};
            \draw (1,-0.5) node[draw,circle] (V12) {$2$};
            \draw (1,0.5) node[draw,circle] (V13) {$3$};
            \draw (2,0) node[draw,circle] (V14) {$4$};
            \draw (3,0.5) node[draw,circle] (V21) {$5$};
            \draw (3,1.5) node[draw,circle] (V22) {$6$};
            \draw (4,1.5) node[draw,circle] (V23) {$7$};
            \draw (4,0.5) node[draw,circle] (V24) {$8$};
            \draw (3,-0.5) node[draw,circle] (V31) {$9$};
            \draw (4,-0.5) node[draw,circle] (V32) {$10$};
            \draw (4,-1.5) node[draw,circle] (V33) {$11$};
            \draw (3,-1.5) node[draw,circle] (V34) {$12$};

            \draw (V11)--(V12);
            \draw (V11)--(V13);
            \draw (V12)--(V14);
            \draw (V13)--(V14);
            \draw (V14)--(V21);
            \draw (V14)--(V31);

            \draw (V21)--(V22);
            \draw (V22)--(V23);
            \draw (V23)--(V24);
            \draw (V24)--(V21);

            \draw (V31)--(V32);
            \draw (V32)--(V33);
            \draw (V33)--(V34);
            \draw (V34)--(V31);
        \end{scope}
        \begin{scope}[xshift=3in, scale=1]
            \draw (-1,0) node[draw,circle] (V1) {$B_3$};
            \draw (0.25,0) node[draw,circle] (V2) {$4$};
            \draw (1,-1) node[draw,circle] (V3) {$B_2$};
            \draw (2,-1) node[draw,circle] (V4) {$9$};
            \draw (3,-1) node[draw,circle] (V5) {$B_1$};
            \draw (1,1) node[draw,circle] (V6) {$B_4$};
            \draw (2,1) node[draw,circle] (V7) {$5$};
            \draw (3,1) node[draw,circle] (V8) {$B_5$};

            \draw (V1)--(V2);
            \draw (V2)--(V3);
            \draw (V3)--(V4);
            \draw (V4)--(V5);
            \draw (V2)--(V6);
            \draw (V6)--(V7); 
            \draw (V7)--(V8);
        \end{scope}
        \end{tikzpicture}
    \end{center}
    Where \begin{multicols}{2}
           \begin{itemize}
        \item $B_1$ is the induced subgraph on vertices $\{9,10,11,12\}$
        \item $B_2$ is the induced subgraph on vertices $\{4,9\}$ 
        \item $B_3$ is the induced subgraph on vertices $\{1,2,3,4\}$
        \end{itemize}  
        \begin{itemize}
        \item $B_4$ is the induced subgraph on vertices $\{4,5\}$
        \item $B_5$ is the induced subgraph on vertices $\{5,6,7,8\}$
    \end{itemize}
    \end{multicols}
\end{example}
It is easy to reason that the block-cut tree is indeed a tree by arguing that it cannot contain cycles. In a block cut tree of a graph $\Gamma$, every leaf is a block in $\Gamma$ and not a cut vertex, and paths in the block-cut tree alternate between cut-vertices and blocks.
Since the block-cut tree ignores all of the internal structure of a $2$-connected component, the block-cut tree of a graph $\Gamma$ and the block-cut tree of its cliqued version $\Gamma'$ are isomorphic as graphs.\par 

We now use the block-cut tree to construct a particular representative of the isomorphism class of $\Gamma$. We will construct a bijection $\phi \colon [n] \to [n]$ to re-label the vertices of $\Gamma$. \par 
Choose a cut vertex $v \vdash \Gamma$ such that $v$ is adjacent to at most $1$ block that is not a leaf in the block-cut tree of $\Gamma$. One may obtain such a vertex $v$ by (1) removing all leaves from the block-cut tree of $\Gamma$ (which must be blocks), then (2) choosing a leaf from the tree that remains (which must be cut vertices). The vertices $5$ and $9$ satisfy this condition in Example \ref{ex:block_cut}.\par 
Let $B$ be the largest block that is also a leaf adjacent to $v$ in the block-cut tree of $\Gamma$. Since $B$ is a leaf in the block-cut tree, there exists only a single cut vertex in $B$, namely $v \in B$. The following is the algorithm that produces $\phi \colon [n] \to [n]$.
\begin{enumerate}
    \item Choose $i \in B$ so that $i \neq v$. Define $\phi(i) \coloneqq n$. Note $n$ is adjacent to at most one cut vertex of $\Gamma$ and is not itself a cut vertex of $\Gamma$.
    \item Define $\phi$ on the remaining vertices in $B$ as follows.  Define $\phi$ on $B \setminus \{i,v\}$ so that $d(j,i) < d(k,i)$ implies $\phi(j) > \phi(k)$ for all $j,k \in B \setminus\{i,v\}$. Let $\phi(v) \coloneqq n-\abs{B}+1$
    \item Define $\phi$ on the remaining vertices of $\Gamma$ as follows. Let $\phi$ be any bijection satisfying if $j,k \in \Gamma \setminus V(B)$, then $d(j,i) < d(k,i)$ implies $\phi(j) > \phi(k)$. 
\end{enumerate}
Now that we have a bijection $\phi \colon [n] \to [n]$, define a new graph $\Gamma'' \coloneqq ([n], \{(\phi(j),\phi(k)) \mid (j,k) \in E(\Gamma)\})$. This graph is clearly isomorphic to $\Gamma$. We call the graph $\Gamma''$ constructed in this manner \emph{naturally labeled}. \par 
\begin{example}\label{ex:natural_label}
    This example will construct a naturally labeled graph from the not-naturally labeled graph with $12$ vertices drawn below.
    
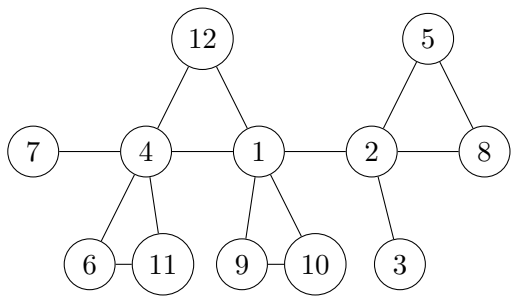
\begin{center} 
    \begin{tikzpicture}
            \begin{scope}[scale=1.5]
                \draw (0,0) node[draw,circle] (V1) {7};
                \draw (0.5,-1) node[draw,circle] (V2) {6};
                \draw (1.15,-1) node[draw,circle] (V3) {11};
                \draw (1,0) node[draw,circle] (V4) {4};
                \draw (1.5,1) node[draw,circle] (V5) {12};
                \draw (1.85,-1) node[draw,circle] (V6) {9};
                \draw (2.5,-1) node[draw,circle] (V7) {10};
                \draw (2,0) node[draw,circle] (V8) {1};
                \draw (3,0) node[draw,circle] (V10) {2};
                \draw (3.25,-1) node[draw,circle] (V9) {3};
                \draw (3.5,1) node[draw,circle] (V11) {5};
                \draw (4,0) node[draw,circle] (V12) {8};

                \draw (V1)--(V4);
                \draw (V2)--(V4);
                \draw (V3)--(V4);
                \draw (V2)--(V3);
                \draw (V4)--(V5);
                \draw (V5)--(V8);
                \draw (V4)--(V8);
                \draw (V6)--(V7);
                \draw (V6)--(V8);
                \draw (V7)--(V8);
                \draw (V8)--(V10);
                \draw (V9)--(V10);
                \draw (V10)--(V11);
                \draw (V10)--(V12);
                \draw (V11)--(V12);
            \end{scope}
        \end{tikzpicture}
    \end{center}
    This has block-cut tree 
    
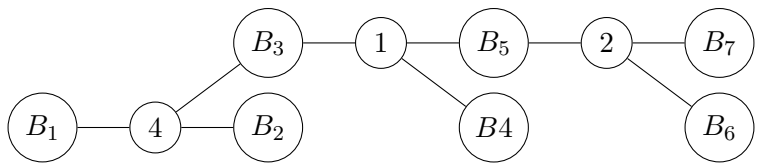
\begin{center} 
    \begin{tikzpicture}
            \begin{scope}[scale=1.5]
                \draw (0,0) node[draw,circle] (B1) {$B_1$};
                \draw (1,0) node[draw,circle] (v1) {$4$};
                \draw (2,0) node[draw,circle] (B2) {$B_2$};
                \draw (2,0.75) node [draw,circle] (B3) {$B_3$};
                \draw (3,0.75) node[draw,circle] (v2) {$1$};
                \draw (4,0) node[draw,circle] (B4) {$B4$};
                \draw (4,0.75) node[draw,circle] (B5) {$B_5$};
                \draw (5,0.75) node[draw,circle] (v3) {$2$};
                \draw (6,0) node[draw,circle] (B6) {$B_6$};
                \draw (6,0.75) node[draw,circle] (B7) {$B_7$};

                \draw (B1)--(v1);
                \draw (v1)--(B2);
                \draw (v1)--(B3);
                \draw (B3)--(v2);
                \draw (v2)--(B4);
                \draw (v2)--(B5);
                \draw (B5)--(v3);
                \draw (v3)--(B6);
                \draw (v3)--(B7);
            \end{scope}
        \end{tikzpicture}
    \end{center}
    with blocks \begin{multicols}{3}
           \begin{itemize}
        \item $B_1$ on  $\{4,7\}$
        \item $B_2$ on $\{4,6,11\}$ 
        \item $B_3$ on $\{1,4,12\}$
        \end{itemize}  
        \begin{itemize}
        \item $B_4$ on $\{1,9,10\}$
        \item $B_5$  on $\{1,2\}$
        \end{itemize} 
        \begin{itemize}
        \item $B_6$ on $\{2,3\}$
        \item $B_7$ on $\{2,5,8\}$
    \end{itemize}
    \end{multicols}
    To choose a vertex $i$ such that $\phi(i) = 12$, first we identify an appropriate cut vertex $v$. There are two cut vertices adjacent to only one non-leaf vertex in the block cut tree, $2$ and $4$. Let $v=2$. The vertex $2$ is adjacent to blocks $B_6$ and $B_7$ in the block-cut tree. Since $B_7$ is bigger than $B_6$, we know that either of $i=5$ or $i=8$ will work. We choose $i=8$ so $\phi(8) = 12$, which concludes step (1).
    
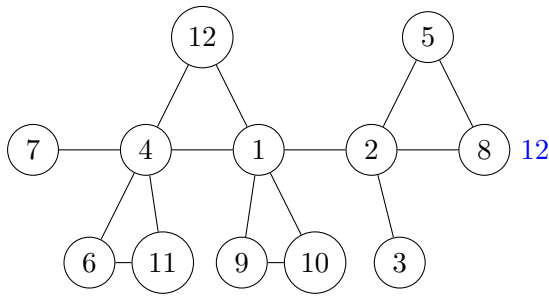
\begin{center} 
    \begin{tikzpicture}
            \begin{scope}[scale=1.5]
                \draw (0,0) node[draw,circle] (V1) {7};
                \draw (0.5,-1) node[draw,circle] (V2) {6};
                \draw (1.15,-1) node[draw,circle] (V3) {11};
                \draw (1,0) node[draw,circle] (V4) {4};
                \draw (1.5,1) node[draw,circle] (V5) {12};
                \draw (1.85,-1) node[draw,circle] (V6) {9};
                \draw (2.5,-1) node[draw,circle] (V7) {10};
                \draw (2,0) node[draw,circle] (V8) {1};
                \draw (3,0) node[draw,circle] (V10) {2};
                \draw (3.25,-1) node[draw,circle] (V9) {3};
                \draw (3.5,1) node[draw,circle] (V11) {5};
                \draw (4,0) node[draw,circle] (V12) {8};
                \draw (4.45,0) node (p12) {{\color{blue}12}};

                \draw (V1)--(V4);
                \draw (V2)--(V4);
                \draw (V3)--(V4);
                \draw (V2)--(V3);
                \draw (V4)--(V5);
                \draw (V5)--(V8);
                \draw (V4)--(V8);
                \draw (V6)--(V7);
                \draw (V6)--(V8);
                \draw (V7)--(V8);
                \draw (V8)--(V10);
                \draw (V9)--(V10);
                \draw (V10)--(V11);
                \draw (V10)--(V12);
                \draw (V11)--(V12);
            \end{scope}
        \end{tikzpicture}
    \end{center}
    Now the block $B_7$ has three vertices, which leaves no choice for defining $\phi$ on the remainder of $B_7$. So $\phi(5) = 11$ and $\phi(4) = 10$. This concludes step (2)
    
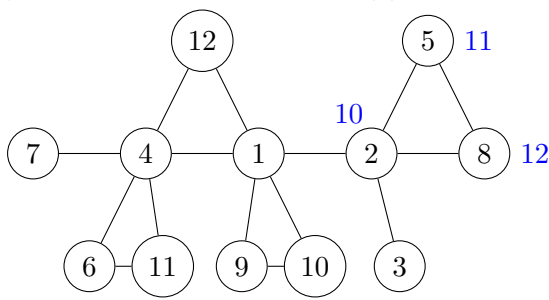
\begin{center} 
    \begin{tikzpicture}
            \begin{scope}[scale=1.5]
                \draw (0,0) node[draw,circle] (V1) {7};
                \draw (0.5,-1) node[draw,circle] (V2) {6};
                \draw (1.15,-1) node[draw,circle] (V3) {11};
                \draw (1,0) node[draw,circle] (V4) {4};
                \draw (1.5,1) node[draw,circle] (V5) {12};
                \draw (1.85,-1) node[draw,circle] (V6) {9};
                \draw (2.5,-1) node[draw,circle] (V7) {10};
                \draw (2,0) node[draw,circle] (V8) {1};
                \draw (3,0) node[draw,circle] (V10) {2};
                \draw (3.25,-1) node[draw,circle] (V9) {3};
                \draw (3.5,1) node[draw,circle] (V11) {5};
                \draw (4,0) node[draw,circle] (V12) {8};
                \draw (4.45,0) node (p12) {{\color{blue}12}};
                \draw (3.95,1) node (p11) {{\color{blue}11}};
                \draw (2.8,0.35) node (p10) {{\color{blue}10}};

                \draw (V1)--(V4);
                \draw (V2)--(V4);
                \draw (V3)--(V4);
                \draw (V2)--(V3);
                \draw (V4)--(V5);
                \draw (V5)--(V8);
                \draw (V4)--(V8);
                \draw (V6)--(V7);
                \draw (V6)--(V8);
                \draw (V7)--(V8);
                \draw (V8)--(V10);
                \draw (V9)--(V10);
                \draw (V10)--(V11);
                \draw (V10)--(V12);
                \draw (V11)--(V12);
            \end{scope}
        \end{tikzpicture}
    \end{center}
    Finally, we define the rest of $\phi$ based on distance from the vertex $i=8$, and replace the old graph with the naturally labeled one. One possible natural label is the following.
    
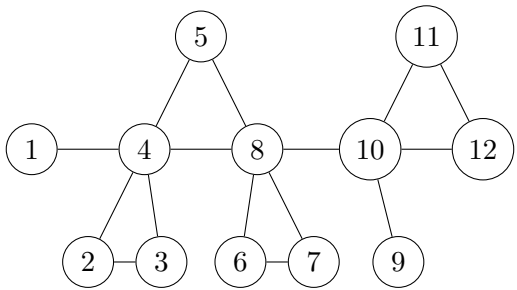
\begin{center} 
    \begin{tikzpicture}
            \begin{scope}[scale=1.5]
                \draw (0,0) node[draw,circle] (V1) {1};
                \draw (0.5,-1) node[draw,circle] (V2) {2};
                \draw (1.15,-1) node[draw,circle] (V3) {3};
                \draw (1,0) node[draw,circle] (V4) {4};
                \draw (1.5,1) node[draw,circle] (V5) {5};
                \draw (1.85,-1) node[draw,circle] (V6) {6};
                \draw (2.5,-1) node[draw,circle] (V7) {7};
                \draw (2,0) node[draw,circle] (V8) {8};
                \draw (3,0) node[draw,circle] (V10) {10};
                \draw (3.25,-1) node[draw,circle] (V9) {9};
                \draw (3.5,1) node[draw,circle] (V11) {11};
                \draw (4,0) node[draw,circle] (V12) {12};

                \draw (V1)--(V4);
                \draw (V2)--(V4);
                \draw (V3)--(V4);
                \draw (V2)--(V3);
                \draw (V4)--(V5);
                \draw (V5)--(V8);
                \draw (V4)--(V8);
                \draw (V6)--(V7);
                \draw (V6)--(V8);
                \draw (V7)--(V8);
                \draw (V8)--(V10);
                \draw (V9)--(V10);
                \draw (V10)--(V11);
                \draw (V10)--(V12);
                \draw (V11)--(V12);
            \end{scope}
        \end{tikzpicture}
    \end{center}
    We note how many choices were made along the way. In particular, the isomorphism class of a graph $\Gamma$ may have many different naturally labeled members.
\end{example}
We can use a natural label to more efficiently identify cut edges and connected components of a graph. The following definitions formalizes this.
\begin{definition}\label{def:dominant}
    Let $\Gamma$ be a naturally labeled graph. If $j \in \Gamma$ is a cut vertex then $i < j$ is \emph{$j$-dominant} if  $i$ is the maximal value vertex in a connected component of $\Gamma - j$ that does not contain the vertex $n$. We call such an $(i,j)$ a \emph{dominant pair}. A dominant pair $(i,j)$ is \emph{strongly dominant} if $(i,j)$ is a cut edge of $\Gamma$, denoted $(i,j) \gg \Gamma$. Otherwise, $(i,j)$ is \emph{weakly dominant}, denoted denoted $(i,j)>\Gamma$. 
    \end{definition} 
There are a four things to note about dominant pairs.
\begin{enumerate}
    \item Even if $(n-1,n)$ is a cut edge, $(n-1,n)$ is never a strongly dominant pair since $n$ is not a cut vertex by the definition of a natural label. However, every other cut edge in a naturally labeled graph $\Gamma$ is a strongly dominant pair, since the higher-labeled vertex in the cut edge must be a cut vertex.
    \item If $j \vdash \Gamma$, then all vertices larger than $j$ must be concentrated in one connected component of $\Gamma - j$. If $k$ is in a connected component of $\Gamma - j$ that does not contain the vertex $n$, then any path from $k$ to $n$ must pass through $j$ and so $d(k,n) > d(j,n)$ and thus $k<j$ by the definition of a natural label. In particular, the only connected component of $\Gamma - j$ whose maximal vertex is greater than $j$ is the one that contains $n$, so each connected component of $\Gamma - j$ that doesn't contain $n$ contains precisely one $j$-dominant vertex.
    \item Since a natural label is constructed by distance from $n$, for each cut vertex $j$, the maximal-labeled vertices in a connected component of $\Gamma - j$ that does not contain $n$ must be adjacent to $j$. In particular, dominant pairs are also edges.
    \item The lower vertex in the dominant pair uniquely determines the pair. In particular, if for contradiction we assume $(i,j)$ and $(i,k)$ are both dominant pairs, then $i$ and $k$ are in the same connected component of $\Gamma - j$, and $i<k$ and so $(i,j)$ is not a dominant pair.
\end{enumerate}
\begin{definition}\label{def:cut_decomp}
    Let $\Gamma$ be naturally labeled and fix $j \in \Gamma$. Let $\Cut_\Gamma(j) \coloneqq \{i \mid (i,j) > \Gamma \text{ or } (i,j) \gg \Gamma\}$, and $\cut_\Gamma(j) \coloneqq \abs{\Cut_\Gamma(j)}$. If $j$ is not a cut vertex, then $\Cut_\Gamma(j) = \emptyset$ and $\cut_\Gamma(j) = 0$.
    If $j$ is a cut vertex of $\Gamma$, then the \emph{cut decomposition} of $\Gamma-j$ is  \[\Gamma-j \coloneqq \Gamma_0^j \cup \bigcup_{i \in \Cut_{\Gamma}(j)}\Gamma_i^j,\]
Where $\Gamma_i^j$ the connected component of $\Gamma - j$ such that $i \in \Gamma_i^j$ and $\Gamma_0^j$ denotes the single connected component of $\Gamma - j$ where $n \in \Gamma_0^j$. 
\end{definition}
So if $j \vdash \Gamma$ is a cut vertex of $\Gamma$ and $k \in [n]$ such that $k > j$, then $k \in \Gamma_0^j$. When $\Gamma$ is obvious from context, we write $\Cut(j)$ and $\cut(j)$ without the subscripts.
\begin{remark}\label{rem:cutedge_component}
    If $(i,j)$ is a cut edge and $j$ is a cut vertex, then $\Gamma_i^j$ as in the cut decomposition of $\Gamma - j$ is one of the two connected components of $([n],E(\Gamma) \setminus \{(i,j)\})$. In particular, $\Gamma_i^j$ is one of the two valid choices for $G_s$ when defining $\bar{f}_A^{(i,j)}$ at the beginning of Section \ref{sec:lin_gen_rels}. From now on, even if $\Gamma$ is not naturally labeled and even if the cut edge is $(i,j) = (n-1,n)$, we will choose $G_s$ to be the connected component of $(V(\Gamma),E(\Gamma)-(i,j))$ that contains $i < j$, so that the notation always agrees with Definition \ref{def:cut_decomp}.
\end{remark} 

\begin{example}\label{linear_example_11}
    The following is the cliqued and naturally labeled graph $\Gamma$ from Example \ref{ex:natural_label}. We have labeled the strongly dominant pairs in double lines and the weakly dominant pairs in dashed lines:
    
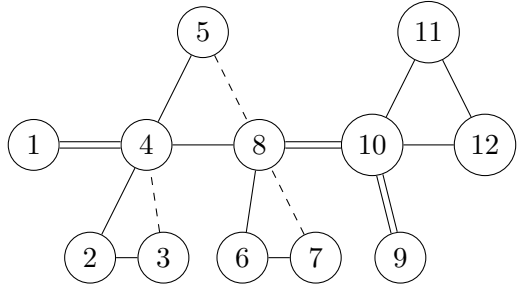
\begin{center}
        \begin{tikzpicture}
            \begin{scope}[scale=1.5]
                \draw (0,0) node[draw,circle] (V1) {$1$};
                \draw (0.5,-1) node[draw,circle] (V2) {$2$};
                \draw (1.15,-1) node[draw,circle] (V3) {$3$};
                \draw (1,0) node[draw,circle] (V4) {$4$};
                \draw (1.5,1) node[draw,circle] (V5) {$5$};
                \draw (1.85,-1) node[draw,circle] (V6) {$6$};
                \draw (2.5,-1) node[draw,circle] (V7) {$7$};
                \draw (2,0) node[draw,circle] (V8) {$8$};
                \draw (3,0) node[draw,circle] (V10) {$10$};
                \draw (3.25,-1) node[draw,circle] (V9) {$9$};
                \draw (3.5,1) node[draw,circle] (V11) {$11$};
                \draw (4,0) node[draw,circle] (V12) {$12$};

                \draw[double distance=2pt] (V1)--(V4);
                \draw (V2)--(V4);
                \draw[dashed] (V3)--(V4);
                \draw (V2)--(V3);
                \draw (V4)--(V5);
                \draw[dashed] (V5)--(V8);
                \draw (V4)--(V8);
                \draw (V6)--(V7);
                \draw (V6)--(V8);
                \draw[dashed] (V7)--(V8);
                \draw[double distance=2pt] (V8)--(V10);
                \draw[double distance=2pt] (V9)--(V10);
                \draw (V10)--(V11);
                \draw (V10)--(V12);
                \draw (V11)--(V12);
            \end{scope}
        \end{tikzpicture}
    \end{center}
The cut vertices are $\{4,8,10\}$, and $\cut_{\Gamma}(10) = \cut_{\Gamma}(8) =\cut_{\Gamma}(4) = 2$. The graph $\Gamma - 8$ is displayed below, with the cut decomposition labeled.
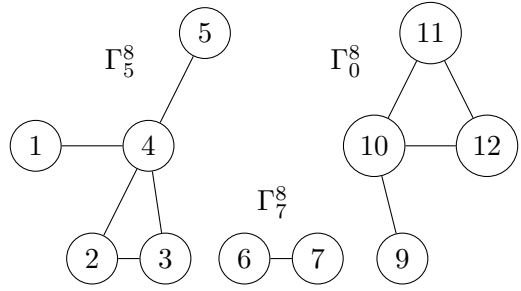
\begin{center}
        \begin{tikzpicture}
            \begin{scope}[scale=1.5]
                \draw (0,0) node[draw,circle] (V1) {$1$};
                \draw (0.5,-1) node[draw,circle] (V2) {$2$};
                \draw (1.15,-1) node[draw,circle] (V3) {$3$};
                \draw (1,0) node[draw,circle] (V4) {$4$};
                \draw (1.5,1) node[draw,circle] (V5) {$5$};
                \draw (1.85,-1) node[draw,circle] (V6) {$6$};
                \draw (2.5,-1) node[draw,circle] (V7) {$7$};
                \draw (3,0) node[draw,circle] (V10) {$10$};
                \draw (3.25,-1) node[draw,circle] (V9) {$9$};
                \draw (3.5,1) node[draw,circle] (V11) {$11$};
                \draw (4,0) node[draw,circle] (V12) {$12$};

                \draw (0.75,0.75) node (A) {$\Gamma_5^8$};
                \draw (2.1,-0.5) node (B) {$\Gamma_7^8$};
                \draw (2.75,0.75) node (C) {$\Gamma_0^8$};
                
                \draw (V1)--(V4);
                \draw (V2)--(V4);
                \draw (V3)--(V4);
                \draw (V2)--(V3);
                \draw (V4)--(V5);
                \draw (V6)--(V7);
                \draw (V9)--(V10);
                \draw (V10)--(V11);
                \draw (V10)--(V12);
                \draw (V11)--(V12);
            \end{scope}
        \end{tikzpicture}
    \end{center}
\end{example}
Lemma \ref{lem:natural_label} below summarizes some properties of a cliqued and naturally labeled graph.
\begin{lemma}\label{lem:natural_label}
    If $\Gamma$ is cliqued and naturally labelled, then 
    \begin{enumerate}
        \item if $i,j \in \adj(n)$ are both adjacent to the vertex $n$ in $\Gamma$, then $(i,j) \in E(\Gamma)$,
        \item if $n-1 \vdash \Gamma$ is a cut vertex of $\Gamma$, then at most one of the connected components $\Gamma_i^{n-1}$ for $i \in \Cut_{\Gamma}(n-1)$ in the cut decomposition of $\Gamma- (n-1)$  is \emph{not} a single vertex, 
        \item If $r,k \in \Gamma_i^j \cap \adj(j)$ are vertices both adjacent to $j\vdash \Gamma$ and in the same connected component of $\Gamma - j$, then $(r,k) \in E(\Gamma)$.
    \end{enumerate}
\end{lemma}
\begin{proof}
    Let $B_0$ be the block in $\Gamma$ that contains $n$.\par 
    (1) If $i,j \in \adj(n)$ then $i,j \in B_0$. Since $\Gamma$ is cliqued, $B_0$ must be a clique and so $(i,j) \in E(\Gamma)$. \par 
    (2) The vertex $n-1$ is a cut vertex of $\Gamma$ if and only if $n$ is a leaf, and $B_0$ is size $2$.  Since $\Gamma$ is naturally labeled, $n-1$ is adjacent to at most one block that is not a leaf in the block-cut tree of $\Gamma$, and $B_0$ is of maximal size among those leaves in the block-cut tree. Thus the blocks adjacent to $n-1$ in the block-cut tree of $\Gamma$ are either not a leaf in the block-cut tree (of which there can only be one), or a leaf in the block-cut tree and size no greater than $2$.\par  
    (3) There exists a path $(r,j,k)$ in $\Gamma$, and another path from $r$ to $k$ in $\Gamma_i^j$, which doesn't contain the vertex $j$. Since $\Gamma$ is cliqued and we know there are two vertex-disjoint paths from $r$ to $k$ in $\Gamma$, it follows that $(r,k) \in E(\Gamma)$.
\end{proof}
    We use Lemma \ref{lem:natural_label} to categorize cliqued and naturally labelled graphs in to three types, based on the structure of $\Gamma$ near the vertex $n$.
    \begin{lemma}\label{lem:graph_type}
        If $\Gamma$ is cliqued and naturally labeled, then it falls in to one of the following three types:
    \begin{enumerate}
        \item[(A)] The edge $(n-1,n)$ is a cut edge of $\Gamma$, and at most one component of $\Gamma - (n-1)$ is \emph{not} an isolated vertex.
        \item[(B)] The vertex $n-2$ is a cut vertex, and the vertices $\{n-2,n-1,n\}$ form a block in $\Gamma$.
        \item[(C)] None of the vertices $\{n,n-1,n-2\}$ are cut vertices and $\{n-2,n-1,n\}$ form a clique in $\Gamma$.
    \end{enumerate}
    \end{lemma}
    \begin{proof}
        Categorize $\Gamma$ by the size of the neighborhood $\adj(n)$.  For every $\Gamma$ exactly one of the following is true: $\abs{\adj{(n)}} = 1$, $\abs{\adj{(n)}} = 2$, or $\abs{\adj(n)} > 2$. Note that, since every block is a clique and $n$ is not a cut vertex, the block $B$ containing $n$ has vertices $N(n) \cup \{n\}$. \par 
        If $\abs{\adj{(n)}} = 1$, then $n-1$ is a cut vertex and $\Gamma$ is type A. The rest of the claim for type A is Lemma \ref{lem:natural_label}(2). \par 
        Suppose $\abs{\adj{(n)}} = 2$. Since $\Gamma$ is naturally labeled, $n-2$ must be a cut vertex. We also have $(n-2,n-1) \in E(\Gamma)$ by Lemma \ref{lem:natural_label}(1). \par 
        Finally suppose $\abs{\adj{(n)}} > 2$. Since $\Gamma$ is naturally labeled, none of $\{n-2,n-1,n\}$ is a cut vertex. Once again $(n-2,n-1) \in E(\Gamma)$ by Lemma \ref{lem:natural_label}(1). \par 
    \end{proof}

Visually, Lemma \ref{lem:graph_type} says that if $\Gamma$ is cliqued and naturally labeled, then $\Gamma$ can be represented diagramatically in one of the following three ways: 
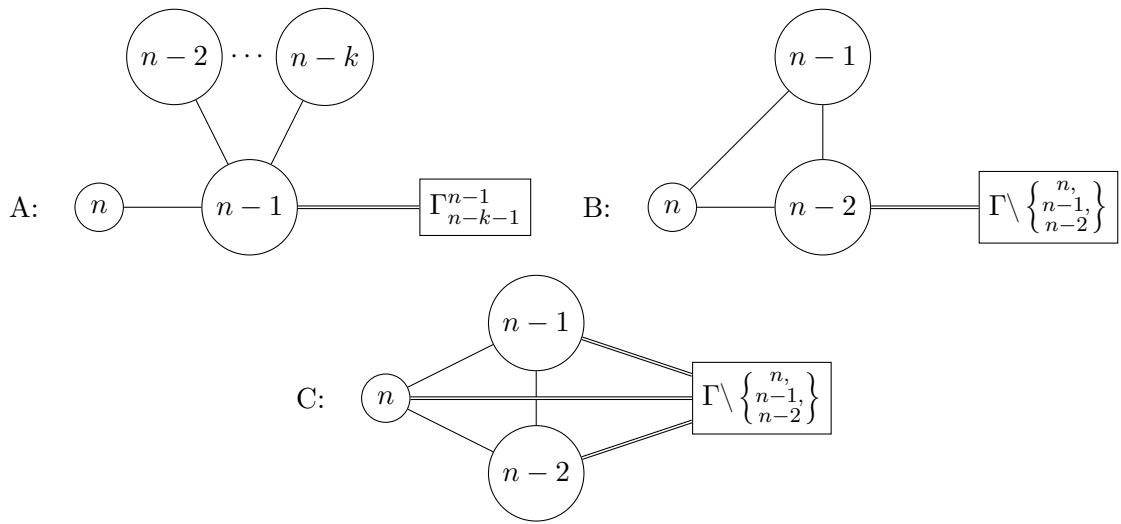
\begin{center}
    \begin{tikzpicture}
        \begin{scope}
            \draw (0,0) node[draw,circle] (V1) {$n$};
            \draw (2,0) node[draw,circle] (V2) {$n-1$};
            \draw (5,0) node[draw,rectangle] (V3) {$\Gamma_{n-k-1}^{n-1}$};
            \draw (-1,0) node (type) {A:};

            \draw (1,2) node[draw,circle] (V4) {$n-2$};
            \draw (2,2) node (dots) {$\cdots$};
            \draw (3,2) node[draw,circle] (V5) {$n-k$};

            \draw (V1)--(V2);
            \draw (V2)--(V4);
            \draw (V2)--(V5);
            \draw[double] (V2)--(V3);
        \end{scope}
        \begin{scope}[xshift = 3in]
            \draw (0,0) node[draw,circle] (V1) {$n$};
            \draw (2,2) node[draw,circle] (V2) {$n-1$};
            \draw (2,0) node[draw,circle] (V3) {$n-2$};
            \draw (5,0) node[draw,rectangle] (V4) {$\Gamma \setminus \left\{\substack{n,\\n-1,\\n-2}\right\}$};            
            \draw (-1,0) node (type) {B:};

            \draw (V1)--(V2);
            \draw (V1)--(V3);
            \draw (V2)--(V3);
            \draw[double] (V3)--(V4);
        \end{scope}
        \begin{scope}[xshift = 1.5in, yshift = -1in]
            \draw (0,0) node[draw,circle] (V1) {$n$};
            \draw (2,1) node[draw,circle] (V2) {$n-1$};
            \draw (2,-1) node[draw,circle] (V3) {$n-2$};
            \draw (5,0) node[draw,rectangle] (V4) {$\Gamma \setminus \left\{\substack{n,\\n-1,\\n-2}\right\}$};            
            \draw (-1,0) node (type) {C:};

            \draw (V1)--(V2);
            \draw (V1)--(V3);
            \draw (V2)--(V3);
            \draw[double] (V2)--(V4);
            \draw[double] (V3)--(V4);
            \draw[double] (V1)--(V4);
        \end{scope}
    \end{tikzpicture}
\end{center} 
In the diagram for type A above, $\cut_{\Gamma}(n-1) = k$. We remark that that in type $B$, the induced subgraph $\Gamma \setminus \{n-2,n-2,n\}$ may be disconnected (such as it would be for the graph in Example \ref{linear_example_11}). On the other hand, in type $C$ the vertex $n-3$ must be in the same block as $n$, $n-1$, and $n-2$, so the induced subgraph $\Gamma \setminus\{n-2,n-1,n\}$ is actually connected. \par 
\begin{example}\label{ex:linear_example_2}
    The graph $\Gamma$ from Example \ref{linear_example_11} is type B. The following is a naturally labeled type A graph.
    
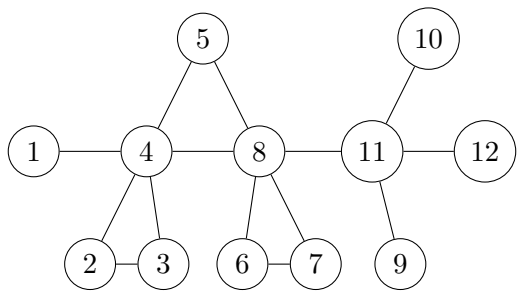
\begin{center} 
    \begin{tikzpicture}
            \begin{scope}[scale=1.5]
                \draw (0,0) node[draw,circle] (V1) {1};
                \draw (0.5,-1) node[draw,circle] (V2) {2};
                \draw (1.15,-1) node[draw,circle] (V3) {3};
                \draw (1,0) node[draw,circle] (V4) {4};
                \draw (1.5,1) node[draw,circle] (V5) {5};
                \draw (1.85,-1) node[draw,circle] (V6) {6};
                \draw (2.5,-1) node[draw,circle] (V7) {7};
                \draw (2,0) node[draw,circle] (V8) {8};
                \draw (3,0) node[draw,circle] (V10) {11};
                \draw (3.25,-1) node[draw,circle] (V9) {9};
                \draw (3.5,1) node[draw,circle] (V11) {10};
                \draw (4,0) node[draw,circle] (V12) {12};

                \draw (V1)--(V4);
                \draw (V2)--(V4);
                \draw (V3)--(V4);
                \draw (V2)--(V3);
                \draw (V4)--(V5);
                \draw (V5)--(V8);
                \draw (V4)--(V8);
                \draw (V6)--(V7);
                \draw (V6)--(V8);
                \draw (V7)--(V8);
                \draw (V8)--(V10);
                \draw (V9)--(V10);
                \draw (V10)--(V11);
                \draw (V10)--(V12);
            \end{scope}
        \end{tikzpicture}
    \end{center}
    A different natural label on the same graph, such as the one below, can have a different classification. The following naturally labeled graph is in the same isomorphism class as the previous, but is type B.   
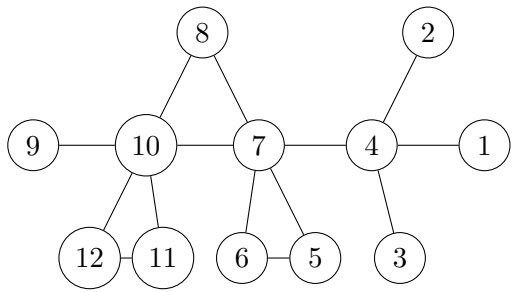
    \begin{center} 
    \begin{tikzpicture}
            \begin{scope}[scale=1.5]
                \draw (0,0) node[draw,circle] (V1) {9};
                \draw (0.5,-1) node[draw,circle] (V2) {12};
                \draw (1.15,-1) node[draw,circle] (V3) {11};
                \draw (1,0) node[draw,circle] (V4) {10};
                \draw (1.5,1) node[draw,circle] (V5) {8};
                \draw (1.85,-1) node[draw,circle] (V6) {6};
                \draw (2.5,-1) node[draw,circle] (V7) {5};
                \draw (2,0) node[draw,circle] (V8) {7};
                \draw (3,0) node[draw,circle] (V10) {4};
                \draw (3.25,-1) node[draw,circle] (V9) {3};
                \draw (3.5,1) node[draw,circle] (V11) {2};
                \draw (4,0) node[draw,circle] (V12) {1};

                \draw (V1)--(V4);
                \draw (V2)--(V4);
                \draw (V3)--(V4);
                \draw (V2)--(V3);
                \draw (V4)--(V5);
                \draw (V5)--(V8);
                \draw (V4)--(V8);
                \draw (V6)--(V7);
                \draw (V6)--(V8);
                \draw (V7)--(V8);
                \draw (V8)--(V10);
                \draw (V9)--(V10);
                \draw (V10)--(V11);
                \draw (V10)--(V12);
            \end{scope}
        \end{tikzpicture}
    \end{center}
\end{example}

\subsection{The Spanning Set Revisited}\label{ssec:lin_gens_revisit}
A natural labeling provides a more convenient indexing for the splines in $\cy_\Gamma$ and $\cc_\Gamma$, using the cut decomposition from Definition \ref{def:cut_decomp}. In particular, we write 
\[\bar{y}_{i,k}^j \coloneqq \bar{y}_{\Gamma_i^j,k}^j\]
so that 
\[\cy_\Gamma = \left\{\left.\bar{y}_{i,k}^j \right| j \vdash \Gamma,\; i \in \Cut_{\Gamma}(j)\cup\{0\},\; k \in [n]\right\}.\]
Additionally, while the notation for individual splines $\bar{f}_A^s \in \cc_\Gamma$ does not change we note that, as stated in Remark \ref{rem:cutedge_component}, we choose $G_{s}$ for $s=(i<j)$ to be equal to $\Gamma_i^j$. \par 
We collect the important naturally-labeled versions of Lemmas \ref{lem:strong_y_spline}, \ref{lem:linear_relations}, and \ref{lem:hard_relation} below in Proposition \ref{prop:hard_linear_relations}.
\begin{proposition}\label{prop:hard_linear_relations}
    Let $\Gamma$ be naturally labeled, and let 
    \[
    \cb_\Gamma \coloneqq \left\{\bar{t}_i \mid i \in [n]\right\} \cup \left\{\bar{x}_i \mid i \in [n]\right\} \cup \left\{\bar{f}_A^s \left| \begin{matrix} s=(i,j) \gg \Gamma,\\ \abs{A} = \abs{\Gamma_i^j}\end{matrix}\right.\right\} \cup \left\{\bar{y}_{i,k}^j \left| \begin{matrix}(i,j) > \Gamma,\\ k \in [n] \end{matrix}\right.\right\}.
    \]
    Then the following hold:
    \begin{enumerate}
        \item If $(i,j) \gg \Gamma$, then $\bar{y}_{i,k}^j \in \C \cb_\Gamma$ for all $k \in [n]$.
        \item If $j \vdash \Gamma$, then $\bar{y}_{0,k}^j \in \C \cb_\Gamma$ for all $k \in [n]$.
        \item If $(n-1,n)$ is a cut edge of $\Gamma$, then $\bar{f}_{A}^{(n-1,n)} \in \C\cb_\Gamma$ for all $A \subset [n]$ where $\abs{A} = n-1$.
    \end{enumerate}
    In particular, $\C\cb_\Gamma = \C\cf_\Gamma$.
\end{proposition}
\begin{proof}
    The first relation (1) is exactly Lemma \ref{lem:strong_y_spline}. \par 
    The second relation (2) follows from (1) together with Lemma \ref{lem:linear_relations}(4). \par 
    The third relation (3) is Lemma \ref{lem:hard_relation}, applied to the cut edge $(n-1,n)$.
\end{proof}

\subsection{Technical Lemmas}\label{ssec:lin_tech_lemmas}
This subsection contains three lemmas that are used within the proof of Theorem \ref{thm:linear_spanning}.\par 
Let $S_n^i \coloneq\{w \in S_n \mid w(i) = n\}$ be a left coset of $S_{n-1}$ in $S_n$. The first Lemma \ref{lem:technical_n1} establishes what values a linear spline $\bar{\rho}$ may take on $S_n^{n-1}$ if $\bar{\rho}$ is not supported on $S_n^n = S_{n-1}$.
\begin{lemma}\label{lem:technical_n1}
    Let $\Gamma$ be naturally labeled, and $\bar{\rho} \in \splines{\Gamma}^1$ where $\bar{\rho} \equiv 0$ on $S_n^n$. If $w,v \in S_n^{n-1}$, then $\bar{\rho}(w) = c_w\left(t_n-t_{w(n)}\right)$ and $\bar{\rho}(v) = c_v\left(t_n-t_{v(n)}\right)$ for some $c_w,c_v \in \C$. Furthermore, if $w(n) = v(n)$, or $\Gamma$ is type $B/C$, then $c_w=c_v$.
\end{lemma}
\begin{proof}
    First, since $\Gamma$ is naturally labeled it follows that $(n-1,n) \in E(\Gamma)$. We will show the first part of the claim for $w \in S_n^{n-1}$, and the same will hold for $v \in S_n^{n-1}$. If $w \in S_n^{n-1}$, then $w(n-1,n) \in S_n^n$, and $\mathcal{L}(w,w(n-1,n)) = \llangle t_{w(n-1)} - t_{w(n)}\rrangle = \llangle t_n - t_{w(n)} \rrangle$. Since $\bar{\rho}$ is linear and $\bar{\rho}(w(n-1,n)) = 0$, the first part of the claim follows. \par 
    Now we prove the second part of the claim. First we assume $w(n) = v(n)$. Two permutations $w,v \in S_n^{n-1}$ have the property $w(n) = v(n)$ if and only if  $v \in wS_{n-2}$. The claim will follow if $c_w=c_v$ for $v = w(r,s)$ whenever $(r,s) \in S_{n-2}$. Let $\{r,s\} \subset [n-2]$. As $\Gamma$ is naturally labeled and so $n$ is not a cut vertex of $\Gamma$, there exists a simple path $(r_0,r_1,...,r_m)$ in $\Gamma$ from $r=r_0$ to $s=r_m$ that does not contain $n$, and by Lemma \ref{lem:diff_ideal} \[\bar{\rho}(w) - \bar{\rho}(w(r,s)) \in \llangle t_{w(r_i)} - t_{w(r_{i-1})} \mid i \in [m]\rrangle.\] In particular, 
        \[ (c_w-c_{w(r,s)})t_n - (c_w - c_{w(r,s)}) t_{w(n)} \in \llangle t_{w(r_i)} - t_{w(r_{i-1})} \mid i \in [m]\rrangle.\]
        The monomial $t_{w(n)}$ does not appear in $\{t_{w(r_i)} - t_{w(r_{i-1})} \mid i \in [m]\}$, and thus $c_{w} = c_{w(r,s)}$. Since the transpositions $(r,s)$ generate $S_{n-2}$, the claim follows. \par 
        Now we prove the claim if $\Gamma$ is type $B/C$. If $\Gamma$ is type $B$ or type $C$ then $\Gamma-(n-1)$ is connected. It suffices to prove $c_w = c_{w(r,s)}$ for $n-1 \notin \{r,s\}$. Since $\Gamma - (n-1)$ is connected, there exists a path $(r_0,...,r_m)$ from $r=r_0$ to $s=r_m$ in $\Gamma$ that does not visit the vertex $n-1$. So
      \[ (c_w-c_{w(r,s)})t_n - c_wt_{w(n)} - c_{w(r,s)} t_{w(r,s)(n)} \in \llangle t_{w(r_i)} - t_{w(r_{i-1})} \mid i \in [m]\rrangle.\]
      Now $t_n = t_{w(n-1)}$ never appears in $\left\{ t_{w(r_i)} - t_{w(r_{i-1})} \mid i \in [m]\right\}$, and thus $c_w=c_v$.
\end{proof}
The second two lemmas assume that $(i,j)$ is a dominant pair in $\Gamma$, and establish what values a linear spline $\bar{\rho}$ may take on $S_n^i$ if $\supp(\bar{\rho}) \cap S_n^j = \emptyset$. Lemma \ref{lem:technical_cutedge} below assumes $(i,j)$ is strongly dominant, and relates $\bar{\rho}(w)$ to $\bar{\rho}(v)$ if $w$ and $v$ are in the same coset $S_n^i$ and they are also in the same coset of the reflection subgroup generated by the transpositions $E(\Gamma) \setminus\{(i,j)\}$.
\begin{lemma}\label{lem:technical_cutedge}
    Let $\Gamma$ be cliqued and naturally labelled, and say $(i,j) \gg \Gamma$. Let $\bar{\rho} \in \splines{\Gamma}^1$ where $\bar{\rho} \equiv 0$ on $S_n^j$. If $w,v \in S_n^{i}$, then $\bar{\rho}(w) = c_w\left(t_n-t_{w(j)}\right)$ and $\bar{\rho}(v) = c_v\left(t_n-t_{v(j)}\right)$ for some $c_w,c_v \in \C$. Furthermore, if $v \in w\langle E(\Gamma)\setminus (i,j)\rangle$, then $c_w=c_v$.
\end{lemma}
\begin{proof}
    First, since $\Gamma$ is naturally labeled it follows that $(i,j) \in E(\Gamma)$. We will show the first part of the claim for $w \in S_n^i$, and the same will hold for $v \in S_n^i$. If $w \in S_n^i$, then $w(i,j) \in S_n^j$, and $\mathcal{L}(w,w(i,j)) = \llangle t_{w(i)} - t_{w(j)}\rrangle = \llangle t_n - t_{w(j)} \rrangle$. Since $\bar{\rho}$ is linear and $\bar{\rho}(w(i,j)) = 0$, the first part of the claim follows from the definition of a spline on $\cg_\Gamma$. \par 
    The reflection subgroup $\llangle E(\Gamma) \setminus\{(i,j)\}\rrangle$ is also generated by the transpositions $\{(r,s) \mid \{r,s\} \subset V(\Gamma_i^j)\} \cup \{(p,q) \mid \{p,q\} \subset [n] \setminus V(\Gamma_i^j)\}$, as this set contains $E(\Gamma) \setminus\{i,j\}$. We will show that, for those generating transpositions, $c_w = c_{w(r,s)}$ and $c_w = c_{w(p,q)}$. \par 
    First we show $c_w = c_{w(r,s)}$ for $\{r,s\} \subset V(\Gamma_i^j)$. Let $(r_0,...,r_m)$ be a path in $\Gamma_i^j$ from $r=r_0$ to $r_m=s$, then by Lemma \ref{lem:diff_ideal} \[\bar{\rho}(w) - \bar{\rho}(w(r,s)) \in \llangle t_{w(r_k)} - t_{w(r_{k-1})} \mid k \in [m]\rrangle.\] In particular, 
        \[ (c_w-c_{w(r,s)})t_n - (c_w - c_{w(r,s)}) t_{w(j)} \in \llangle t_{w(r_k)} - t_{w(r_{k-1})} \mid k \in [m]\rrangle.\]
        Since $j$ is not in the path $(r_0,...,r_m)$, the monomial $t_{w(j)}$ does not appear in $\{t_{w(r_k)} - t_{w(r_{k-1})} \mid k \in [m]\}$, and thus $c_{w} = c_{w(r,s)}$.\par 
        Now let $\{p,q\} \subset [n] \setminus V(\Gamma_i^j)$. Since $(i,j)$ is a cut edge, the induced subgraph of $\Gamma$ with vertex set $[n] \setminus V(\Gamma_i^j)$ is connected. Let $(p_0,...,p_m)$ be a path from $p = p_0$ to $q=p_m$ in $\Gamma$ that does \emph{not} contain $i$. 
        By Lemma \ref{lem:diff_ideal} \[\bar{\rho}(w) - \bar{\rho}(w(p,q)) \in \llangle t_{w(p_k)} - t_{w(p_{k-1})} \mid k \in [m]\rrangle.\] In particular, 
        \[ (c_w-c_{w(p,q)})t_n - c_wt_{w(j)} + c_{w(p,q)} t_{w(p,q)(j)} \in \llangle t_{w(p_k)} - t_{w(p_{k-1})} \mid k \in [m]\rrangle.\]
        Since $i$ is not in the path $(p_0,...,p_m)$, the monomial $t_n = t_{w(i)}$ does not appear in $\{t_{w(p_k)} - t_{w(p_{k-1})} \mid k \in [m]\}$, and thus $c_w=c_{w(p,q)}$.\par 
        Since the reflection subgroup $\langle E(\Gamma)\setminus(i,j)\rangle$ is generated by $\{(r,s) \mid \{r,s\} \subset \Gamma_i^j\} \cup \{(p,q) \mid \{p,q\} \in [n]\setminus \Gamma_i^j\}$, the claim follows.
\end{proof}
Lemma \ref{lem:technical_cutvert} below assumes that $(i,j)$ is weakly dominant, then relates $\bar{\rho}(w)$ and $\bar{\rho}(v)$ if $w$ and $v$ are in the same coset $S_n^i$.
\begin{lemma}\label{lem:technical_cutvert}
    Let $\Gamma$ be cliqued and naturally labelled, and say $(i,j)>\Gamma$. Let $\bar{\rho} \in \splines{\Gamma}^1$ where $\bar{\rho} \equiv 0$ on $S_n^j$. If $w,v \in S_n^{i}$, then $\bar{\rho}(w) = c_w\left(t_n-t_{w(j)}\right)$ and $\bar{\rho}(v) = c_v\left(t_n-t_{v(j)}\right)$ for some $c_w,c_v \in \C$. Furthermore, $c_w=c_v$.
\end{lemma}
\begin{proof}
    The proof of first part of this claim is identical to the first part of the proof of Lemma \ref{lem:technical_cutedge}. \par 
    Since $(i,j)$ is not a cut edge and $\Gamma$ is naturally labeled, there exists $(k,j) \in E(\Gamma)$ with $k<i<j$ and $k \in \Gamma_i^j$. By Lemma \ref{lem:natural_label}(3), $(i,k) \in E(\Gamma)$. If $u \in S_n^k$ then $\bar{\rho}(u) = c_u\left(t_n-t_{u(j)}\right)$ for the same reason that $\bar{\rho}(w)$ and $\bar{\rho}(v)$ take this form. We will prove the slightly stronger claim that $c_w=c_v$ for all $w,v \in S_n^i \sqcup S_n^k$. We proceed for now assuming that the induced subgraph of $\cg_\Gamma$ with vertex set $S_n^i \sqcup S_n^k$ is connected, and will verify that this assumption holds afterwards. \par 
    If $S_n^i \sqcup S_n^k$ is connected, it will suffice to check if $c_w=c_v$ for adjacent elements $w,v \in S_n^i \sqcup S_n^k$. We check edges in two cases: those within $S_n^i$ (resp. $S_n^k$), and the edges $(w,w(i,k))$ between $S_n^i$ and $S_n^k$.\par  
         If $(w,w(p,q))$ is an edge in $\cg_\Gamma$ between elements of $S_n^i$, then $i \notin\{p,q\}$ and \[\bar{\rho}(w) - \bar{\rho}(w(p,q)) = (c_w-c_{w(p,q)})t_n - c_wt_{w(j)} + t_{w(p,q)(j)} \in \llangle t_{w(q)}-t_{w(p)}\rrangle\]
         Since $n \notin \{w(p),w(q)\}$, it follows $c_w=c_{w(p,q)}$. The same logic holds for edges in $\cg_\Gamma$ between two elements of $S_n^k$. \par 
         For edges $(w,w(i,k))$ in $\cg_\Gamma$ between $w \in S_n^i$ and $w(i,k) \in S_n^k$, compute 
         \[\bar{\rho}(w) - \bar{\rho}(w(i,k)) = c_wt_n-c_{w(i,k)}t_{w(k)} - (c_w-c_{w(i,k)})t_{w(j)} \in \llangle t_{w(i)}-t_{w(k)}\rrangle \] 
         Since $w(j) \notin \{w(i),w(k)\}$, it follows that $c_w=c_{w(i,k)}$.\par 
         If $S_n^i\sqcup S_n^k$ is connected, and equality holds on every edge, it follows that $c_w=c_v$ for all $w,v \in S_n^i$.\par 
        Now we will prove that the induced subgraph of $\cg_\Gamma$ with vertex set $S_n^i \sqcup S_n^k$ is connected. Since $\Gamma$ is cliqued, $(i,k) \in E(\Gamma)$. In particular, if $w \in S_n^i$ then $w$ is connected in $\cg_\Gamma$ directly to $w(i,k)$ in $S_n^k$.\par 
         Let $\{r,s\} \subset [n] \setminus\{i\}$. We prove in three cases that for all $r,s \neq i$, the permutation $w \in S_n^i$ is connected to $w(r,s) \in S_n^i$ within the induced subgraph of $\cg_\Gamma$ with vertex set $S_n^i\sqcup S_n^k$. It will follow by symmetry (replace $i$ with $k$) and that the induced subgraph of $\cg_\Gamma$ with vertex set $S_n^i\sqcup S_n^k$ is connected. The three cases are (i) there exists a simple path from $r$ to $s$ in $\Gamma$ that does not visit the vertex $i$, (ii) simple paths from $r$ to $s$ in $\Gamma$ must visit $i$, but need not visit $k$, and (iii) simple paths from $r$ to $s$ must visit the vertex $i$ and the vertex $k$. \par
         (i) If there exists a path $(p_0,...,p_\ell)$ in $\Gamma - i$ from $r=p_0$ to $s=p_\ell$, then $w$ is connected to $w(r,s)$ via only elements in $S_n^i$ since
\[ w(r,s) = (p_0,p_1) \cdots (p_{\ell-1},p_\ell) \cdots (p_0,p_1).\]         
         We will use this path computation implicitly in (ii) and (iii).\par 
         (ii) If there exists a path $(r,...,i,...,s)$ from $r$ to $s$ in $\Gamma$ containing $i$ but not $k$, then $r$, $i$, and $s$ are in the same connected component in $\Gamma - k$. Consider
         \[w(r,s) = w(i,k)(i,r)(i,s)(i,r)(i,k).\]
         This sequence of transpositions gives a path in $S_n^i \sqcup S_n^k$ from $w$ to $w(r,s).$ Below is a diagram that shows how each right multiplication moves between $S_n^i$ and $S_n^k$. 
         \begin{center}
                 \begin{tikzpicture}[scale=2.0]
                 \draw (0,0) node (V1) {$S_n^i$};
                 \draw (1.5,0) node (V2) {$S_n^k$};
                 \draw (3,0) node (V3) {$S_n^i$};

                 \path[->] (V1) edge node[below] {$(i,k)$} (V2) 
                   (V2) edge[out=45,in=125,looseness=2.5] node[above] {$(i,r)(i,s)(i,r)$} (V2) 
                   (V2) edge node[below] {$(i,k)$} (V3);
             \end{tikzpicture}
         \end{center}\par 
         (iii) If both $i$ and $k$ must be in a simple path from $r$ to $s$, it suffices to assume this path takes the form $(r,...,i,k,...,s)$. In particular, the first piece $(r,...,i)$ is a path in $\Gamma - k$ and the second piece $(k,...,s)$ is a path in $\Gamma - i$. Consider 
         \[
         w(r,s) = w(i,k)(r,j)(i,k)(j,k)(k,s)(j,k)(i,k)(r,j)(i,k).
         \]
         This sequence of transpositions gives a path from $w$ to $w(r,s)$ in $S_n^i\sqcup S_n^k$. Below is a diagram detailing how each right multiplication moves between $S_n^i$ and $S_n^k$.
         \begin{center}
             \begin{tikzpicture}[scale=2.0]
                 \draw (0,0) node (V1) {$S_n^i$};
                 \draw (1,0) node (V2) {$S_n^k$};
                 \draw (2.5,0) node (V3) {$S_n^i$};
                 \draw (4,0) node (V4) {$S_n^k$};
                 \draw (5,0) node (V5) {$S_n^i$};

                 \path[->] (V1) edge node[below] {$(i,k)$} (V2) 
                   (V2) edge[out=45,in=125,looseness=2.5] node[above] {$(r,j)$} (V2) 
                   (V2) edge node[below] {$(i,k)$} (V3) 
                   (V3) edge[out=45,in=125,looseness=2.5] node[above] {$(j,k)(k,s)(j,k)$} (V3) 
                   (V3) edge node[below] {$(i,k)$} (V4) 
                   (V4) edge[out=45,in=125,looseness=2.5] node[above] {$(r,j)$} (V4)
                   (V4) edge node[below] {$(i,k)$} (V5) ;
             \end{tikzpicture}
         \end{center}\par 
         So subgraph with vertices $S_n^i\sqcup S_n^k$ is connected, our earlier assumption is verified and we have the claim.
\end{proof}

\section{Proof of the Linear Spanning Theorem}\label{sec:lin_span}
This section shows that the collection $\cf_\Gamma$ from Equation \ref{eqn:linear_generators} (below Lemma \ref{lem:linear_generators}) is a $\C$-spanning set of $\splines{\Gamma}^1$. In other words, we prove $\C\cf_\Gamma = \splines{\Gamma}^1$. First we require a lemma on the compatibility of these splines on $S_n$ with splines on $S_{n-1}$.
\begin{lemma}\label{lem:base_case}
    Let $\Gamma$ on $[n]$ be cliqued and naturally labelled. Let $\cf_\Gamma^{(n)} \coloneqq \{\bar{\rho}\vert_{S_{n-1}} \mid \bar{\rho} \in \cf_{\Gamma},\; \bar{\rho}(w) \in \polyr{n-1}\;\text{for all}\;w \in S_{n-1}\}$. Then 
    \[\C\cf_\Gamma^{(n)} = \C\cf_{\Gamma-n}.\]
\end{lemma}
\begin{proof}
     First, note that the Cayley graph $\cg_{\Gamma - n}$ is equal to the induced subgraph of $\cg_\Gamma$ with vertex set $S_{n-1}$. In particular, each element of $\cf_\Gamma^{(n)}$ is in fact a spline in $\splines{\Gamma - n}$ \par 
     By Lemma \ref{lem:linear_relations}(4), for each cut vertex $j$ in $\Gamma$ and $G$ the connected component of $\Gamma - j$ that contains $n$, we may remove the splines $\{\bar{y}_{G,k}^j\vert_{S_{n-1}} \mid k \in [n]\}$ from $\cf_{\Gamma}^{(n)}$ without changing the $\C$-span. Similarly by Lemma \ref{lem:linear_relations}(4), for each cut vertex $j \neq n-1$ of $\Gamma - n$ and connected component $G$ of $(\Gamma - n)-j$ that contains $n-1$,  we may remove the splines $\{\bar{y}_{G,k}^j \mid k \in [n]\}$ from $\cf_{\Gamma-n}$. \par 
     Let $\bar{\rho} \in \cf_{\Gamma}$ such $\bar{\rho}\vert_{S_{n-1}}$ is a nonzero element of $\cf_\Gamma^{(n)}$. This means that $\bar{\rho} \neq \bar{t}_n$ and $\bar{\rho} \neq \bar{x}_n$. Additionally, by the definitions if $\bar{\rho} = \bar{f}_A^s$ we must have $n \notin A$ (otherwise $\bar{\rho} \equiv 0$) and $s \neq (n-1,n)$, and if $\bar{\rho} = \bar{y}_{i,k}^j$ then $k \neq n$. So we have a combinatorial description for the elements of $\cf_\Gamma^{(n)}$. In particular, as collections of functions from $S_{n-1}$ to $\poly$, we wish to show that the following two sets have the same $\C$-span:
     \[
     \cf_\Gamma^{(n)} =\ct_{n-1} \cup \cx_{n-1} \cup \left\{\bar{f}_A^s \left| \begin{matrix} s\text{ cut edge of }\Gamma,\\ \abs{A} = \abs{G_s} \\ s \neq (n-1,n), \\ n \notin A \end{matrix}\right.\right\} \cup \left\{\bar{y}_{G,k}^j \left| \begin{matrix}j \vdash \Gamma,\\ n \notin G,\\ k \in [n-1] \end{matrix}\right.\right\}
     \]
     and
     \[
     \cf_{\Gamma-n} = \ct_{n-1} \cup \cx_{n-1} \cup \left\{\bar{f}_A^s \left| \begin{matrix} s\text{ cut edge of }\Gamma-n,\\ \abs{A} = \abs{G_s}\end{matrix}\right.\right\} \cup \left\{\bar{y}_{G,k}^j \left| \begin{matrix}j \vdash \Gamma-n,\\ n-1 \notin G,\\ k \in [n-1] \end{matrix}\right.\right\},
     \]
     where for the cut edges $s = (i<j)$ of either $\Gamma$ or $\Gamma - n$, the component $G_s$ is the connected component of the graph with edge $s$ removed that contains $i$. \par 
     The equalities for $\bar{t}_i \in \ct_{n-1}$ and $\bar{x}_i\in \cx_{n-1}$ is obvious, so we focus on the latter two subsets. The remainder of the proof is in three cases: whether the cliqued and naturally labeled graph $\Gamma$ is type A, B, or C. Each argument amounts to matching the cut vertices and cut edges of $\Gamma$ to those in $\Gamma - n$ (and vice versa). Each match gives pairs of splines in the third and fourth subsets above that are in fact equal to each other. Then we ensure that wherever these graph objects do not align, the ``unmatched" splines in each set are contained within the other's $\C$-span.  \par 
    \textbf{Type A:} First, we compare the cut edges of $\Gamma$ and $\Gamma - n$ and ensure that each spline in the third subsets of both $\cf_{\Gamma}^{(n)}$ and $\cf_{\Gamma-n}$ are contained within the span of the other set. If $\Gamma$ is type A, then every cut edge of $\Gamma - n$ is also a cut edge of $\Gamma$. Every cut edge of $\Gamma$ that isn't $(n-1,n)$ is also a cut edge of $\Gamma - n$. Finally, if $s = (i<j) \neq (n-1,n)$ is a cut edge, then the connected component of  $([n] E(\Gamma) \setminus\{s\})$ that contains $i$ and the connected component of $([n-1],E(\Gamma-n)\setminus\{s\})$ that contains $i$ are equal, since these are the components with lower-valued vertices and thus are unaffected by removing $n$. So the third subsets in $\cf_\Gamma^{(n)}$ and $\cf_{\Gamma - n}$ above are actually equal. \par 
    Second, we compare the cut vertices and associated connected components of $\Gamma$ and $\Gamma - n$ and ensure that each spline in the fourth subsets of both $\cf_{\Gamma}^{(n)}$ and $\cf_{\Gamma-n}$ are contained within the span of the other set. There are two cases, $\cut_{\Gamma}(n-1) = 1$ and $\cut_{\Gamma}(n-1) > 1$. If $\cut_{\Gamma}(n-1) > 1$, every cut vertex in $\Gamma$ is a cut vertex in $\Gamma - n$, and vice versa. If $j \vdash \Gamma-n$ where $j \neq n-1$, if $G$ is the component of $\Gamma - j$ that contains $n$ then $G-n$ is the component of $(\Gamma-n)-j$ that contains $n-1$. If $j = n-1$, then $\cf_{\Gamma-n}$ contains every $\bar{y}_{G,k}^{n-1}$ and $\cf_\Gamma^{(n)}$ contains every $\bar{y}_{G,k}^{n-1}$ such that $n \notin G$. Either way, these two collections of splines are identical, so the fourth subsets in $\cf_{\Gamma}^{(n)}$ and $\cf_{\Gamma - n}$ are in fact equal. \par 
    If $\cut_{\Gamma}(n-1) = 1$, then $n-1$ is not a cut vertex of $\Gamma-n$, so $\cf_{\Gamma - n}$ does not contain the spline $\bar{y}_{G,k}^{n-1} \in \cf_\Gamma^{(n)}$ where $V(G) = [n-2]$. However, in this case for all $w \in S_{n-1}$ we compute 
    \begin{align*} \bar{y}_{G,k}^{n-1}(w) &= \begin{cases}
        t_k - t_{w(n-1)} & w^{-1}(k) \in [n-2] \\ 
        0 & w^{-1}(k) = n-1.
    \end{cases}\\
    &= \bar{t}_k(w) - \bar{x}_{n-1}(w) \in \C\cf_{\Gamma-n}.
    \end{align*}
    Every other cut vertex $j \vdash \Gamma - n$ and connected component $G$ of $(\Gamma - n)-j$ (that doesn't contain $n-1$) is also a cut vertex of $\Gamma$ and connected component of $\Gamma - j$ (that doesn't contain $n$), so each spline of the form $\bar{y}_{G,k}^j$ in $\cf_{\Gamma - n}$ has a direct counterpart in $\cf_\Gamma^{(n)}$. Thus $\C\cf_\Gamma^{(n)} = \C\cf_{\Gamma-n}$, and the claim holds in type A. \par 
    \textbf{Type B:} First we compare the cut vertices and associated connected components, to match elements in the fourth subsets. If $\Gamma$ is type $B$, then every cut vertex in $\Gamma - n$ is also a cut vertex of $\Gamma$, and for all $j \vdash \Gamma$ the connected component of $\Gamma - j$ that contains $n$ also contains $n-1$, so the fourth subsets in $\cf_\Gamma^{(n)}$ and $\cf_{\Gamma - n}$ are equal. \par 
    Now compare the cut edges, to match elements in the third subsets. Every cut edge in $\Gamma$ is a cut edge in $\Gamma - n$, however the edge $(n-2,n-1)$ is a cut edge in $\Gamma - n$ but not in $\Gamma$. For this cut edge we let $G_{(n-2,n-1)}$ be the subgraph with vertex set $[n-2]$. So $\cf_{\Gamma - n}$ has a subset $\left\{\bar{f}_{[n-1] \setminus k}^{(n-2,n-1)} \mid k \in [n-1]\right\}$ of splines that is not a subset of $\cf_\Gamma^{(n)}$. We will show that these splines are contained within the span $\C\cf_\Gamma^{(n)}$. By Lemma \ref{lem:hard_relation} applied to the leaf $(n-2,n-1)$ in $\Gamma-n$, each spline in $\left\{\bar{f}_{[n-1] \setminus k}^{(n-2,n-1)}\in  \mid k \in [n-1]\right\} \subset \cf_{\Gamma-n}$ is a linear combination of the remaining splines in $\cf_{\Gamma-n}$. Since the fourth subsets are equal and every \emph{other} cut edge of $\Gamma - n$ is a cut edge of $\Gamma$, each of these remaining splines are also in $\cf_{\Gamma}^{(n)}$. So the one subset $\left\{\bar{f}_{[n-1] \setminus k}^{(n-2,n-1)} \mid k \in [n-1]\right\}$ of un-matched splines in $\cf_{\Gamma-n}$ is contained within the span $\C\cf_\Gamma^{(n)}$, and so $\C\cf_\Gamma^{(n)} = \C\cf_{\Gamma-n}$.\par
    \textbf{Type C:} If $\Gamma$ is type C, then every cut vertex or edge in $\Gamma-n$ is also a cut vertex or edge in $\Gamma$, and vice versa. Additionally, for any $j \vdash \Gamma$ the connected component of $\Gamma - j$ containing $n$ also contains $n-1$. So all indexing data is the same, and so $\cf_{\Gamma}^{(n)} = \cf_{\Gamma-n}$. Thus the claim holds in Type C.
\end{proof}
The proof of Theorem \ref{thm:linear_spanning} below assumes a natural label, so we will use the indexing conventions for $\cf_\Gamma$ described in Subsection \ref{ssec:lin_gens_revisit}. In particular, we will heavily use the weakly dominant $(i,j) > \Gamma$ and strongly dominant $(i,j) \gg \Gamma$ pairs in Definition \ref{def:dominant}. Now we are able to prove that $\C\cf_\Gamma = \splines{\Gamma}^1$, and compute a recursive dimension formula.
\begin{theorem}\label{thm:linear_spanning}
     Let $\Gamma$ be a connected graph. The splines $\cf_\Gamma$ from Equation (\ref{eqn:linear_generators}) form a $\C$-spanning set of $\splines{\Gamma}^1$. Furthermore, if $\Gamma$ is cliqued and naturally labeled, then 
    \begin{align*}
    \dim_\C(\splines{\Gamma}^1) = 1&+\dim_\C(\splines{\Gamma-n}^1) + \begin{cases} \choos{n-1}{1} &\text{if }\Gamma \; \text{is type A} \\ 1 &\text{if }\Gamma \; \text{is type B/C} \end{cases} 
    \\&+ \sum_{(i,j) \gg \Gamma} \choos{n-1}{\abs{\Gamma_i^j}-1} + \abs{\left\{(i,j) > \Gamma \right\}}.
    \end{align*}
\end{theorem}
  
\begin{proof}
     It suffices to assume for both parts of the claim that $\Gamma$ is cliqued and naturally labeled. Recall the decomposition $S_n = S_{n}^1 \sqcup \cdots \sqcup S_n^{n}$ where $S_n^i \coloneqq \{w \in S_n \mid w(i) = n\}$. Note $S_n^n = S_{n-1}$. Let $\bar{\rho} \in \splines{\Gamma}^1$. We will prove that $\bar{\rho} \in \C\cf_\Gamma$, and proceed by induction on $n$. The base case is $n=3$, where $\splines{\Gamma} = \C\cf_\Gamma$ is easily verified by hand (there are only two connected graphs on three vertices), and either way follows from \cite{ayzenberg2022second}. \par 
     In each of the three steps to the proof given below, we use elements of $\cf_\Gamma$ to replace $\bar{\rho}$ with a spline supported on a strictly smaller subset of $S_n$. To track $\dim_\C(\splines{\Gamma}^1)$, we will create a set $\cb$ of linearly independent elements of $\splines{\Gamma}^1$. \par 
    $\left(\textbf{Step 1: }S_n^n\right)$ This step applies the induction assumption to replace $\bar{\rho}$ with a spline supported on $S_n^1 \sqcup \cdots \sqcup S_n^{n-1}$. If $w,v \in S_n^n = S_{n-1}$ with $w^{-1}v = (i,j) \in E(\Gamma - n)$, then $\bar{\rho}(w) - \bar{\rho}(v) = c\left(t_{w(i)}-t_{w(j)}\right)$ where $c \in \C$. Since $w(i) \neq n$ and $w(j) \neq n$, the coefficient $[t_n]\bar{\rho}(w)$ of $t_n$ in $\bar{\rho}(w)$ must be equal to $[t_n]\bar{\rho}(v)$. Since $\Gamma - n$ is connected, the induced subgraph of $\cg_\Gamma$ with vertex set $S_n^n$ is connected. Moreover, the coefficient of $t_n$ is the same for all $\bar{\rho}(u)$ where $u \in S_n^n$. Let $c_n \coloneqq [t_n]\bar{\rho}(u)$ for $u \in S_n^n$. Then $[t_n]\left(\bar{\rho}-c_n\bar{t}_n\right)(u) = 0$ for all $u \in S_{n-1}$.\par 
    So we replace $\bar{\rho}$ with $\bar{\rho} - c_n\bar{t}_n$, and now $\bar{\rho}(u) \in \polyr{n-1}$ when $u \in S_n^n$. Let $\cb \coloneqq \{\bar{t}_n\}$. We will add linearly independent elements to $\cb$ throughout the proof and keep track of $\abs{\cb}$.\par   
    By Lemma \ref{lem:base_case} and the induction hypothesis,  $\splines{\Gamma-n}^1 = \C\cf_{\Gamma-n} = \C\cf_\Gamma^{(n)}$. So we may assume that $\bar{\rho}\vert_{S_{n-1}} \equiv 0$. Add to $\cb$ the $\dim_\C(\splines{\Gamma-n}^1)$-many splines required. Note these splines are independent once restricted to $S_n^n = S_{n-1}$, so any nontrivial linear combination will have elements of $S_n^n$ in its support. At this point $\abs{\cb} = \dim_\C(\splines{\Gamma - n}^1)+1$, and $\bar{\rho} \equiv 0$ on $S_n^n$.\par 
    $\left(\textbf{Step 2: } S_n^{n-1}\right)$ Next, we use elements of $\cf_\Gamma$ to replace $\bar{\rho}$ with a spline that evaluates to $0$ on $S_n^{n-1} \sqcup S_n^n$. The process is slightly different for graphs of type $A$ and types $B/C$.\par 
    Since $\Gamma$ is cliqued and naturally labeled, $(n-1,n) \in E(\Gamma)$. Thus for all $w \in S_n^{n-1}$ there is an edge $(w,w(n-1,n)) \in E(\cg_\Gamma)$ where $w(n-1,n) \in S_n^n$. Each of these edges are labeled $\llangle t_{w(n-1)}-t_{w(n)}\rrangle = \llangle t_n-t_{w(n)}\rrangle$. Since $\bar{\rho}\equiv 0$ on $S_n^n$, we have that $\bar{\rho}(w) = c_w\left(t_n-t_{w(n)}\right)$ for some $c_w \in \C$ for all $w \in S_n^{n-1}$.\par 
    \textbf{Type A:} By Lemma \ref{lem:technical_n1}, if $w,v \in S_n^{n-1}$ and $w(n) = v(n)$, then $c_w=c_v$.  
     Let $c_k \coloneqq c_w$ when $w(n) = k$. If $w(n-1) = n$ and $w(n) = k < n$, then by definition $\bar{f}_{[n] \setminus\{k\}}^{(n-1,n)}(w) = t_n-t_{w(n)}$. On the other hand if $w(n) = n$ then  $\bar{f}_{[n] \setminus\{k\}}^{(n-1,n)}(w) = 0$ whenever $k \neq n$. It follows that  
     \[\bar{\rho} - \sum_{i=1}^{n-1}c_k \bar{f}_{[n] \setminus \{k\}}^{(n-1,n)} \equiv 0\]
     on $S_n^{n-1}\sqcup S_n^n$. Add the $n-1$ coset splines $\bar{f}_{[n] \setminus \{k\}}^{(n-1,n)}$ used above to $\cb$, which are linearly independent from the splines already in $\cb$ since they are not supported on $S_n^n$ and have disjoint support on $S_n^{n-1}$. In type A, at this point $\abs{\cb} = 1 +\dim_\C(\splines{\Gamma - n}^1) + \choos{n-1}{1}$, and $\cb$ is linearly independent, even if we restrict the splines in $\cb$ to $S_n^{n-1} \sqcup S_n^n$.\par 
     \textbf{Type B/C:} We proceed in the same format as type A. By Lemma \ref{lem:technical_n1}, if $w,v \in S_n^{n-1}$, then $c_w=c_v$. write $c \coloneqq c_w$ for $w \in S_n^{n-1}$. Then 
     \[ \bar{\rho} - c(\bar{t}_n - \bar{x}_n) \equiv 0\]
     on $S_n^{n-1}\sqcup S_n^n$. Add the single linearly independent spline $\bar{t}_n - \bar{x}_n$ to $\cb$. In type B/C at this point $\abs{\cb} = \dim_\C(\splines{\Gamma - n}^1) + 2$, and $\cb$ is linearly independent, even if we restrict each spline to $S_n^{n-1} \sqcup S_n^n$.\par 
     $\left(\textbf{Step 3: } S_n^i\right)$ Now given a spline $\bar{\rho}$ such that $\bar{\rho} \equiv 0$ on $S_n^{i+1} \sqcup \cdots S_n^n$, we show how to replace it with a spline that vanishes on $S_n^i \sqcup \cdots S_n^n$. This step is repeated until $\bar{\rho}$ vanishes on all of $S_n$. Assume that $\bar{\rho} \equiv 0$ on $S_n^{i+1} \sqcup \cdots \sqcup S_n^n$. Additionally, we assume that the splines in $\cb$ are linearly independent, moreover the set remains linearly independent once each spline is restricted to $S_n^{i+1} \sqcup \cdots \sqcup S_n^n$. In particular, any nontrivial linear combination of splines in $\cb$ is nonzero on $S_n^{i+1} \sqcup \cdots \sqcup S_n^n$ (and therefore $S_n$).  The remainder of the proof is type A/B/C-independent, but still requires three cases. \par 
     First, a formulation of $\bar{\rho}(w)$ for $w \in S_n^i$ that will be used in each case. Since $\Gamma$ is naturally labeled and $i \neq n$, there exists $j \in [n]$ such that $i<j$ and $(i,j) \in E(\Gamma)$, so $(w,w(i,j)) \in E(\cg_\Gamma)$. If $w \in S_n^i$ then $w(i,j) \in S_n^j$, so \[\bar{\rho}(w) = c_w(t_{w(i)} - t_{w(j)}) = c_w(t_n-t_{w(j)})\]
     for some $c_w \in \C$. \par 
     \textbf{Case 1:} If $i$ is \emph{not} $j$-dominant for any $j \in [n]$, since $\Gamma$ is naturally labeled there exist (at least) two vertices $j,k$ where $i < j<k$ and $\{(i,j),(i,k)\}\subset E(\Gamma)$. It follows that
     \[\bar{\rho}(w) = c_w\left(t_n-t_{w(j)}\right) = c'_w\left(t_n-t_{w(k)}\right).\]
     This is not possible for $c_w,c'_w \in \C$ unless $c_w=c'_w=0$, and so $\bar{\rho}(w) = 0$. In short, $\bar{\rho} \equiv 0$ on $S_n^i$ and we do not need any splines from $\cf_\Gamma$ to achieve this.\par 
     \textbf{Case 2:} If there exists $j \in [n]$ where $(i,j) \gg \Gamma$, then $i$ is the maximal vertex in its connected component of $\Gamma - j$. Thus, the vertex $j$ is the \emph{only} element in the neighborhood $N(i)$ of the vertex $i$ that is greater than $i$ (so $i$ is not $k$-dominant for any $k \neq j$), and $(i,j)$ is a cut edge of $\Gamma$. By Lemma \ref{lem:technical_cutedge} if $v \in w\langle E(\Gamma)\setminus(i,j)\rangle$, then $c_v=c_w$.\par 
     Recall that $v \in w\langle E(\Gamma)\setminus(i,j)\rangle$ if and only if $w(V(\Gamma_i^j)) = v(V(\Gamma_i^j))$. Let \[\ca \coloneqq \left\{ A \subset [n] \left| \abs{A} = \abs{\Gamma_i^j},\; n \in A \right.\right\},\] and write $c_A \coloneqq c_w$ if $w(V(\Gamma_i^j)) = A$. If $w(i) = n$ then $w^{-1}(V(\Gamma_i^j)) = A \in \ca$, and we compute $\bar{f}_{A}^{(i,j)}(w) = t_n - t_{w(j)}$. Also, since $\Gamma$ is naturally labeled, if $k>i$ then $k \notin \Gamma_i^j$. In particular, if $k>i$ and $w(k) = n$, then $w^{-1}(V(\Gamma_i^j)) \notin \ca$, and so $\bar{f}_A^{(i,j)}(w) = 0$ for all $w \in S_n^{i+1} \sqcup \cdots \sqcup s_n^n$. It follows that 
     \[ \bar{\rho} - \sum_{A \in \ca} c_A \bar{f}_A^{(i,j)} \equiv 0\]
     on $S_n^i \sqcup \cdots \sqcup S_n^n$. We replace $\bar{\rho}$ with this spline. The coset splines $\bar{f}_A^{(i,j)}$ for $A \in \ca$ have disjoint support among themselves and are only supported on $S_n^r$ for $r \leq i$, so $\{\bar{f}_A^{(i,j)} \mid A \in \ca\} \cup \cb$ is linearly independent, even when each spline is restricted to $S_n^i \sqcup \cdots \sqcup S_n^n$. Each time we use Case 2, i.e. for each $(i,j) \gg \Gamma$,  we add $\abs{\ca} = \choos{n-1}{\abs{\Gamma_i^j}-1}$ -many splines to $\cb$.  \par 
     \textbf{Case 3:} If there exists $j \in [n]$ where $(i,j) > \Gamma$, the vertex $j$ is the \emph{only} element in the neighborhood $N(i)$ of the vertex $i$ that is greater than $i$ (so $i$ is not $k$-dominant for any $k \neq j$), but $(i,j)$ is not a cut edge.
     By Lemma \ref{lem:technical_cutvert}, if $w,v \in S_n^i$, then $c_w = c_v \eqqcolon c$. Finally, we confirm that 
     \[\bar{\rho} - c\, \bar{y}_{i,n}^j \equiv 0\]
     on $S_n^i \sqcup \cdots \sqcup S_n^n$. We replace $\bar{\rho}$ with this spline. The single spline $\bar{y}_{i,n}^j$ is supported on $S_n^r$ for $r \leq i$, and so $\{\bar{y}_{i,n}^j\} \cup \cb$ is linearly independent, even when restricted to $S_n^i \sqcup \cdots \sqcup S_n^n$. Each time Case 3 is used, i.e. for each $(i,j) > \Gamma$, we add $1$ spline to $\cb$. \par 
     When $i=1$ is reached, we have used $\cf_\Gamma$ to replace $\bar{\rho}$ with a spline $\bar{\rho} \equiv 0$ on all of $S_n$. Thus $\bar{\rho} \in \C\cb$ and the set $\cb \subseteq \cf_\Gamma$ is linearly independent (the restriction is now to the whole symmetric group $S_n^1 \sqcup \cdots \sqcup S_n^n = S_n$), and $\abs{\cb} = \dim_\C(\splines{\Gamma}^1)$ is as claimed.
\end{proof}
\section{The Linear Dimension Formula}\label{sec:lin_dim}
This section constructs a combinatorial invariant of simple graphs that is also the $\C$-dimension of the associated linear splines. \par 
Let $\Gamma$ be a connected graph on at least three vertices. First, if $j \vdash \Gamma$ is a cut vertex, let $\cut_\Gamma(j)+1$ be the number of connected components in $\Gamma - j$. This is a straightforward expansion of the definition we gave for $\cut_\Gamma(j)$ from naturally labeled graphs to all graphs. When $\Gamma$ is fixed we may drop the subscript and write $\cut(j)$.\par  
\begin{definition}\label{def:internal_leaf_blocks}
Recall the construction of a block-cut tree in Definition \ref{def:block_cut}. In this tree, every leaf is a block of $\Gamma$. Let $\LB_\Gamma$ be the set of blocks in $\Gamma$ that are leaves in the block-cut tree. We call elements of $\LB_\Gamma$ \emph{leaf blocks} of $\Gamma$. Let $\IB_\Gamma$ be the set of blocks in $\Gamma$ that are not leaves in the block-cut tree. We call elements of $\IB_\Gamma$ \emph{internal blocks} of $\Gamma$. Note when $\Gamma$ is $2$-connected, $\LB_\Gamma = \emptyset$ and $\IB_\Gamma = \{\Gamma\}$, in particular, when the block-cut tree of $\Gamma$ is a single vertex (i.e. when $\Gamma$ is $2$-connected), we consider $\Gamma$ to be an internal block.
\end{definition}
If a block $B$ in $\Gamma$ is size $\abs{B} = 2$, that block must consist of two vertices in $\Gamma$ connected by an edge. Since blocks are maximal $2$-connected subgraphs, this edge must be a cut edge. In particular, blocks $B$ in $\IB_\Gamma$ of size $\abs{B} = 2$ are in bijection with cut edges of $\Gamma$ that are not leaf edges. Let $\IC_\Gamma \coloneqq \{(i,j) \in E(\Gamma) \mid V(B) = \{i,j\} \text{ for some }B \in \IB_\Gamma\}$. The elements of $\IC_\Gamma$ are \emph{internal cut edges} of $\Gamma$. \par 
\begin{example}\label{ex:dim_blockcut_sets}
    Consider the graph 
    
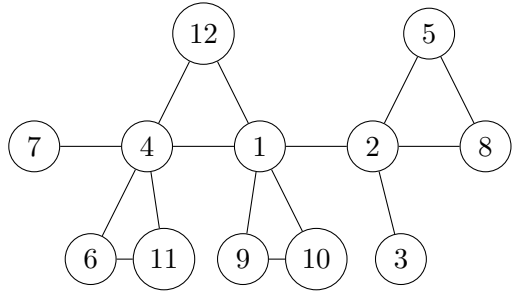
\begin{center} 
    \begin{tikzpicture}
            \begin{scope}[scale=1.5]
                \draw (0,0) node[draw,circle] (V1) {7};
                \draw (0.5,-1) node[draw,circle] (V2) {6};
                \draw (1.15,-1) node[draw,circle] (V3) {11};
                \draw (1,0) node[draw,circle] (V4) {4};
                \draw (1.5,1) node[draw,circle] (V5) {12};
                \draw (1.85,-1) node[draw,circle] (V6) {9};
                \draw (2.5,-1) node[draw,circle] (V7) {10};
                \draw (2,0) node[draw,circle] (V8) {1};
                \draw (3,0) node[draw,circle] (V10) {2};
                \draw (3.25,-1) node[draw,circle] (V9) {3};
                \draw (3.5,1) node[draw,circle] (V11) {5};
                \draw (4,0) node[draw,circle] (V12) {8};

                \draw (V1)--(V4);
                \draw (V2)--(V4);
                \draw (V3)--(V4);
                \draw (V2)--(V3);
                \draw (V4)--(V5);
                \draw (V5)--(V8);
                \draw (V4)--(V8);
                \draw (V6)--(V7);
                \draw (V6)--(V8);
                \draw (V7)--(V8);
                \draw (V8)--(V10);
                \draw (V9)--(V10);
                \draw (V10)--(V11);
                \draw (V10)--(V12);
                \draw (V11)--(V12);
            \end{scope}
        \end{tikzpicture}
    \end{center}
    with block-cut tree 
    
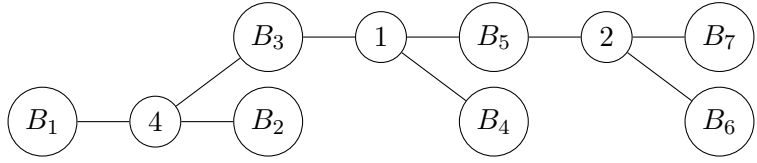
\begin{center} 
    \begin{tikzpicture}
            \begin{scope}[scale=1.5]
                \draw (0,0) node[draw,circle] (B1) {$B_1$};
                \draw (1,0) node[draw,circle] (v1) {$4$};
                \draw (2,0) node[draw,circle] (B2) {$B_2$};
                \draw (2,0.75) node [draw,circle] (B3) {$B_3$};
                \draw (3,0.75) node[draw,circle] (v2) {$1$};
                \draw (4,0) node[draw,circle] (B4) {$B_4$};
                \draw (4,0.75) node[draw,circle] (B5) {$B_5$};
                \draw (5,0.75) node[draw,circle] (v3) {$2$};
                \draw (6,0) node[draw,circle] (B6) {$B_6$};
                \draw (6,0.75) node[draw,circle] (B7) {$B_7$};

                \draw (B1)--(v1);
                \draw (v1)--(B2);
                \draw (v1)--(B3);
                \draw (B3)--(v2);
                \draw (v2)--(B4);
                \draw (v2)--(B5);
                \draw (B5)--(v3);
                \draw (v3)--(B6);
                \draw (v3)--(B7);
            \end{scope}
        \end{tikzpicture}
    \end{center}
    This is the construction from the beginning of Example \ref{ex:natural_label}. The leaf blocks are 
    \[\LB_\Gamma = \left\{ B_1,B_2,B_4,B_6,B_7  \right\}.\]
    The internal blocks are 
    \[\IB_\Gamma = \left\{B_3,B_5 \right\}.\]
    Within those internal blocks, $\abs{B_5} = 2$ and $B_5$ corresponds to the cut edge $(1,2)$. So 
    \[\IC_\Gamma = \left\{(1,2)\right\},\]
    and $(1,2)$ is the only cut edge in $\Gamma$ that is not a leaf edge.
\end{example}
For a connected graph $\Gamma$ on $n$ vertices, we define 
\begin{equation}\label{eqn:linear_dimension}
    D_\Gamma \coloneqq 2n-1 - \sum_{j \vdash \Gamma} \cut_\Gamma(j) + n\left(\abs{\LB_\Gamma} + \abs{\{B \in \IB_\Gamma \mid \abs{B} > 2\}} - 1\right) + \sum_{s \in \IC_\Gamma} \choos{n}{\abs{G_s}},
\end{equation}
where $G_s$ is defined as in Section \ref{sec:lin_gen_rels}, i.e. $G_s$ is one of the two connected components of the graph $([n],E(\Gamma)\setminus\{s\})$. This formula is unaffected by a choice of component, as the sizes of the two connected components in $([n],E(\Gamma)\setminus\{s\})$ sum to $n$ and $\choos{n}{\abs{G_s}} = \choos{n}{n-\abs{G_s}}$.
\begin{remark}
    The invariant $D_\Gamma$ might be more concisely written as 
    \[D_\Gamma = n-1 - \sum_{j \vdash \Gamma} \cut_\Gamma(j) + n\left(\abs{\LB_\Gamma} + \abs{\{B \in \IB_\Gamma \mid \abs{B} > 2\}}\right) + \sum_{s \in \IC_\Gamma} \choos{n}{\abs{G_s}},\]
    but the format in Equation \ref{eqn:linear_dimension} is more conducive to the proofs that follow.
\end{remark}
\begin{example}\label{ex:dim_blockcut_compute}
    Consider the graph $\Gamma$ from Example \ref{ex:dim_blockcut_sets} above. The three cut vertices are $4$, $1$, and $2$. Each of those cut vertices separate $\Gamma$ in to three connected components, so $\cut_\Gamma(4) = \cut_\Gamma(1) = \cut_\Gamma(2) = 2$. The block $B_3$ is the only block in $\IB_\Gamma$ with more than two vertices, so $\abs{\{B \in \IB_\Gamma \mid \abs{B} > 2\}} = 1$. The only internal cut edge is $(1,2)$, and this cut edge separates $\Gamma$ in to a component of size $8$ and a component of size $4$. We choose the component with vertex set $\{1,4,6,7,9,10,11,12\}$, but note that $\choos{12}{8} = \choos{12}{12-8} =  \choos{12}{4}$. We compute that 
    \[D_\Gamma = 2\cdot 12-1 - \left(2+2+2\right) +12\left(5+1-1\right) + \choos{12}{8} = 572  \]
\end{example}
Since $D_\Gamma$ is defined using only the block-cut tree, cut edges, and cut vertices of $\Gamma$, it is an invariant of the isomorphism class of $\Gamma$. We note that if $\Gamma$ is $2$-connected, it follows that $D_\Gamma = 2n-1$. \par 
Lemma \ref{lem:linear_dimension_natural} below gives a formulation of $D_\Gamma$ specific to naturally labeled graphs. Its proof constructs important bijections that will be used later in the proofs of Proposition \ref{prop:lin_dimension} and Corollary \ref{cor:nonlabel_reps}.
\begin{lemma}\label{lem:linear_dimension_natural}
    Let $\Gamma$ be a naturally labeled graph. Then 
    \[
    D_\Gamma = 2n-1 + \sum_{(i,j)>\Gamma} n + \sum_{(i,j) \gg \Gamma} \choos{n}{\abs{\Gamma_i^j}} - \sum_{j \vdash \Gamma} \cut_\Gamma(j).
    \]
\end{lemma}
\begin{proof}
    We compare the right-hand side of the formula above with (\ref{eqn:linear_dimension}). The clear cancellation between the two sides of the claimed equality is $2n-1-\sum_{j \vdash \Gamma} \cut_\Gamma(j)$. Thus the claim will follow if 
    \begin{equation}
    n\left(\abs{\LB_\Gamma} + \abs{\{B \in \IB_\Gamma \mid \abs{B} > 2\}} - 1\right) + \sum_{s \in \IC_\Gamma} \choos{n}{\abs{G_s}}=  \sum_{(i,j)>\Gamma} n + \sum_{(i,j) \gg \Gamma} \choos{n}{\abs{\Gamma_i^j}} . \tag{$*$}\label{inteqn:lin_dim_equality}
    \end{equation}
    Since $\Gamma$ is naturally labeled, the block containing $n$ is a leaf in the block-cut tree. So $\abs{\LB_\Gamma} - 1$ can be interpreted as the number of leaf blocks that do not contain $n$. \par 
    In the remainder of the proof, we pair blocks in $\LB_\Gamma$ and $\IB_\Gamma$ that contribute on the left side of the claimed equality (\ref{inteqn:lin_dim_equality}) with dominant pairs $(i,j)$ that contribute to the right side, in order to identify cancellations. The pairings that we will prove and use are as follows:
    \begin{enumerate}
        \item Leaf blocks $B \in \LB_\Gamma$ with $\abs{B} = 2$ and $n \notin B$ contribute $n$ to left side of (\ref{inteqn:lin_dim_equality}), and are in bijection with strongly dominant pairs $(i,j)$ where $\abs{\Gamma_i^j} = 1$, which contribute $\choos{n}{1} = n$ to the right side.
        \item The set of internal cut edges $\IC_\Gamma$ is equal to the set of strongly dominant pairs $(i,j)$ such that $\abs{\Gamma_i^j} >1$, and they both contribute $\choos{n}{\abs{G_s}} = \choos{n}{\abs{\Gamma_i^j}}$. 
        \item Leaf blocks and internal blocks of size at least three (i.e. all blocks of size at least $3$) that do \emph{not} contain the vertex $n$ contribute $n$ to the left side of the claimed equality, and are in bijection with weakly dominant pairs $(i,j)$, each of which contributes $n$ to the right side. 
    \end{enumerate}
    This list also serves as an outline of the proof that follows.\par 
    (1) Let $B$ be a leaf block of size $2$ with vertex set $V(B) = \{i,j\}$ where $i<j$. Then the single edge $(i,j)$ within $B$ is a cut edge of $\Gamma$, and since $B$ is a leaf block that cut edge $(i,j)$ must separate a single vertex. If this block does not contain $n$ (so $j \neq n$), then since $\Gamma$ is naturally labeled, the cut edge must be a dominant pair and that separated vertex must be $i$. So $(i,j) \gg \Gamma$ and $V(\Gamma_i^j) = \{i\}$, and $\abs{\Gamma_i^j} = 1$. On the other hand, if $(i,j) \gg \Gamma$ and $\abs{\Gamma_i^j} = 1$, then $i$ must be a leaf, and the subgraph $(\{i,j\},\{(i,j)\})$ is a block in $\Gamma$ that does not contain $n$. In particular, we have shown 
    \[\{B \in \LB_\Gamma \mid \abs{B} = 2,\; n \notin B\} = \{B \in \LB_\Gamma \mid V(B) = \{i,j\},\; (i,j) \gg \Gamma\}.\]
    So the leaf blocks $B$ in $\Gamma$ of size $2$ are in natural bijection with the dominant pairs $(i,j) \gg \Gamma$ where $\abs{\Gamma_i^j} = 1$. Formally, $\abs{\{B \in \LB_\Gamma \mid \abs{B} = 2,\; n \notin B\}} = \abs{\left\{(i,j) \gg \Gamma \mid \abs{\Gamma_i^j} = 1\right\}}$. Thus
    \begin{align*}
        \abs{\LB_\Gamma} + \abs{\{B \in \IB_\Gamma \mid \abs{B} > 2\}} - 1 &= \abs{\{B \in \LB_\Gamma \mid \abs{B} = 2\}}+\abs{\{B \in \LB_\Gamma \mid \abs{B} > 2\}} \\&\hspace{2cm}+ \abs{\{B \in \IB_\Gamma \mid \abs{B} > 2\}} - 1 \\
        &= \abs{\left\{(i,j) \gg \Gamma \mid \abs{\Gamma_i^j} = 1\right\}}+\abs{\{B \in \LB_\Gamma \mid \abs{B} > 2,\; n \notin B\}} \\&\hspace{2cm}+ \abs{\{B \in \IB_\Gamma \mid \abs{B} > 2\}} \\
        &= \abs{\left\{(i,j) \gg \Gamma \mid \abs{\Gamma_i^j}=1\right\}} \\&\hspace{2cm}+ \abs{\{B \mid B \text{ a block in }\Gamma,\; \abs{B} > 2,\; n \notin B\}}.
    \end{align*}
    We cancel ${\ds n\abs{\left\{(i,j) \gg \Gamma \mid \abs{\Gamma_i^j}=1\right\}}}$ from the left side and ${\ds \sum_{\substack{(i,j) \gg \Gamma \\ \abs{\Gamma_i^j}=1}} \choos{n}{1}}$ from the right side of (\ref{inteqn:lin_dim_equality}), and it remains to prove that 
    \begin{equation}
        n\abs{\{B \mid B \text{ a block in }\Gamma,\; \abs{B} > 2,\; n \notin B\}} + \sum_{s \in \IC_\Gamma} \choos{n}{\abs{G_s}}=  \sum_{(i,j)>\Gamma} n + \sum_{\substack{(i,j) \gg \Gamma \\ \abs{\Gamma_i^j}>1 }} \choos{n}{\abs{\Gamma_i^j}}. \tag{$**$}\label{inteqn:lin_dim_equality2}
    \end{equation}
    (2) Now if $s = (i,j) \in \IC_\Gamma$, then $s$ is a cut edge and since $\Gamma$ is naturally labeled, $n \notin \{i,j\}$. So if $(i,j) \in \IC_\Gamma$ then $(i,j) \gg \Gamma$ and $\abs{\Gamma_i^j} > 1$, otherwise $i$ is a leaf and the block $B$ where $V(B) = \{i,j\}$ is not an internal block. On the other hand, if $(i,j) \gg \Gamma$ and $\abs{\Gamma_i^j} > 1$, then $i$ cannot be a leaf, $(i,j)$ must be a cut edge and the block $B$ where $V(B) = \{i,j\}$ is not a leaf in the block-cut tree. So $\IC_\Gamma = \left\{(i,j) \gg \Gamma \mid \abs{\Gamma_i^j} > 1\right\}$. In particular,
    \[
    \sum_{s \in \IC_\Gamma} \choos{n}{\abs{G_s}} = \sum_{\substack{(i,j) \gg \Gamma\\ \abs{\Gamma_i^j} > 1}} \choos{n}{\abs{\Gamma_i^j}}.
    \]
    After cancelling this value from both sides of (\ref{inteqn:lin_dim_equality2}), it remains to show that 
    \begin{equation}
        n\abs{\{B \mid B \text{ a block in }\Gamma,\; \abs{B} > 2,\; n \notin B\}} = \sum_{(i,j)>\Gamma} n. \tag{$***$}\label{inteqn:lin_dim_equality3}
    \end{equation}
    (3) We argue that 
    \[\abs{\{B \mid B \text{ a block in }\Gamma,\; \abs{B} > 2,\; n \notin B\}} = \abs{\{(i,j) > \Gamma\}}.\]
    The bijection is as follows. If $B$ is a block in $\Gamma$ that does not contain $n$, then there is a unique path from $B$ to the block $B_0$ that contains $n$ in the block-cut tree of $\Gamma$. The first edge in that path is from $B$ to a cut vertex of $\Gamma$ that is contained within $B$. Let $j$ be this cut vertex, and let $i$ be the maximal vertex in $B$ that is \emph{not} equal to $j$. Note that $B$ is not connected to $n$ in $\Gamma - j$. \par 
    Since $\Gamma$ is naturally labeled, $(i,j)$ is an edge and $i$ must also be the maximal vertex in its connected component of $\Gamma - j$ (every path from $n$ to that component first passes through $j$, and $i$ is the largest vertex in that component that is adjacent to $j$). Since $\abs{B} > 2$, there is at least one other element of $B$ in the neighborhood of $j$, so $(i,j)$ is not a cut edge. In particular, $(i,j) > \Gamma$. \par 
    On the other hand, if $(i,j) > \Gamma$ let $B$ be the block containing $i$. This block cannot contain $n$ by the definition of a natural label, and must be of size at least $2$ since $(i,j)$ is not a cut edge. \par 
    So we have a bijection, proving as a consequence (\ref{inteqn:lin_dim_equality3}), and the claim follows.
\end{proof}
The bijection between certain dominant pairs and blocks that we constructed in the the proof of Lemma \ref{lem:linear_dimension_natural} above will be used again to compute a recursive formula for $D_\Gamma$ and provide a label-independent formula for $\lrep{\Gamma}_1$ and $\rrep{\Gamma}_1$ in Corollary \ref{cor:nonlabel_reps} below. Visually, when $\Gamma$ is naturally labeled we have the following correspondences between blocks in $\Gamma$ and dominant pairs.
\begin{center}
    \begin{tabular}{r|ccccc}
        Blocks $B$    & $\begin{Bmatrix} \abs{B} > 2\\ n \notin B  \end{Bmatrix}$ && $\begin{Bmatrix}B \in \LB_\Gamma \\ \abs{B} = 2,\;n \notin B  \end{Bmatrix}$ && $\begin{Bmatrix}B \in \IB_\Gamma \\ \abs{B} = 2  \end{Bmatrix}$ \\ &&& \\
                      & $\Big\updownarrow$ && $\Big\updownarrow$ && $\Big\updownarrow$ \\ &&& \\
        $\begin{matrix} \text{Dominant} \\ \text{pairs }(i,j)\end{matrix}$ & $\left\{(i,j) > \Gamma\right\}$ && $\begin{Bmatrix} (i,j) \gg \Gamma \\ \abs{\Gamma_i^j} = 1 \end{Bmatrix}$ && $\begin{Bmatrix} (i,j) \gg \Gamma \\ \abs{\Gamma_i^j} > 1 \end{Bmatrix}$
    \end{tabular}
\end{center}
\begin{example}\label{ex:blockpair_bijection}
    Consider the naturally labeled graph $\Gamma$ drawn below
    
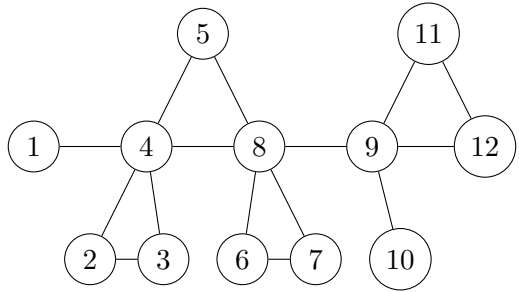
\begin{center} 
    \begin{tikzpicture}
            \begin{scope}[scale=1.5]
                \draw (0,0) node[draw,circle] (V1) {1};
                \draw (0.5,-1) node[draw,circle] (V2) {2};
                \draw (1.15,-1) node[draw,circle] (V3) {3};
                \draw (1,0) node[draw,circle] (V4) {4};
                \draw (1.5,1) node[draw,circle] (V5) {5};
                \draw (1.85,-1) node[draw,circle] (V6) {6};
                \draw (2.5,-1) node[draw,circle] (V7) {7};
                \draw (2,0) node[draw,circle] (V8) {8};
                \draw (3,0) node[draw,circle] (V10) {9};
                \draw (3.25,-1) node[draw,circle] (V9) {10};
                \draw (3.5,1) node[draw,circle] (V11) {11};
                \draw (4,0) node[draw,circle] (V12) {12};

                \draw (V1)--(V4);
                \draw (V2)--(V4);
                \draw (V3)--(V4);
                \draw (V2)--(V3);
                \draw (V4)--(V5);
                \draw (V5)--(V8);
                \draw (V4)--(V8);
                \draw (V6)--(V7);
                \draw (V6)--(V8);
                \draw (V7)--(V8);
                \draw (V8)--(V10);
                \draw (V9)--(V10);
                \draw (V10)--(V11);
                \draw (V10)--(V12);
                \draw (V11)--(V12);
            \end{scope}
        \end{tikzpicture}
    \end{center}
    with block-cut tree 
    
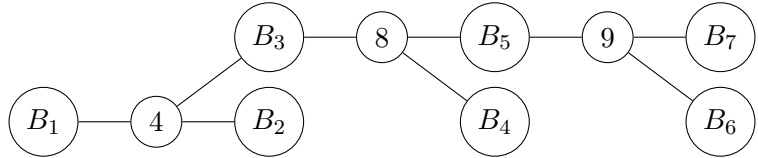
\begin{center} 
    \begin{tikzpicture}
            \begin{scope}[scale=1.5]
                \draw (0,0) node[draw,circle] (B1) {$B_1$};
                \draw (1,0) node[draw,circle] (v1) {$4$};
                \draw (2,0) node[draw,circle] (B2) {$B_2$};
                \draw (2,0.75) node [draw,circle] (B3) {$B_3$};
                \draw (3,0.75) node[draw,circle] (v2) {$8$};
                \draw (4,0) node[draw,circle] (B4) {$B_4$};
                \draw (4,0.75) node[draw,circle] (B5) {$B_5$};
                \draw (5,0.75) node[draw,circle] (v3) {$9$};
                \draw (6,0) node[draw,circle] (B6) {$B_6$};
                \draw (6,0.75) node[draw,circle] (B7) {$B_7$};

                \draw (B1)--(v1);
                \draw (v1)--(B2);
                \draw (v1)--(B3);
                \draw (B3)--(v2);
                \draw (v2)--(B4);
                \draw (v2)--(B5);
                \draw (B5)--(v3);
                \draw (v3)--(B6);
                \draw (v3)--(B7);
            \end{scope}
        \end{tikzpicture}
    \end{center}
    The blocks of size at least $2$ that do not contain $n$ correspond to weakly dominant pairs in the following manner: 
    \begin{center}
        \begin{tabular}{c|c}
            Blocks & Pairs \\\hline
            $B_2$ on $\{2,3,4\}$ & $(3,4) > \Gamma$ \\
            $B_3$ on $\{4,5,8\}$ & $(5,8) > \Gamma$ \\
            $B_4$ on $\{6,7,8\}$ & $(7,8) > \Gamma$ \\
        \end{tabular}
    \end{center}
    The leaf blocks of size $2$ that do not contain $n$ correspond to strongly dominant pairs $(i,j)$ where $\abs{\Gamma_i^j} = 1$ in the following manner:
    \begin{center}
        \begin{tabular}{c|c}
            Blocks & Pairs \\\hline
            $B_1$ on $\{1,4\}$ & $(1,4) > \Gamma$ \\
            $B_6$ on $\{9,10\}$ & $(9,10) > \Gamma$ \\
        \end{tabular}
    \end{center}
    Finally the internal block $B_5$ on $\{8,9\}$ (so of size $2$) corresponds to the dominant pair $(8,9) \gg \Gamma$.
\end{example}

\begin{example}
    When $\Gamma$ is naturally labeled, each sum appearing in the formula in Lemma \ref{lem:linear_dimension_natural} has the following combinatorial interpretation:
    \begin{itemize}
        \item $n\;$: Every vertex of $\Gamma$ contributes $1$ to the sum,
        \item ${\ds n-1-\sum_{j \vdash \Gamma} \cut_\Gamma(j)}\;$: Every vertex that is \emph{not} the lower vertex in a dominant pair or the vertex $n$ contributes an additional $1$ to the sum (recall that the lower vertices in a dominant pair determine the pair uniquely).
        \item ${\ds \sum_{j \vdash \Gamma} \sum_{(i,j)>\Gamma} n }\;$: Every weakly dominant pair contributes $n$ to the sum.
        \item ${\ds \sum_{j \vdash \Gamma} \sum_{(i,j)\gg \Gamma} \choos{n}{\abs{\Gamma_i^j}}}\;$: Every strongly dominant pair $(i,j)$ contributes $\choos{n}{\abs{\Gamma_i^j}}$ to the sum.
    \end{itemize}
    Below we have drawn $\Gamma$ from Example \ref{ex:linear_example_2}, with the associated values for $D_\Gamma$ in blue. In this picture, dashed lines indicate weak dominance and double lines indicate strong dominance.
    
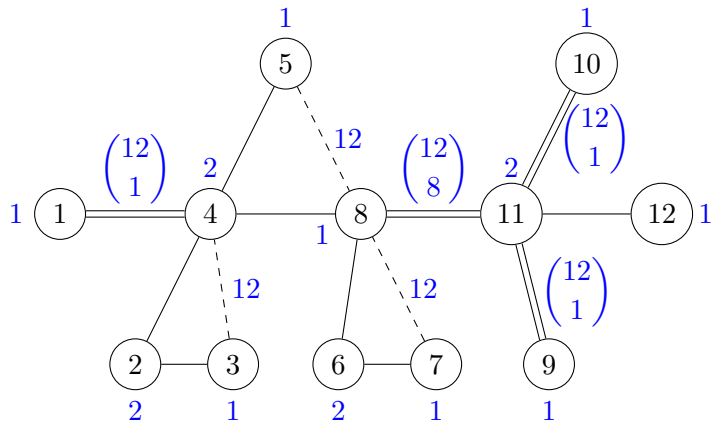
\begin{center}
    \begin{tikzpicture}
        \begin{scope}[scale=2]
            \draw (0,0) node[draw,circle] (V1) {$1$} node[left=10pt]{{\color{blue} 1}};
            \draw (0.5,-1) node[draw,circle] (V2) {$2$} node[below=10pt]{{\color{blue} 2}};
            \draw (1.15,-1) node[draw,circle] (V3) {$3$} node[below=10pt]{{\color{blue} 1}};
            \draw (1,0) node[draw,circle] (V4) {$4$} node[above=10pt]{{\color{blue} 2}};
            \draw (1.5,1) node[draw,circle] (V5) {$5$} node[above=10pt]{{\color{blue} 1}};
            \draw (1.85,-1) node[draw,circle] (V6) {$6$} node[below=10pt]{{\color{blue} 2}};
            \draw (2.5,-1) node[draw,circle] (V7) {$7$} node[below=10pt]{{\color{blue} 1}};
            \draw (2,0) node[draw,circle] (V8) {$8$} node[below=8pt,left=8pt]{{\color{blue} 1}};
            \draw (3,0) node[draw,circle] (V11) {$11$} node[above=10pt]{{\color{blue} 2}};
            \draw (3.25,-1) node[draw,circle] (V9) {$9$} node[below=10pt]{{\color{blue} 1}};
            \draw (3.5,1) node[draw,circle] (V10) {$10$} node[above=10pt]{{\color{blue} 1}};
            \draw (4,0) node[draw,circle] (V12) {$12$} node[right=10pt]{{\color{blue} 1}};

            \draw[double distance=2pt] (V1)-- node[above] {\color{blue}$\choos{12}{1}$}(V4);
            \draw (V2)--(V4);
            \draw[dashed] (V3)-- node[right] {{\color{blue}$12$}}(V4);
            \draw (V2)--(V3);
            \draw (V4)--(V5);
            \draw[dashed] (V5)-- node[right] {{\color{blue}$12$}}(V8);
            \draw (V4)--(V8);
            \draw (V6)--(V7);
            \draw (V6)--(V8);
            \draw[dashed] (V7)-- node[right] {{\color{blue}$12$}}(V8) ;
            \draw[double distance=2pt] (V8)-- node[above] {{\color{blue}$\choos{12}{8}$}}(V11);
            \draw[double distance=2pt] (V9)-- node[right] {\color{blue}$\choos{12}{1}$}(V11);
            \draw[double distance=2pt] (V11)-- node[right] {\color{blue}$\choos{12}{1}$}(V10);
            \draw (V11)--(V12);
        \end{scope}
    \end{tikzpicture}
    \end{center}
\end{example}
\begin{remark}
    It is not obvious that the formula in Lemma \ref{lem:linear_dimension_natural} is the same for different naturally labeled graphs in the isomorphism class. However, it is clear from the definition in (\ref{eqn:linear_dimension}) that $D_\Gamma$ is an invariant, so they must be equal.
\end{remark}
The following Proposition \ref{prop:lin_dimension} gives an inductive formula for $D_\Gamma$ when $\Gamma$ is cliqued and naturally labeled that matches the inductive formula for $\dim_\C(\splines{\Gamma}^1)$ from Theorem \ref{thm:linear_spanning}.
\begin{proposition}\label{prop:lin_dimension}
    If $\Gamma$ is a cliqued and naturally labeled graph on at least $4$ vertices, then $D_\Gamma$ can be computed recursively as follows:
    \begin{align*}
    D_\Gamma = 1&+D_{\Gamma-n} + \begin{cases} \choos{n-1}{1} &\text{if } \Gamma \; \text{is type A} \\ 1 &\text{if }\Gamma \; \text{is type B/C} \end{cases} 
    \\&+ \sum_{(i,j) \gg \Gamma} \choos{n-1}{\abs{\Gamma_i^j}-1} + \abs{\left\{(i,j) >  \Gamma \right\}}.
    \end{align*}
\end{proposition}
Before proving Proposition \ref{prop:lin_dimension}, we prove several computational lemmas and construct a finer categorization of type A graphs. The key idea is to concretely describe the block-cut tree of $\Gamma - n$ in terms of the block-cut tree of $\Gamma$.  \par
If $B$ is a block in $\Gamma$ that does not contain $n$, then $B$ is still $2$-connected in $\Gamma - n$. Additionally, if $i$ and $k$ where $n \notin \{i,k\}$ are vertices in $\Gamma$ such that $i$ and $k$ are in different connected components of $\Gamma - j$, then $i$ and $k$ are also in different connected components of $ (\Gamma - n) - j$. In particular, If $B$ is a block in $\Gamma$ that does not contain $n$, then $B$ is a block in $\Gamma - n$. \par 
Lemma \ref{lem:dim_typeBC} below describes the simplest case, when $\Gamma$ is either type B or type C
\begin{lemma}\label{lem:dim_typeBC}
    Let $\Gamma$ be a cliqued and naturally labeled graph on at least $4$ vertices. If $\Gamma$ is of type B or C, then the following equalities hold:
    \begin{itemize}
        \item $\abs{\LB_\Gamma} = \abs{\LB_{\Gamma - n}}$,
        \item $\IB_\Gamma = \IB_{\Gamma - n}$, and 
        \item $\IC_\Gamma = \IC_{\Gamma-n}$.
    \end{itemize}
\end{lemma}
\begin{proof}
    Let $B_0$ be the block in $\Gamma$ that contains $n$. Since $\Gamma$ is type B or C, $\abs{B_0} > 2$.  Since $\Gamma$ is cliqued, the subgraph $B_0 -n$ has at least $2$ vertices and is also a clique, and in particular $B_0 - n$ is a block in $\Gamma - n$. Every other block or cut vertex in $\Gamma$ is a block or cut vertex in $\Gamma - n$. Thus, the block-cut tree of $\Gamma$ is isomorphic to the block-cut tree of $\Gamma - n$. \par 
    Since the block-cut trees are isomorphic they have the same number of leaves, and so $\abs{\LB_\Gamma} = \abs{\LB_{\Gamma - n}}$. Since the only block that changes is $B_0$ and all non-leaf blocks in $\Gamma$ are non-leaf blocks in $\Gamma - n$, it follows that $\IB_\Gamma = \IB_{\Gamma - n}$. This equality implies $\IC_\Gamma = \IC_{\Gamma-n}$ by definition.
\end{proof}
 We now assume that $\Gamma$ is type A. We will give a similar computation for type-A graphs, but there are several cases to consider. Since $(n-1,n)$ is a cut edge, it corresponds to a block $B_0$ containing $ n$ in $\Gamma$ that has no natural counterpart in the block-cut tree of $\Gamma - n$. Not only that, but $n-1$ may not be a cut vertex in $\Gamma-n$. \par 
 Addressing this requires further decomposition of type A graphs, which we denote A1, A2, A3, A4, and A5. We will define them carefully below, but the consequences in each case in terms of moving from the block-cut tree of $\Gamma$ to that of $\Gamma - n$ are essentially as follows (recall $B_0$ is the block in $\Gamma$ containing $n$):
 \begin{itemize}
     \item[(A1)] The block cut tree of $\Gamma - n$ is simply that of $\Gamma$ with the block $B_0$ removed.
     \item[(A2)] The block-cut tree of $\Gamma - n$ is the block cut tree of $\Gamma$ with the block $B_0$ and the cut vertex $n-1$ removed, and an internal block $B'$ of size $2$ for $\Gamma$ becomes a leaf block for $\Gamma-n$.
     \item[(A3)]  The block-cut tree of $\Gamma - n$ is the block cut tree of $\Gamma$ with the block $B_0$ and the cut vertex $n-1$ removed, but every other internal (resp. leaf) block of $\Gamma$ remains an internal (resp. leaf) block of $\Gamma - n$.
     \item[(A4)]  The block-cut tree of $\Gamma - n$ is the block cut tree of $\Gamma$ with the block $B_0$ and the cut vertex $n-1$ removed, and an internal block $B'$ of size greater than $2$ for $\Gamma$ becomes a leaf block for $\Gamma - n$.
     \item[(A5)] The graph $\Gamma - n$ is $2$-connected.
 \end{itemize}\par 
 First, a type A graph $\Gamma$ is type A1 if $\cut_\Gamma(n-1) > 1$. Graphically, $\Gamma$ looks like 
 
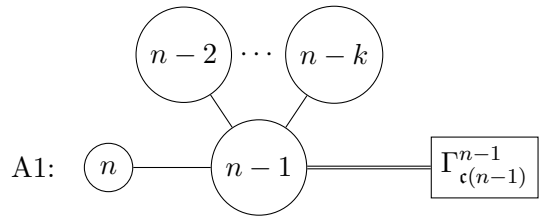
\begin{center}
     \begin{tikzpicture}
         \begin{scope}
            \draw (0,0) node[draw,circle] (V1) {$n$};
            \draw (2,0) node[draw,circle] (V2) {$n-1$};
            \draw (5,0) node[draw,rectangle] (V3) {$\Gamma_{\cut(n-1)}^{n-1}$};
            \draw (-1,0) node (type) {A1:};

            \draw (1,1.5) node[draw,circle] (V4) {$n-2$};
            \draw (2,1.5) node (dots) {$\cdots$};
            \draw (3,1.5) node[draw,circle] (V5) {$n-k$};

            \draw (V1)--(V2);
            \draw (V2)--(V4);
            \draw (V2)--(V5);
            \draw[double] (V2)--(V3);
        \end{scope}
     \end{tikzpicture}
 \end{center}
 Note that in type A1, $\Gamma - n$ has the same cut vertices as $\Gamma$, so the block-cut tree of $\Gamma - n$ is the block-cut tree of $\Gamma$ with $B_0$ removed. \par 
If $\Gamma$ is type A but not type A1, then $n-1$ is a cut vertex of $\Gamma$ but it is not a cut vertex of $\Gamma - n$. So the block-cut tree of $\Gamma - n$ is the block-cut tree of $\Gamma$ with both the block $B_0$ and the cut vertex $n-1$ removed. In particular, $n-1$ is contained within precisely two blocks in $\Gamma$, one that contains $n$ (so $B_0$) and one that contains $n-2$. Let $B'$ be the block in $\Gamma$ that contains $n-1$ and $n-2$.  \par 
We say $\Gamma$ is type A2 if $\abs{B'} = 2$, i.e. $V(B') = \{n-2,n-1\}$. Graphically, $\Gamma$ and its block- cut tree look like
 
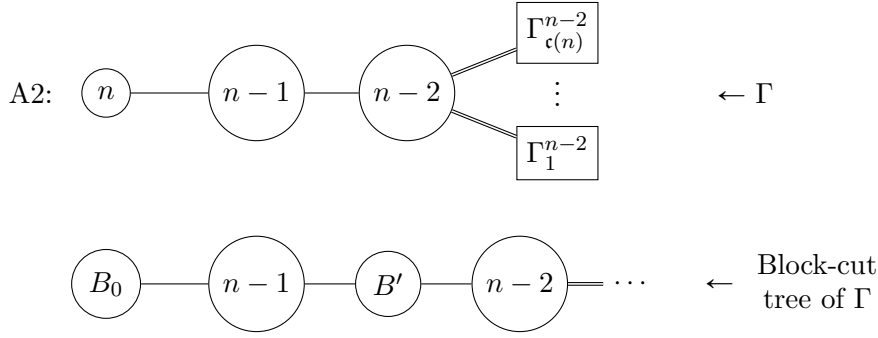
\begin{center}
     \begin{tikzpicture}
        \begin{scope}
            \draw (0,0) node[draw,circle] (V1) {$n$};
            \draw (2,0) node[draw,circle] (V2) {$n-1$};
            \draw (4,0) node[draw,circle] (V3) {$n-2$};
            \draw (6,-0.8) node[draw,rectangle] (V4) {$\Gamma_1^{n-2}$};
            \draw (6,0.8) node[draw,rectangle] (V5) {$\Gamma_{\cut(n)}^{n-2}$};
            \draw (6,0.1) node (V9) {$\vdots$};
            \draw (-1,0) node (type) {A2:};
            \draw (8.5,0) node (gr) { $\leftarrow$ $\Gamma$};

            \draw (V1)--(V2);
            \draw (V2)--(V3);
            \draw[double] (V3)--(V4);
            \draw[double] (V3)--(V5);
        \end{scope}
        \begin{scope}[yshift=-1in]
             \draw (0,0) node[draw,circle] (V1) {$B_0$};
            \draw (2,0) node[draw,circle] (V2) {$n-1$};
            \draw (3.75,0) node[draw,circle] (V3) {$B'$};
            \draw (5.5,0) node[draw,circle ] (V4) {$n-2$};   
            \draw (7,0) node (V5) {$\cdots$};

            \draw (9.2,0) node (gr) { $\leftarrow$ \begin{tabular}{c}
                 Block-cut \\ tree of $\Gamma$
            \end{tabular}} ;

            \draw (V1)--(V2);
            \draw (V2)--(V3);
            \draw (V3)--(V4);
            \draw[double] (V4)--(V5);
        \end{scope}

        \end{tikzpicture}
 \end{center}
 Note that the $B'$ where $V(B') = \{n-2,n-1\}$ is associated to an internal cut edge in $\Gamma$, and is a leaf block in $\Gamma - n$. \par  
 The graph $\Gamma$ is type A3 if $\abs{B'} > 2$ and $B'$ contains more than $2$ cut vertices of $\Gamma$. So $B'$ is adjacent to more than two vertices in the block-cut tree of $\Gamma$. Graphically, the block-cut tree of $\Gamma$ looks like
 
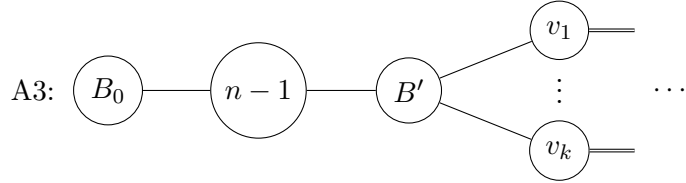
\begin{center}
    \begin{tikzpicture}
        \begin{scope}
             \draw (0,0) node[draw,circle] (V1) {$B_0$};
            \draw (2,0) node[draw,circle] (V2) {$n-1$};
            \draw (4,0) node[draw,circle] (V3) {$B'$};
            \draw (6,-0.8) node[draw,circle] (V4) {$v_k$};
            \draw (6,0.8) node[draw,circle] (V5) {$v_1$};
            \draw (6,0.1) node (V9) {$\vdots$};
            \draw (7.5,0) node (v10) {$\cdots$};
            \draw (-1,0) node (type) {A3:};

            \draw (V1)--(V2);
            \draw (V2)--(V3);
            \draw (V3)--(V4);
            \draw (V3)--(V5);
            \draw[double] (V4)--(7,-0.8);
            \draw[double] (V5)--(7,0.8);
            
        \end{scope}
    \end{tikzpicture}
\end{center}
 where $k > 1$. Note that $B'$ is an internal block in both $\Gamma$ and $\Gamma - n$. \par 
 The graph $\Gamma$ is type A4 if $\abs{B'} > 2$, and $B'$ contains precisely $2$ cut vertices of $\Gamma$, $n-1$ and some other cut vertex $v$ of $\Gamma$. In particular, $B'$ is adjacent to precisely two vertices in the block-cut tree of $\Gamma$. Graphically, the block-cut tree of $\Gamma$ looks like
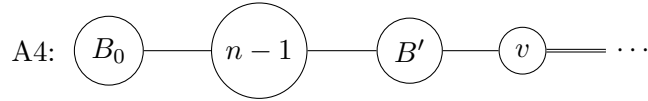
\begin{center}
    \begin{tikzpicture}
        \begin{scope}
             \draw (0,0) node[draw,circle] (V1) {$B_0$};
            \draw (2,0) node[draw,circle] (V2) {$n-1$};
            \draw (4,0) node[draw,circle] (V3) {$B'$};
            \draw (5.5,0) node[draw,circle ] (V4) {$v$};   
            \draw (7,0) node (V5) {$\cdots$};
            
            \draw (-1,0) node (type) {A4:};

            \draw (V1)--(V2);
            \draw (V2)--(V3);
            \draw (V3)--(V4);
            \draw[double] (V4)--(V5);
        \end{scope}
    \end{tikzpicture}
\end{center}
Note that $B'$ is an internal block of size at least $3$ in $\Gamma$ and a leaf block in $\Gamma - n$. \par 
Finally, the graph $\Gamma$ is type A5 if $n-1$ is the only cut vertex in $B'$. In particular, $B_0$ and $B'$ are the only two blocks in $\Gamma$, so the block cut tree of $\Gamma$ is
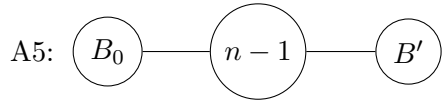
\begin{center}
    \begin{tikzpicture}
        \begin{scope}
             \draw (0,0) node[draw,circle] (V1) {$B_0$};
            \draw (2,0) node[draw,circle] (V2) {$n-1$};
            \draw (4,0) node[draw,circle] (V3) {$B'$};
            
            \draw (-1,0) node (type) {A5:};

            \draw (V1)--(V2);
            \draw (V2)--(V3);
        \end{scope}
    \end{tikzpicture}
\end{center}
The computations for type A5 are generally easy, as $\Gamma - n$ is $2$-connected. \par 
Lemmas \ref{lem:dim_LB}, \ref{lem:dim_IB}, and \ref{lem:dim_IC} below use this finer categorization of type A graphs to explicitly compute the relationship between sums in the formulas of $D_\Gamma$ and $D_{\Gamma-n}$. For the proofs of Lemmas \ref{lem:dim_IB}, \ref{lem:dim_IC}, and \ref{lem:dim_LB} below, $B_0$ is the block in $\Gamma$ that contains $n$ and (if $\Gamma$ type A2-A5) $B'$ is the block in $\Gamma$ that contains $n-1$ but not $n$ (so $B'$ is still a block in $\Gamma - n$).
\begin{lemma}\label{lem:dim_LB}
    Let $\Gamma$ be a cliqued and naturally labeled graph on at least $4$ vertices. Then the number of leaf blocks in $\Gamma - n$ is
    \[\abs{\LB_{\Gamma - n}} = \begin{cases}
        \abs{\LB_{\Gamma}} & \text{ if } \Gamma \text{ is type A2 or A4} \\
        \abs{\LB_{\Gamma}}-1 & \text{ if } \Gamma \text{ is type A1 or A3} \\
        \abs{\LB_{\Gamma}}-2 & \text{ if } \Gamma \text{ is type A5.}
    \end{cases}\]
\end{lemma}
\begin{proof}
    If $\Gamma$ is type A2 or A4, then $B' \in \LB_{\Gamma-n}$ is a leaf block of $\Gamma-n$, but $B' \in \IB_\Gamma$ is an internal block of $\Gamma$. So $\LB_{\Gamma-n} = \left(\LB_{\Gamma} \setminus\{B_0\}\right) \cup \{B'\}$. Less formally, we lose a block $B_0$ and gain a block $B'$, maintaining the same size. \par 
    If $\Gamma$ is type $A1$ or $A3$, every leaf block in $\Gamma - n$ is a leaf block in $\Gamma$, but we still lose $B_0$. So the size decrements by $1$. \par 
    If $\Gamma$ is type A5, then $\Gamma$ has two leaf blocks ($B_0 $ and $B'$), whereas $\Gamma - n$ is a clique, with zero leaf blocks. We directly compute $\abs{\LB_{\Gamma}} = 2$ and $\abs{\LB_{\Gamma-n}} = 0$.
\end{proof}
\begin{lemma}\label{lem:dim_IB}
    Let $\Gamma$ be a cliqued and naturally labeled graph on at least $4$ vertices. Then
    \[\abs{\{B \in \IB_{\Gamma-n} \mid \abs{B} > 2\}} = \begin{cases}
        \abs{\{B \in \IB_\Gamma \mid \abs{B} > 2\}} & \text{ if } \Gamma \text{ is type A1, A2, or A3} \\
        \abs{\{B \in \IB_\Gamma \mid \abs{B} > 2\}}-1 & \text{ if } \Gamma \text{ is type A4} \\
        \abs{\{B \in \IB_\Gamma \mid \abs{B} > 2\}}+1 & \text{ if } \Gamma \text{ is type A5.}
    \end{cases}\]
\end{lemma}
\begin{proof}
    If $\Gamma$ is type A1, then every internal block of $\Gamma$ is an internal block of $\Gamma'$ and vice versa. If $\Gamma$ is type A2, then $B'$ where $V(B') = \{n-2,n-1\}$ is an internal block for $\Gamma$ but a leaf block for $\Gamma - n$. However, $B'$ is not counted above since $\abs{B'} = 2$. Every other internal block of $\Gamma$ is an internal block of $\Gamma - n$ and vice versa. If $\Gamma$ is type A3, then $B'$ is still an internal block in $\Gamma - n$, because it is adjacent to more than $2$ cut vertices in the block-cut tree of $\Gamma - n$. Every other internal block is also the same, and so the equality follows. \par 
    If $\Gamma$ is type A4, then $B'$ is an internal block of $\Gamma$ but a leaf block in $\Gamma - n$. Every other internal block in $\Gamma - n$ is an internal block in $\Gamma$, so $\abs{\{B \in \IB_{\Gamma-n} \mid \abs{B} > 2\}} = \abs{\{B \in \IB_\Gamma \mid \abs{B} > 2\}}-1$. \par 
    If $\Gamma$ is type A5, then $\IB_\Gamma = \emptyset$. On the other hand, $\IB_{\Gamma - n} = \{B'\}$. Since $\Gamma$ has at least $4$ vertices, we know that $\abs{B'} > 2$, and the claim follows.
\end{proof}
\begin{lemma}\label{lem:dim_IC}
    Let $\Gamma$ be a cliqued and naturally labeled graph on at least $4$ vertices. Then the internal cut edges in $\Gamma-n$ are
    \[\IC_{\Gamma-n} = \begin{cases}
        \IC_\Gamma & \text{ if } \Gamma \text{ is type  A1, A3, A4, or A5} \\
        \IC_\Gamma \setminus \{(n-2,n-3)\} & \text{ if } \Gamma \text{ is type A2.}
    \end{cases}\]
\end{lemma}
\begin{proof}
    The set $\IC_\Gamma$ is the set of internal cut edges in $\Gamma$. The two types where $B'$ is an internal block in $\Gamma$ but a leaf block in $\Gamma - n$ are types A2, A3, and A4. In types A3 and A4 the block $B'$ is assumed to have size $\abs{B} > 2$, and so does not correspond to an element of $\IC_\Gamma$. If $\Gamma$ is type A2, $B'$ contributes to $\IC_\Gamma$ but not $\IC_{\Gamma - n}$, and that contribution is precisely the cut edge $(n-2,n-1)$.
\end{proof}
Now we are ready to prove the recursive formula for $D_\Gamma$.
\begin{proof}[Proof of Proposition \ref{prop:lin_dimension}]
Consider the sum $\sum_{j \vdash \Gamma - n} \cut_{\Gamma - n}(j)$. If $\Gamma$ is type B or C, then every cut vertex of $\Gamma$ is a cut vertex of $\Gamma - n$ and vice versa. For each such cut vertex $j \vdash \Gamma$ (and $j \vdash \Gamma - n$), since $\Gamma$ is type B/C no connected component of $\Gamma - j$ consists of only the vertex $n$, so $\cut_\Gamma(j) = \cut_{\Gamma - n}(j)$. \par 
If $\Gamma$ is type A1, then every cut vertex of $\Gamma$ is a cut vertex of $\Gamma - n$ but $\cut_\Gamma(n-1) = \cut_{\Gamma - n}(n-1) + 1$. If $\Gamma$ is type A2, A3, A4, or A5 then $n-1$ is not a cut vertex of $\Gamma - n$ but $\cut_\Gamma(n-1) = 1$. So we set $\cut_{\Gamma - n}(n-1) \coloneqq 0$ and the same relationship as in type A1 applies. 
Now for every other $j \neq n-1$, the vertex $j \vdash \Gamma$ if and only if $j \vdash \Gamma - n$, and then $\cut_\Gamma(j) = \cut_{\Gamma - n}(j)$. So we compute
    \begin{align*}
    \sum_{j \vdash \Gamma-n} \cut_{\Gamma-n}(j)  &= \sum_{j \vdash \Gamma} \cut_\Gamma(j) - \begin{cases}
            1 &\text{if } \Gamma \text{ type A}\\
            0 &\text{if } \Gamma \text{ type B/C.}
        \end{cases} 
    \end{align*} 
    Next, we compare the contributions to $D_\Gamma$ and $D_{\Gamma - n}$ by leaf blocks and internal blocks of size greater than $2$. By Lemmas \ref{lem:dim_typeBC}, \ref{lem:dim_IB} and \ref{lem:dim_IC}, it follows that 
    \[\abs{LB_{\Gamma-n}} +\abs{\{B \in \IB_{\Gamma-n} \mid \abs{B} > 2\}} = \abs{LB_\Gamma} + \abs{\{B \in \IB_{\Gamma} \mid \abs{B} > 2\}} - \begin{cases}
        1 & \text{ if $\Gamma$ type A but not A2,} \\
        0 & \text{ if $\Gamma$ type B, C, or A2.}
    \end{cases} \] 
    Since $\Gamma$ is naturally labeled, if $s \in \IC_{\Gamma-n}$ (so $s \neq (n,n-1)$), then the connected component of $([n-1],E(\Gamma-n) \setminus s)$ that does not contain the vertex $n-1$ is equal to the connected component of $([n],E(\Gamma)\setminus\{s\})$ that does not contain the vertex $n$. It is important for proving the recursion that for each internal cut edge $s \in \IC_{\Gamma - n} \subset \IC_\Gamma$ we always pick the component $G_s$ to be equal in $\Gamma$ and $\Gamma - n$ (i.e. always choose $n \notin G_s \subset \Gamma$), so we adopt this convention. If $\Gamma$ is type A2 (i.e. $\IC_\Gamma \neq \IC_{\Gamma - n}$) we remove the ``over-counting" from the internal cut edge $(n-2,n-1) \in \IC_\Gamma$ below, and get that 
    \[
    \sum_{s \in \IC_{\Gamma-n}} \choos{n-1}{\abs{G_s}} = \sum_{s \in \IC_{\Gamma}} \choos{n-1}{\abs{G_s}} - \begin{cases}
        \choos{n-1}{n-2} & \Gamma \text{ is type A2}\\
        0 &\text{otherwise.}
    \end{cases}
    \]
    Now we compute 
    \begin{align*}
        D_{\Gamma - n} &= 2n-3 - \sum_{j \vdash \Gamma-n} \cut_{\Gamma-n}(j) + (n-1)\left(\abs{\LB_{\Gamma-n}} + \abs{\{B \in \IB_{\Gamma-n} \mid \abs{B} > 2\}} - 1\right) \\
        &\hspace{1cm}+ \sum_{s \in \IC_{\Gamma-n}} \choos{n-1}{\abs{G_s}} \\
        &= 2n-3 - \sum_{j \vdash \Gamma}\cut_{\Gamma}(j) + \begin{cases}
            1 &\text{if } \Gamma \text{ type A}\\
            0 &\text{if } \Gamma \text{ type B/C.}
        \end{cases} \\
        &\hspace{1cm} + (n-1)\left(\abs{LB_\Gamma} + \abs{\{B \in \IB_{\Gamma} \mid \abs{B} > 2\}} - 1\right) -\begin{cases}
        n-1 & \text{ if $\Gamma$ type A but not A2} \\
        0 & \text{ if $\Gamma$ type B, C, or A2.}
    \end{cases} \\
    &\hspace{1cm} + \sum_{s \in \IC_{\Gamma}} \choos{n-1}{\abs{G_s}} - \begin{cases}
        \choos{n-1}{n-2} & \Gamma \text{ is type A2}\\
        0 &\text{otherwise.}
    \end{cases}\\
    &= 2n-2 - \sum_{j \vdash \Gamma}\cut_{\Gamma}(j) - \begin{cases}
            n-1 &\text{if } \Gamma \text{ type A}\\
            1 &\text{if } \Gamma \text{ type B/C.}
        \end{cases} \\
        &\hspace{1cm} + (n-1)\left(\abs{LB_\Gamma} + \abs{\{B \in \IB_{\Gamma} \mid \abs{B} > 2\}} - 1\right) + \sum_{s \in \IC_{\Gamma}} \choos{n-1}{\abs{G_s}} 
    \end{align*}
    Now we have $D_{\Gamma-n}$ in a form that closely resembles that of $D_\Gamma$. Recall Pascal's identity $\choos{n}{k} = \choos{n-1}{k}+\choos{n-1}{k-1}$. For each $s \in \IC_\Gamma$ we have that 
    \[\choos{n}{\abs{G_s}} - \choos{n-1}{\abs{G_s}} = \choos{n-1}{\abs{G_s}-1}.\]
    In particular, 
    \[
    \sum_{s \in \IC_\Gamma} \choos{n}{\abs{G_s}} - \sum_{s \in \IC_\Gamma}\choos{n-1}{\abs{G_s}} = \sum_{s \in \IC_\Gamma}\choos{n-1}{\abs{G_s}-1}.
    \]
    So we compute that the difference $D_\Gamma - D_{\Gamma - n}$ is equal to
    \begin{align*}
        1 + \left(\abs{\LB_\Gamma} + \abs{\{B \in \IB_{\Gamma} \mid \abs{B} > 2\}} - 1\right) + \sum_{s \in \IC_\Gamma}\choos{n-1}{\abs{G_s}-1} + \begin{cases}
        n-1 & \text{ if $\Gamma$ type A} \\
        1 & \text{ if $\Gamma$ type B/C.}
    \end{cases}
    \end{align*}
    The remainder of the proof is essentially the same as the proof for Lemma \ref{lem:linear_dimension_natural}. In particular,
    \begin{enumerate}
        \item Leaf blocks of size $2$ that do not contain the vertex $n$ are in bijection with strongly dominant pairs $(i,j)$ where $\abs{\Gamma_i^j} = 1$,
        \item The set of internal cut edges $\IC_\Gamma$ is equal to the set of strongly dominant pairs $(i,j)$ where $\abs{\Gamma_i^j} >1$, and 
        \item Leaf blocks and internal blocks of size at least three (i.e. all blocks of size at least $3$) that do \emph{not} contain the vertex $n$ are in bijection with weakly dominant pairs $(i,j)$. 
    \end{enumerate} 
     We note that $\choos{n-1}{1-1} = 1$, and get that 
    \[ \left(\abs{\LB_\Gamma} + \abs{\{B \in \IB_{\Gamma} \mid \abs{B} > 2\}} - 1\right) + \sum_{s \in \IC_\Gamma}\choos{n-1}{\abs{G_s}-1}  =  \sum_{(i,j) \gg \Gamma} \choos{n-1}{\abs{\Gamma_i^j}-1} + \abs{\left\{(i,j) >  \Gamma \right\}}. \]
    Thus $D_\Gamma - D_{\Gamma - n}$ has the claimed form.
\end{proof}

\begin{corollary}
    Let $\Gamma$ be a connected graph on $[n]$ where $n \geq 3$, and $D_\Gamma$ as defined in Equation (\ref{eqn:linear_dimension}). Then $\dim_\C(\splines{\Gamma}^1) = D_\Gamma$. Moreover, $\dim_\C(\mathrm{L}_\Gamma)_1 = \dim_\C(\mathrm{R}_\Gamma)_1 = D_\Gamma - n$.
\end{corollary}
\begin{proof}
    It suffices to assume that $\Gamma$ is cliqued and naturally labeled. In this case, by Theorem \ref{thm:linear_spanning} and Proposition \ref{prop:lin_dimension} both $\dim_\C(\splines{\Gamma}^1)$ and $D_\Gamma$ follow the same recursion. Thus it suffices to show equality on all connected graphs on $3$ vertices. Quick computation confirms that 
    \[\dim_\C(\splines{K_3}^1) = 5 = D_{K_3}\text{ and } \dim_\C(\splines{P_3}^1) = 7 = D_{P_3},\]
    and so the two statistics must be equal. The ``moreover" part follows from the fact that $\C\{\bar{t}_1,...,\bar{t}_n\}$ and $\C\{\bar{x}_1,...,\bar{x}_n\}$ are the ($n$-dimensional) linear subspace of $\splines{\Gamma}^1$ quotiented to obtain $\mathrm{L}_\Gamma$ and $\mathrm{R}_\Gamma$ respectively.
\end{proof}
We now have a closed combinatorial formula for the $\C$-dimension of the first graded piece of $\splines{\Gamma}^1$. In particular, we have the dimension of the representations corresponding to $\lrep{\Gamma}_1$ and $\rrep{\Gamma}_1$. We are also quite close to constructing $\C$-bases of $\splines{\Gamma}^1$, as the formulae for $D_\Gamma$ in Lemma \ref{lem:linear_dimension_natural} and the natural one for $\abs{\cb_\Gamma}$ are very similar.
\section{The Left and Right Linear Representations}\label{sec:lin_reps}
This subsection computes $\lrep{\Gamma}_1$ and $\rrep{\Gamma}_1$ for all connected $\Gamma$, proving Theorem \ref{intthm:linear} from the introduction. We prove this assuming that $\Gamma$ is naturally labeled, and get the label-independent formula as a corollary. The computation is direct, and achieved by computing the dot action on two subsets $\mathcal{LS}_\Gamma$ and $\mathcal{RS}_\Gamma$ of $\cb_\Gamma$ that project to bases of $(\lqot{\Gamma})_1$ and $(\rqot{\Gamma})_1$, respectively. \par 
Say $\Gamma$ is a naturally labeled connected graph. By Proposition \ref{prop:hard_linear_relations} and Theorem \ref{thm:linear_spanning}, if 
\[
\cb_\Gamma \coloneqq \left\{\bar{t}_i \mid i \in [n]\right\} \cup \left\{\bar{x}_i \mid i \in [n]\right\} \cup \left\{\bar{f}_A^s \left| \begin{matrix} s=(i,j) \gg \Gamma,\\ \abs{A} = \abs{\Gamma_i^j}\end{matrix}\right.\right\} \cup \left\{\bar{y}_{i,k}^j \left| \begin{matrix}(i,j) > \Gamma,\\ k \in [n] \end{matrix}\right.\right\},
\]
then $\splines{\Gamma}^1 = \C\cb_{\Gamma}$. For the first graded piece of $\lqot{\Gamma}$ and $\rqot{\Gamma}$, we will remove elements from $\cb_\Gamma$ using the relations in Lemma \ref{lem:linear_relations}, prove that the image of what remains is a basis by dimension, then compute the representations on $(\lqot{\Gamma})_1$ and $(\rqot{\Gamma})_1$ using Lemma \ref{lem:linear_dotact_elts} below.
\begin{lemma}\label{lem:linear_dotact_elts}
    Let $w \in S_n$. Then 
    \[w \cdot \bar{f}_A^{(i,j)} = \bar{f}_{w(A)}^{(i,j)} \text{  and  } w \cdot \bar{y}_{r,k}^j = \bar{y}_{r,w(k)}^j\]
\end{lemma}
The proof of Lemma \ref{lem:linear_dotact_elts} is direct from the definitions. \par 
Now we will define two subsets of $\cb_\Gamma$, one for $(\lqot{\Gamma})_1$ and one for $(\rqot{\Gamma})_1$, that project to $\C$-bases in these quotients. \par 
For the linear piece of the left quotient $(\lqot{\Gamma})_1$, we first remove from $\cb_\Gamma$ the splines $\{\bar{t}_1,...,\bar{t}_n\}$. We may also discard 
\begin{itemize}
    \item[(1)] The single spline $\bar{x}_n$ by Lemma \ref{lem:linear_relations}(1),
    \item[(2)] The splines $\{\bar{x}_i \mid (i,j) \gg \Gamma\}$ by Lemma \ref{lem:linear_relations}(2), and
    \item[(3)] The splines $\{\bar{x}_i \mid (i,j) > \Gamma\}$ by Lemma \ref{lem:linear_relations}(3).
\end{itemize}
Note the set of splines in (2) and (3) is size $\abs{\{\bar{x}_i \mid (i,j) \gg \Gamma\} \cup \{\bar{x}_i \mid (i,j) > \Gamma\}} = \sum_{j \vdash \Gamma} \cut(j)$. So the image of 
\[ \mathcal{LS}_\Gamma \coloneqq \left\{\bar{x}_r \in \cx_{n-1} \mid r \text{ is not $s$-dominant }\forall s\in [n]\right\} \cup \left\{\bar{f}_A^s \mid s \gg \Gamma\right\} \cup \left\{\bar{y}_{r,k}^j \left| (r,j) > \Gamma,\;\; k \in [n]\right.\right\}\]
in $(\lqot{\Gamma})_1$ is a spanning set. Note that the size of $\left\{\bar{x}_r \in \cx_{n-1} \mid r \text{ is not $s$-dominant }\forall s\in [n]\right\}$ is $n-1-\sum_{j\vdash \Gamma} \cut(j)$, the size of $\left\{\bar{f}_A^s \mid s \gg \Gamma\right\}$ is ${\ds \sum_{(i,j) \gg \Gamma} \choos{n}{\abs{\Gamma_i^j}}}$, and the size of $\left\{\bar{y}_{r,k}^j \left| (r,j) > \Gamma,\;\; k \in [n]\right.\right\}$ is ${\ds \sum_{(i,j) > \Gamma}n}$. Thus the size of $\mathcal{LS}_\Gamma$ is precisely the dimension $D_\Gamma - n$ of $(\lqot{\Gamma})_1$, as computed in Lemma \ref{lem:linear_dimension_natural}. So $\mathcal{LS}_\Gamma$ projects to a basis of $(\lqot{\Gamma})_1$. In fact, $\mathcal{LS}_\Gamma$ is a permutation basis for the dot action representation, from which it is easy to compute the dot action representation (we will state and prove this in Theorem \ref{thm:linear_reps}).\par 
For the linear piece of the right quotient $(\rqot{\Gamma})_1$, we first remove from $\cb_\Gamma$ the splines  $\{\bar{x}_1,...,\bar{x}_n\}$. Let $m_{ij} \coloneqq \abs{\Gamma_i^j}$, and let $\left\{A_p \mid p \in \left[\choos{n}{m_{ij}}\right]\right\}$ be an enumeration of the $\choos{n}{m_{ij}}$-many subsets $A$ associated to a strongly dominant pair $(i,j) \gg \Gamma$. By Lemma \ref{lem:linear_relations}, the following three relations hold in $\rqot{\Gamma}$:
\[ \sum_{r=1}^n \bar{t}_r \sim 0,\;\; \sum_{A \subset [n]} \bar{f}_A^{(i,j)} \sim 0, \;  \text{and}\;\; \sum_{k=1}^n \bar{y}_{i,k}^j \sim 0.\]
The natural subset of $\cb_\Gamma$ whose image spans $(\rqot{\Gamma})_1$ is therefore  
\[\mathcal{RS}_\Gamma \coloneqq \{\bar{t}_r-\bar{t}_{r+1} \mid r \in [n-1]\}\cup \left\{ \bar{f}_{A_p}^{(i,j)} - \bar{f}_{A_{p+1}}^{(i,j)} \left| \begin{matrix} p \in \left[\choos{n}{m_{ij}}-1\right], \\ (i,j) \gg \Gamma \end{matrix} \right.\right\} \cup \left\{ 
\bar{y}_{r,k}^{j} - \bar{y}_{r,k+1}^j \left| \begin{matrix} (r,j) > \Gamma,\\ k\in [n-1] \end{matrix}   \right.\right\}\]
The first subset is size $n-1$. The second subset is size ${\ds  \sum_{(i,j)\gg \Gamma} \left(\choos{n}{\abs{\Gamma_i^j}} - 1 \right)}$ and the third subset is size ${\ds  \sum_{(i,j)>\Gamma} (n-1) }$. The number of (strong or weak) dominant pairs is \[\abs{(i,j) \in E(\Gamma) \mid (i,j) \gg \Gamma \text{ or } (i,j) > \Gamma\}} = \sum_{j \vdash \Gamma} \cut(j),\] so
\[
 \sum_{(i,j)\gg \Gamma} \left(\choos{n}{\abs{\Gamma_i^j}} - 1 \right)+ \sum_{(i,j)>\Gamma} (n-1) = \sum_{(i,j)\gg \Gamma} \choos{n}{\abs{\Gamma_i^j}} + \sum_{(i,j)>\Gamma} n -\sum_{j \vdash \Gamma}\cut(j).
\]
So the image of $\mathcal{RS}_\Gamma$ is a basis for $(\rqot{\Gamma})_1$.
\begin{theorem}\label{thm:linear_reps}
    Let $\Gamma$ be a naturally labeled graph. If $(i,j) \gg \Gamma$, define the partition $\lambda_{ij} \coloneqq \left(n-\abs{\Gamma_i^j},\abs{\Gamma_i^j}\right)$ (re-ordered if necessary). Then 
    \[
    \lrep{\Gamma}_1 = \sum_{(i,j) \gg \Gamma} h_{\lambda_{ij}} + \sum_{(i,j) > \Gamma} h_{n-1,1} + \left(n-1-\sum_{j \vdash \Gamma} \cut(j)\right) h_n,
    \]
    and 
    \[
    \rrep{\Gamma}_1 = s_{n-1,1} + \sum_{(i,j) \gg \Gamma} \left(h_{\lambda_{ij}}-s_n\right) + \sum_{(i,j) > \Gamma} s_{n-1,1}.
    \]
\end{theorem}
\begin{proof}
    Since $\C\mathcal{LS}_\Gamma$ and $\C\mathcal{RS}_\Gamma$ are $S_n$-invariant vector spaces, the dot action on each is a representation. Since the projection of these spaces to $(\lqot{\Gamma})_1$ and $(\rqot{\Gamma})_1$ are in fact isomorphisms, the symmetric functions $\lrep{\Gamma}_1$ and $\rrep{\Gamma}_1$ are the characters of the dot action representation on $\C\mathcal{LS}_\Gamma$ and $\C\mathcal{RS}_\Gamma$ respectively. Each of the identified subsets in bases $\mathcal{LS}_\Gamma$ and $\mathcal{RS}_\Gamma$ span $S_n$-invariant subspaces of $\C\mathcal{LS}_\Gamma$ and $\C\mathcal{RS}_\Gamma$ respectively. \par 
    First we will compute each part of $\lrep{\Gamma}_1$. The dot action fixes each $\bar{x}_i$, and so the character of the dot action representation on $\C\left\{\bar{x}_r \in \cx_{n-1} \mid r \text{ is not $s$-dominant }\forall s\in [n]\right\}$ is ${\ds \left(n-1-\sum_{j \vdash \Gamma} \cut(j)\right) h_n}$. By Lemma \ref{lem:linear_dotact_elts} the character of the dot action representation restricted to $\C\left\{\bar{f}_A^s \mid s \gg \Gamma\right\}$ is ${\ds \sum_{(i,j) \gg \Gamma} h_{\lambda_{ij}}}$. By Lemma \ref{lem:linear_dotact_elts} as well the character of the dot action representation on $\C\left\{\bar{y}_{r,k}^j \left| (r,j) > \Gamma,\;\; k \in [n]\right.\right\}$ is ${\ds \sum_{(i,j) > \Gamma} h_{n-1,1}}$. \par 
    Now we'll compute each part of $\rrep{\Gamma}_1$. Each of the following computations use the same principle argument. If $K$ is an integer, and the set $\{e_i \mid i \in [K]\}$ is a permutation basis of some permutation representation of $S_n$ with character $h_\lambda$, then the vector ${\ds \sum_{i=1}^K e_i}$ is invariant under that representation. Furthermore, the character of the representation on the orthogonal subspace spanned by $\{e_{i+1} - e_{i} \mid i \in [K-1]\}$ is $h_\lambda - s_n$. \par 
    The character of the dot action representation on $\C\{t_i \mid i \in [n]\}$ is $h_{n-1,n}$, and so the dot action representation on $\C\{\bar{t}_r-\bar{t}_{r+1} \mid r \in [n-1]\}$ is $h_{n-1,n}-s_n = s_{n-1,1}$. The character of the dot action representation on $\C\left\{ \bar{f}_{A_p}^{(i,j)}\left|  p \in \left[\choos{n}{m_i}\right], \; (i,j) \gg \Gamma \right.\right\}$ is ${\ds \sum_{(i,j) \gg \Gamma} h_{\lambda_{ij}}}$, and so the character of the dot action representation on 
    $\C\left\{ \bar{f}_{A_p}^{(i,j)} - \bar{f}_{A_{p+1}}^{(i,j)} \left|  p \in \left[\choos{n}{m_i}-1\right], \; (i,j) \gg \Gamma \right.\right\}$ is ${\ds \sum_{(i,j) \gg \Gamma} h_{\lambda_{ij}} - s_n}$. The character of the dot action representation on $\C\left\{ 
\left.\bar{y}_{r,k}^{j}\right| (r,j) > \Gamma,\; k\in [n]   \right\}$ is ${\ds \sum_{(i,j) > \Gamma} h_{n-1,1}}$, and so the character of the dot action representation on $\C \left\{\left. 
\bar{y}_{r,k}^{j} - \bar{y}_{r,k+1}^j \right| (r,j) > \Gamma,\; k\in [n-1]   \right\}$ is ${\ds \sum_{(i,j) > \Gamma}(h_{n-1,1}-s_n) = \sum_{(i,j) > \Gamma}s_{n-1,1}}$. 
\end{proof}
 The Schur-expansion of $h_{\lambda_{ij}}$ and is easy to compute since $\lambda_{ij}$ is only a two-part partition. In particular, if $K$ is the larger of $\abs{\Gamma_i^j}$ and $n-\abs{\Gamma_i^j}$, then ${\ds h_{\lambda_{ij}}-s_n = \sum_{m=0}^{n-K-1} s_{K+m,n-K-m}}$. \par 
The following corollary gives the label-independent description, from the statistics on block-cut trees described in Section \ref{sec:lin_dim}.
\begin{corollary}\label{cor:nonlabel_reps}
    Let $\Gamma$ be a connected simple graph, and let $\LB_\Gamma$, $\IB_\Gamma$, and $\IC_\Gamma$ be the leaf blocks, internal blocks, and internal cut edges of $\Gamma$ as defined in the beginning of Section \ref{sec:lin_dim}. For $(i,j) \in \IC_\Gamma$, let $\lambda_{ij}$ be the partition $\left(n-\abs{G_{(i,j)}},\abs{G_{(i,j)}}\right)$ (re-ordered if necessary), where $G_{(i,j)}$ is a connected component of the graph $\left([n],E(\Gamma)\setminus\{(i,j)\}\right)$. Then 
    \[
    \lrep{\Gamma}_1 = \sum_{(i,j) \in \IC_\Gamma} h_{\lambda_{ij}} + \left( \abs{\LB_\Gamma} +  \abs{\{B \in \IB_\Gamma \mid \abs{B} > 2\}} - 1\right) h_{n-1,1} + \left(n-1-\sum_{j \vdash \Gamma} \cut(j)\right) h_n,
    \]
    and 
    \[
    \rrep{\Gamma}_1 =\sum_{(i,j) \in \IC_\Gamma} \left(h_{\lambda_{ij}}-s_n\right) + \left( \abs{\LB_\Gamma} +  \abs{\{B \in \IB_\Gamma \mid \abs{B} > 2\}}\right)s_{n-1,1}.
    \]
\end{corollary}
\begin{proof}
    This follows directly from Theorem \ref{thm:linear_reps} and the bijections/equalities described in the proof of Lemma \ref{lem:linear_dimension_natural} (and also the proof of Proposition \ref{prop:lin_dimension}).
\end{proof}
We note that in the statement of Theorem \ref{intthm:linear} in the introduction, the sets are $E_1 = \IC_\Gamma$ and $E_2 = \LB_\Gamma \cup \{B \in \IB_\Gamma \mid \abs{B} > 2\}$, and the integer $k = n-1-\sum_{j \vdash \Gamma} \cut(j)$.
    \begin{example}
        Let $\Gamma$ be the graph from Example \ref{ex:blockpair_bijection}. We may compute the representations with either Theorem \ref{thm:linear_reps} or Corollary \ref{cor:nonlabel_reps}. Then \[ \lrep{\Gamma}_1 = 6h_{12} + 4h_{11,1} + h_{8,4}\]
        and \[ \rrep{\Gamma}_1 = 4s_{11,1} + h_{8,4}-s_{12} = 4s_{11,1} + (s_{8,4} + s_{9,3} + s_{10,2} + s_{11,1}).\]
    \end{example}
We note that, by this formula, the symmetric function $\lrep{\Gamma}_1$ is $h$-positive for all graphs $\Gamma$. So Theorem \ref{thm:linear_reps} and Corollary \ref{cor:nonlabel_reps} prove an extension of linear part of the graded Stanley-Stembridge conjecture from Hessenberg graphs to all connected graphs.

\bibliographystyle{alpha}
\bibliography{splines}

\begin{thebibliography}{DMPS92}

\bibitem[AB21]{Ayzenberg2018isospectral}
Anton Ayzenberg and Victor Buchstaber.
\newblock Manifolds of isospectral matrices and {H}essenberg varieties.
\newblock {\em Int. Math. Res. Not. IMRN}, 2021(21):16671--16692, 2021.

\bibitem[Ale21]{Alexandersson2020Lollipop}
Per Alexandersson.
\newblock L{LT} polynomials, elementary symmetric functions and melting
  lollipops.
\newblock {\em J. Algebraic Combin.}, 53(2):299--325, 2021.

\bibitem[AMS22]{ayzenberg2022second}
A.~A. Ayzenberg, M.~Masuda, and T.~Sato.
\newblock The second cohomology of regular semisimple {H}essenberg varieties
  from {GKM} theory.
\newblock {\em Tr. Mat. Inst. Steklova}, 317:5--26, 2022.

\bibitem[AN21]{Abreu_Nigro20}
Alex Abreu and Antonio Nigro.
\newblock Chromatic symmetric functions from the modular law.
\newblock {\em J. Combin. Theory Ser. A}, 180:Paper No. 105407, 30, 2021.

\bibitem[AP18]{ALEXANDERSSON2018LLTchromsym}
Per Alexandersson and Greta Panova.
\newblock L{LT} polynomials, chromatic quasisymmetric functions and graphs with
  cycles.
\newblock {\em Discrete Math.}, 341(12):3453--3482, 2018.

\bibitem[BC18]{brosnan_chow_dotactn_is_chromsym}
Patrick Brosnan and Timothy~Y. Chow.
\newblock Unit interval orders and the dot action on the cohomology of regular
  semisimple {H}essenberg varieties.
\newblock {\em Adv. Math.}, 329:955--1001, 2018.

\bibitem[Bil99]{billey1999schub_formula}
Sara~C. Billey.
\newblock Kostant polynomials and the cohomology ring for {$G/B$}.
\newblock {\em Duke Math. J.}, 96(1):205--224, 1999.

\bibitem[Bla16]{Blasiak2016LLTMacdonald}
Jonah Blasiak.
\newblock Haglund's conjecture on 3-column {M}acdonald polynomials.
\newblock {\em Math. Z.}, 283(1-2):601--628, 2016.

\bibitem[CHL22]{chohonglee_second_cohom}
Soojin Cho, Jaehyun Hong, and Eunjeong Lee.
\newblock Permutation module decomposition of the second cohomology of a
  regular semisimple {H}essenberg variety.
\newblock {\em Int. Math. Res. Not.}, 12 2022.

\bibitem[CHL23]{ChoHongLee_Bases}
Soojin Cho, Jaehyun Hong, and Eunjeong Lee.
\newblock Bases of the equivariant cohomologies of regular semisimple
  {H}essenberg varieties.
\newblock {\em Adv. Math.}, 423:Paper No. 109018, 81, 2023.

\bibitem[Cho]{chow_linearp}
Timothy Chow.
\newblock e-positivity of the coefficient of t in $x_g(t)$.

\bibitem[Dah19]{Dahlberg19}
Samantha Dahlberg.
\newblock Triangular ladders $p_{d,2}$ are $e$-positive, 2019.

\bibitem[DMPS92]{DMPS1992hessenbergvarieties}
F.~De~Mari, C.~Procesi, and M.~A. Shayman.
\newblock Hessenberg varieties.
\newblock {\em Trans. Amer. Math. Soc.}, 332(2):529--534, 1992.

\bibitem[Gas96]{Gasharov96}
Vesselin Gasharov.
\newblock Incomparability graphs of {$(3+1)$}-free posets are {$s$}-positive.
\newblock {\em Discrete Math.}, 157(1-3):193--197, 1996.

\bibitem[GKM98]{GKM_theory}
Mark Goresky, Robert Kottwitz, and Robert MacPherson.
\newblock Equivariant cohomology, {K}oszul duality, and the localization
  theorem.
\newblock {\em Invent. Math.}, 131(1):25--83, 1998.

\bibitem[GP13]{GuayPaquet13}
Mathieu Guay-Paquet.
\newblock A modular relation for the chromatic symmetric functions of
  (3+1)-free posets, 2013.

\bibitem[GP16]{guaypaquet2016shar_wachs_conj}
Mathieu Guay-Paquet.
\newblock A second proof of the shareshian--wachs conjecture, by way of a new
  hopf algebra, 2016.

\bibitem[GTV16]{tymoczko2016generalized_splines}
Simcha Gilbert, Julianna Tymoczko, and Shira Viel.
\newblock Generalized splines on arbitrary graphs.
\newblock {\em Pacific J. Math.}, 281(2):333--364, 2016.

\bibitem[HNY20]{Huh2020LLTLollipop}
JiSun Huh, Sun-Young Nam, and Meesue Yoo.
\newblock Melting lollipop chromatic quasisymmetric functions and {S}chur
  expansion of unicellular {LLT} polynomials.
\newblock {\em Discrete Math.}, 343(3):111728, 21, 2020.

\bibitem[HP19]{harada2017cohomology}
Megumi Harada and Martha~E. Precup.
\newblock The cohomology of abelian {H}essenberg varieties and the
  {S}tanley{\textendash}{S}tembridge conjecture.
\newblock {\em Algebr. Comb.}, 2(6):1059--1108, 2019.

\bibitem[Lee21]{Lee2021LLTlinear}
Seung~Jin Lee.
\newblock Linear relations on {LLT} polynomials and their k-{S}chur positivity
  for {$k=2$}.
\newblock {\em J. Algebraic Combin.}, 53(4):973--990, 2021.

\bibitem[LT00]{Leclerc2000LLT_LR_KL}
Bernard Leclerc and Jean-Yves Thibon.
\newblock Littlewood-{R}ichardson coefficients and {K}azhdan-{L}usztig
  polynomials.
\newblock In {\em Combinatorial methods in representation theory ({K}yoto,
  1998)}, volume~28 of {\em Adv. Stud. Pure Math.}, pages 155--220. Kinokuniya,
  Tokyo, 2000.

\bibitem[SS93]{SS1993immanantsofJTmatrices}
Richard~P. Stanley and John~R. Stembridge.
\newblock On immanants of {J}acobi-{T}rudi matrices and permutations with
  restricted position.
\newblock {\em J. Combin. Theory Ser. A}, 62(2):261--279, 1993.

\bibitem[Sta95]{STANLEY1995chromsym}
Richard~P. Stanley.
\newblock A symmetric function generalization of the chromatic polynomial of a
  graph.
\newblock {\em Adv. Math.}, 111(1):166--194, 1995.

\bibitem[SW16]{SW2016chromaticquasisymmetric}
John Shareshian and Michelle~L. Wachs.
\newblock Chromatic quasisymmetric functions.
\newblock {\em Adv. Math.}, 295:497--551, 2016.

\bibitem[Tym08a]{tymoczko2008permutation}
Julianna~S. Tymoczko.
\newblock Permutation actions on equivariant cohomology of flag varieties.
\newblock In {\em Toric topology}, volume 460 of {\em Contemp. Math.}, pages
  365--384. Amer. Math. Soc., Providence, RI, 2008.

\bibitem[Tym08b]{tymoczko2008schubertreps}
Julianna~S. Tymoczko.
\newblock Permutation representations on {S}chubert varieties.
\newblock {\em Amer. J. Math.}, 130(5):1171--1194, 2008.

\end{thebibliography}

\end{document}